\documentclass[11pt,a4paper]{amsart}

\usepackage{mathtools}
\usepackage{graphicx}
\usepackage{amsthm}
\usepackage{amssymb}
\usepackage{caption}
\usepackage{hyperref}
\usepackage{nicematrix}
\usepackage{enumitem}
\usepackage{booktabs}
\usepackage[dvipsnames]{xcolor}
\usepackage[left=30mm, right=30mm]{geometry}

\usepackage{tikz}
\usepackage{tikz-cd}
\usetikzlibrary{calc}
\usetikzlibrary{decorations.pathmorphing}
\usetikzlibrary{decorations.pathreplacing}

\theoremstyle{definition}
\newtheorem{definition}{Definition}[section]

\newtheorem{conjecture}[definition]{Conjecture}
\newtheorem{example}[definition]{Example}

\newtheorem{problem}[definition]{Problem}

\theoremstyle{plain}
\newtheorem{corollary}[definition]{Corollary}
\newtheorem{lemma}[definition]{Lemma}
\newtheorem{proposition}[definition]{Proposition}
\newtheorem{theorem}[definition]{Theorem}

\numberwithin{equation}{section}

\tikzset{
	vertex/.style={
		circle,
		minimum size=1.8mm,
		fill,
		inner sep=0,
		outer sep=0,
	},
	edge/.style={
		line width=.25mm,
	}
}

\tikzset{
	x=10mm,
	y=10mm,
	treenode/.style={
		circle,
		minimum size=1mm,
		inner sep=0,
		fill=gray,
	},
	edgelabel/.style={
		node font=\small,
		inner sep =.5mm,
	},
	columnlabel/.style={
		anchor=north,
	}
}

\ExplSyntaxOn

\cs_new_protected:Npn \cycle #1
{
	\seq_set_split:Nnn\l_tmpa_seq{~}{#1}
	(
	\seq_use:Nn\l_tmpa_seq{\,} 
	)
}

\cs_new_protected:Npn \chain #1
{
	\seq_set_split:Nnn\l_tmpa_seq{~}{#1}
	(
	\seq_use:Nn\l_tmpa_seq{\,} 
	\rangle
}

\ExplSyntaxOff

\newcommand*\defterm{\emph}

\DeclareMathOperator{\dom}{Dom}
\DeclareMathOperator{\im}{Im}

\newcommand*\girth[1]{\mathrm{girth}(#1)}
\newcommand*\knitdegree[1]{\mathrm{kd}(#1)}

\newcommand*\cliquenumber[2][]{\omega\parens[#1]{#2}}

\newcommand*\commgraph[1]{\mathcal{G}(#1)}

\newcommand*\centre[1]{Z(#1)}
\newcommand*\idemp[1]{E(#1)}

\DeclarePairedDelimiter{\abs}{\lvert}{\rvert}

\DeclarePairedDelimiter{\parens}{\lparen}{\rparen}
\DeclarePairedDelimiter{\bracks}{\lbrack}{\rbrack}
\DeclarePairedDelimiter{\braces}{\{}{\}}

\DeclarePairedDelimiter{\set}{\{}{\}}
\DeclarePairedDelimiterX{\gset}[2]{\{}{\}}{\,#1:#2\,}
\newcommand\gsetsplit[3][]{\mathopen#1\{\,#2:#3\,\mathclose#1\}}

\newcommand*\Xn{\set{1,\ldots,n}}
\newcommand*\X[1]{\set{1,\ldots,#1}}

\newcommand*{\ptr}[1]{\mathcal{P}(#1)}
\newcommand*{\tr}[1]{\mathcal{T}(#1)}
\newcommand*{\Tr}[1]{\mathcal{T}_{#1}}
\newcommand*{\psym}[1]{\mathcal{I}(#1)}

\newcommand*{\sym}[1]{\mathcal{S}(#1)}

\newcommand*\id[1]{\mathrm{id}_{#1}}

\newcommand*\setclass[2]{\mathcal{C}(#1,#2)}
\newcommand*\maxnull[1]{(#1)\xi}
\newcommand*\alphanull[1]{(#1)\alpha}
\newcommand*\commidemp[2]{\Gamma^{#1}_{#2}}
\newcommand*\nulltr[3]{N^{#1}_{#2,\ldots,#3}}

\newcommand*\new{\infty}
\newcommand*\newtr[1]{#1_{\new}}
\newcommand*{\nullptr}[2]{\Omega^{#1}_{#2}}

\makeatletter
\newcommand*{\sizeddelimiter}[2]{\bBigg@{#1}#2}
\makeatother

\allowdisplaybreaks

\begin{document}

\title[Characterizing the largest commutative semigroups of certain types]{Characterizing the largest commutative (full and partial) transformation semigroups of certain types}

\author{Tânia Paulista}
\address[T. Paulista]{%
Center for Mathematics and Applications (NOVA Math) \& Department of Mathematics\\
NOVA School of Science and Technology\\
NOVA University of Lisbon\\
2829--516 Caparica\\
Portugal
}
\email{%
tpl.paulista@gmail.com
}
\thanks{This work is funded by national funds through the FCT -- Fundação para a Ciência e a Tecnologia, I.P., under the scope of the projects UID/297/2025 and UID/PRR/297/2025 (Center for Mathematics and Applications - NOVA Math). The author is also funded by national funds through the FCT -- Fundação para a Ciência e a Tecnologia, I.P., under the scope of the studentship 2021.07002.BD}

\thanks{The author is thankful to her supervisors António Malheiro and Alan J. Cain for all the support, encouragement and fruitful discussions; and also for reading a draft of this paper}


\begin{abstract}
	Let $X$ be a finite set. Let $\tr{X}$ be the transformation semigroup on $X$ and let $\ptr{X}$ be the partial transformation semigroup on $X$. This paper is a contribution to the problem of characterizing the largest commutative subsemigroups of $\tr{X}$ (respectively, $\ptr{X}$). In the process of looking for these semigroups, we also characterize the largest commutative subsemigroups of idempotents of $\tr{X}$ (respectively, $\ptr{X}$); as well as the largest commutative subsemigroups of $\tr{X}$ (respectively, $\ptr{X}$) that contain a unique idempotent. We also provide an alternative way to determine the largest commutative nilpotent subsemigroups of $\tr{X}$ (which were previously characterized by Cain, Malheiro and the present author); and we describe the largest commutative nilpotent subsemigroups of $\ptr{X}$.
	
	These results allow us to make conclusions regarding the clique numbers of the commuting graphs of $\tr{X}$ and of $\ptr{X}$. We also determine their girths and knit degrees.
	

\end{abstract}

\maketitle

\section{Introduction}

This paper revolves around the problem of determining, for a finite set $X$, the maximum size of a commutative subsemigroup of the transformation semigroup $\tr{X}$ on $X$ (respectively, partial transformation semigroup $\ptr{X}$ on $X$), as well as the maximum-order commutative subsemigroups of $\tr{X}$ (respectively, $\ptr{X}$). This problem motivates the characterization of the maximum-order commutative subsemigroups of idempotents of $\tr{X}$ (respectively, $\ptr{X}$) and the maximum-order commutative subsemigroups of $\tr{X}$ (respectively, $\ptr{X}$) that contain exactly one idempotent.

The problem of describing, for a given group/semigroup, the maximum-order subgroups/subsemigroups that satisfy certain properties has been studied by several authors, for various subgroups/semigroups, and with variations in the selected properties. For instance, in 1989, Burns and Goldsmith \cite{Symmetric_group} characterized the maximum-order abelian subgroups of the symmetric group and, in 1999, Vdovin \cite{Alternating_group} characterized the maximum-order abelian subgroups of the alternating group.

There has also been considerate work with the transformation semigroup $\tr{X}$, for a finite set $X$. In 1976, Biggs, Rankin and Reis \cite{Non_commutative_nilpotent_semigroups} showed that the maximum size of a nil subsemigroup of $\tr{X}$ is $\parens{\abs{X}-1}!$. In 2004, Holzer and König studied maximum-order subsemigroups of $\tr{X}$ that are $2$-generated: they proved that, when $\abs{X}\geqslant 7$ is prime, a maximum-order $2$-generated subsemigroup of $\tr{X}$ could be found in a particular `nice' class of subsemigroups of $\tr{X}$. In 2008, Gray and Mitchell \cite{Largest_subsemigroups_transformation} obtained the size of the largest left zero/right zero/completely simple/inverse subsemigroups of $\tr{X}$. In 2023, Cameron et al. \cite{Null_semigroups} described, for a finite set $X$, the null subsemigroups of $\tr{X}$ of maximum size. Furthermore, for $\abs{X}\geqslant 4$, they determined the maximum size of a $3$-nilpotent subsemigroup of $\tr{X}$. More recently, Cain, Malheiro and the present author \cite{Commutative_nilpotent_transformation_semigroups_paper} discovered that the commutative nilpotent subsemigroups of $\tr{X}$ of maximum size are precisely the maximum-order null subsemigroups of $\tr{X}$ characterized by Cameron et al.

Other authors chose to investigate this type of problem in the symmetric inverse semigroup $\psym{X}$ on $X$ (where $X$ is a finite set). In 2007, André, Fernandes and Mitchell \cite{2-generated_submonoids_I_X} obtained, for $\abs{X}\geqslant 3$, the maximum size of a $2$-generated subsemigroup of $\psym{X}$. In 2015, Araújo, Bentz and Konieczny \cite{Commuting_graph_I_X} proved that there is exactly one maximum-order commutative inverse subsemigroup of $\psym{X}$ --- the semigroup of idempotents of $\psym{X}$. They also determined the maximum size of a commutative nilpotent subsemigroup of $\psym{X}$ and demonstrated that, with a few exceptions, the semigroups that achieve that size are all null semigroups. Moreover, they proved that, when $\abs{X}\leqslant 9$, there is just one commutative subsemigroup of $\psym{X}$ of maximum size --- which is the unique commutative \emph{inverse} subsemigroup of $\psym{X}$ of maximum size --- and, when $\abs{X}\geqslant 10$, they proved that the commutative subsemigroups of $\psym{X}$ of maximum size can be obtained by adding the identity transformation to the commutative nilpotent subsemigroups of $\psym{X}$ of maximum size.

The remainder of this section is dedicated to explaining the structure of the paper. In Section~\ref{Preliminaries} we have some basic definitions concerning simple graphs, simple digraphs, commuting graphs, full and partial transformation semigroups, alphabets and words.

In Section~\ref{sec: largest comm smg idemp T(X)} we describe, for a finite set $X$, the maximum-order commutative subsemigroups of idempotents of $\tr{X}$ and the unique maximum-order commutative subsemigroup of idempotents of $\ptr{X}$.

In Section~\ref{sec: largest comm smg 1 idemp T(X)} we describe the maximum-order commutative subsemigroups of $\tr{X}$ with a unique idempotent. In order to prove this result, we use an improved and more complex version of a new combinatorial technique (introduced by Cain, Malheiro and the present author \cite{Commutative_nilpotent_transformation_semigroups_paper}) that involves representing transformation semigroups as rooted labelled trees. We will see that the result proved by Cain, Malheiro and the present author \cite{Commutative_nilpotent_transformation_semigroups_paper} concerning maximum-order commutative nilpotent subsemigroups of $\tr{X}$ is a corollary of the result concerning maximum-order commutative subsemigroups of $\tr{X}$ with a unique idempotent. Additionally, we describe the maximum-order commutative subsemigroups of $\ptr{X}$ with a unique idempotent and, as a corollary, we characterize the maximum-order nilpotent commutative subsemigroups of $\ptr{X}$.

In Section~\ref{sec: largest comm smg T(X)} we focus on the largest commutative subsemigroups of $\tr{X}$ and of $\ptr{X}$. We prove that, when $\abs{X}\leqslant 6$, the largest commutative subsemigroups of $\tr{X}$ are, with a minor exception, precisely the largest commutative subsemigroups of idempotents of $\tr{X}$ (described in Section~\ref{sec: largest comm smg idemp T(X)}); and, when $\abs{X}\geqslant 7$, we give a lower bound for the maximum size of a commutative subsemigroup of $\tr{X}$. (The largest commutative subsemigroups of $\tr{X}$ with a unique idempotent, described in Section~\ref{sec: largest comm smg 1 idemp T(X)}, are involved in the determination of such a lower bound.) Furthermore, we show that, when $\abs{X}\leqslant 5$, the unique largest commutative subsemigroup of $\ptr{X}$ is the unique largest commutative subsemigroup of idempotents of $\ptr{X}$ (described in Section~\ref{sec: largest comm smg idemp T(X)}); and, when $\abs{X}\geqslant 6$, we provide a lower bound and an upper bound for the size of a largest commutative subsemigroup of $\ptr{X}$. (The largest commutative subsemigroups of $\ptr{X}$ with a unique idempotent, described in Section~\ref{sec: largest comm smg 1 idemp T(X)}, are involved in the determination of such a lower bound.)

In Section~\ref{sec: properties comm graph T(X)} we consider the commuting graphs of $\tr{X}$ and of $\ptr{X}$. We use the results from Section~\ref{sec: largest comm smg T(X)} concerning the largest commutative subsemigroups of $\tr{X}$ and of $\ptr{X}$ to study the clique numbers of the commuting graphs of $\tr{X}$ and of $\ptr{X}$. We also obtain the knit degree of $\tr{X}$ and of $\ptr{X}$, as well as the girth of their commuting graphs.

Finally, in Section~\ref{Open problems}, we discuss the open problem of characterizing, when $\abs{X}\geqslant 6$ (respectively, $\abs{X}\geqslant 5$), the largest commutative subsemigroups of $\tr{X}$ (respectively, $\ptr{X}$).

This paper is based on Chapters 3 and 4 of the author's Ph.D. thesis \cite{My_thesis}.

\section{Preliminaries} \label{Preliminaries}

For general background on graphs see, for example, \cite{Graphs_Wilson}. For general background on semigroups we use \cite{Nine_chapters_Cain}.

\subsection{Simple graphs}

A \defterm{simple graph} $G=(V,E)$ consists of a non-empty set $V$ --- whose elements are called \defterm{vertices} --- and a set $E$ --- whose elements are called \defterm{edges} --- formed by $2$-subsets of $V$. Throughout this subsection we will assume that $G=(V,E)$ is a simple graph.

Let $x$ and $y$ be vertices of $G$. If $\set{x,y}\in E$, then we say that the vertices $x$ and $y$ are \defterm{adjacent}. If $\set{x,z}\notin E$ for all $z\in V$ (that is, if $x$ is not adjacent to any other vertex), then we say that $x$ is an \defterm{isolated vertex}.


A simple graph $H=\parens{V',E'}$ is a \defterm{subgraph} of $G$ if $V'\subseteq V$ and $E'\subseteq E$. Note that, since $H$ is a simple graph, the elements of $E'$ are $2$-subsets of $V'$.

Given $V'\subseteq V$, the \defterm{subgraph induced by $V'$} is the subgraph of $G$ whose set of vertices is $V'$ and where two vertices are adjacent if and only if they are adjacent in $G$ (that is, the set of edges of the induced subgraph is $\braces{\braces{x,y}\in E: x,y\in V'}$).

A \defterm{complete graph} is a simple graph where all distinct vertices are adjacent to each other.


A \defterm{null graph} is a simple graph with no edges and where all vertices are isolated vertices.


A \defterm{path} in $G$ from a vertex $x$ to a vertex $y$ is a sequence of pairwise distinct vertices (except, possibly, $x$ and $y$) $x=x_1,x_2,\ldots,x_n=y$ such that $\braces{x_1,x_2}, \braces{x_2,x_3},\ldots, \braces{x_{n-1},x_n}$ are pairwise distinct edges of $G$. The \defterm{length} of the path is the number of edges of the path; thus, the length of our example path is $n-1$. If $x=y$ then we call the path a \defterm{cycle}. Whenever we want to mention a path, we will write that $x=x_1-x_2-\cdots-x_n=y$ is a path (instead of writing that $x=x_1,x_2,\ldots,x_n=y$ is a path).

Let $K\subseteq V$. We say that $K$ is a \defterm{clique} in $G$ if $\braces{x,y}\in E$ for all $x,y\in K$, that is, if the subgraph of $G$ induced by $K$ is complete. The \defterm{clique number} of $G$, denoted $\cliquenumber{G}$, is the size of a largest clique in $G$, that is, $\cliquenumber{G}=\max\left\{|K|: K \text{ is a clique in } G\right\}$.

If the graph $G$ contains cycles, then the \defterm{girth} of $G$, denoted $\girth{G}$, is the length of a shortest cycle in $G$. If $G$ contains no cycles, then $\girth{G}=\infty$.

\subsection{Simple digraphs}

A \defterm{simple digraph} $G=\parens{V,A}$ consists of a non-empty set $V$ and a subset $A$ of $\parens{V\times V}\setminus\gset{\parens{x,x}}{x\in V}$. The elements of $V$ are called \defterm{vertices} and the elements of $A$ are called \defterm{arcs}. Throughout this subsection we will assume that $G=\parens{V,A}$ is a simple digraph.

Let $x$ and $y$ be vertices of $G$. If $\parens{x,y}\in A$ (that is, if $\parens{x,y}$ is an arc), then we say that there is an arc from $x$ to $y$.

The \defterm{outdegree} (respectively, \defterm{indegree}) of a vertex $x$ is the number of arcs in $A$ of the form $\parens{x,y}$ (respectively, $\parens{y,x}$).

A \defterm{directed path} in $G$ from a vertex $x$ to a vertex $y$ is a sequence of pairwise distinct vertices (except, possibly, $x$ and $y$) $x=x_1,x_2,\ldots,x_n=y$ such that $\parens{x_i,x_{i+1}}\in A$ for all $i\in\X{n-1}$.

In Section~\ref{sec: largest comm smg 1 idemp T(X)} we will use directed rooted trees as a tool to prove some results. A \defterm{directed rooted tree} is a simple digraph with a distinguished element --- called the \defterm{root} of the tree --- such that there is a unique directed path from the root to each vertex. We observe that the indegree of the root is equal to $0$. A vertex of outdegree $0$ is called a \defterm{leaf} of the tree.

\subsection{Commuting graphs}
\label{sec: (extended) commgraph}

Recall that the \defterm{center} of a semigroup $S$ is the set
\begin{displaymath}
	\centre{S}= \gset{x\in S}{xy=yx \text{ for all } y\in S},
\end{displaymath}
whose elements are called the \defterm{central} elements of $S$.

The \defterm{commuting graph} of a finite non-commutative semigroup $S$, denoted $\commgraph{S}$, is the simple graph whose set of vertices is $S\setminus \centre{S}$ and where two distinct vertices $x,y\in S\setminus \centre{S}$ are adjacent if and only if $xy=yx$.

We note that the semigroup must be non-commutative because otherwise we would obtain an empty vertex set. 

This definition of commuting graph is used in several other papers: see \cite{Commuting_graph_T_X, Commuting_graph_I_X, Commuting_graph_symmetric_alternating_groups}, for example. Other authors define commuting graphs of semigroups in a slightly different way: in this alternative definition the vertices of the commuting graph are all the elements of the semigroup (instead of just the non-central ones). This definition is used, for example, in \cite{Graphs_arise_as_commuting_graphs_groups, Cameron_commuting_graphs_notes, Commuting_graphs_groups_split}.

The next lemma, whose proof is straightforward, shows the relationship between the largest commutative subsemigroups of a semigroup and the largest cliques in its commuting graph.

\begin{lemma}\label{preli: largest cliques, commutative subsemigroups}
	Let $S$ be a finite non-commutative semigroup and let $\centre{S}\subseteq T\subseteq S$. Then $T$ is a commutative subsemigroup of $S$ of maximum size if and only if $T\setminus\centre{S}$ is a clique in $\commgraph{S}$ of maximum size. In this case we have $\cliquenumber{\commgraph{S}}=\abs{T}-\abs{\centre{S}}$.
\end{lemma}

The two concepts that follow were first introduced by Araújo, Kinyon and Konieczny \cite{Commuting_graph_T_X}, specifically for commuting graphs of semigroups. They had an important role in settling a conjecture posed by Schein (see \cite{Schein_conjecture}) related to $r$-semisimple bands.

Let $S$ be a non-commutative semigroup. A \defterm{left path} in $\commgraph{S}$ is a path $x_1,\ldots,x_n$ in $\commgraph{S}$ such that $x_1\neq x_n$ and $x_1x_i=x_nx_i$ for all $i\in\braces{1,\ldots,n}$. If $\commgraph{S}$ contains left paths, then the \defterm{knit degree} of $S$, denoted $\knitdegree{S}$, is the length of a shortest left path in $\commgraph{S}$.

\subsection{(Full and partial) transformation semigroups}

Let $X$ be a set. The \defterm{transformation semigroup} on $X$, denoted $\tr{X}$, is the semigroup formed by all the (full) transformations on $X$ (that is, all the functions whose domain is $X$ and whose image is contained in $X$) and whose multiplication is the composition of functions. The \defterm{partial transformation semigroup} on $X$, denoted $\ptr{X}$, is the semigroup formed by all the partial transformations on $X$ (that is, all the functions whose domain and image are both contained in $X$) and whose multiplication is the composition of functions. The \defterm{symmetric inverse semigroup} on $X$, denoted $\psym{X}$, is the semigroup of partial injective transformations on $X$ (that is, all the injective functions whose domain and image are both contained in $X$) and whose multiplication is the composition of functions. The \defterm{symmetric group} on $X$, denoted $\sym{X}$, is the group of bijections on $X$ (that is, all the bijective functions whose domain and image are both equal to $X$) and whose multiplication is the composition of functions.

The \defterm{rank} of a transformation $\alpha\in\tr{X}$ is the size of $\im\alpha$.

In the course of the paper we are going to denote by $\id{Y}$, where $Y\subseteq X$, the restriction of $\id{X}$ (the identity transformation on $X$) to the set $Y$; that is, $\id{Y}=\id{X}|_Y$. 

For the remainder of the paper, $X$ will denote a finite set.

In the upcoming sections we will frequently need to use a commutative subsemigroup of $\tr{X}$ to construct a commutative subsemigroup of $\tr{Y}$, for some $Y\subsetneq X$. The next result, which will be used frequently, shows us how we can do that.

\begin{lemma}\label{T(X): lemma induction}
	Let $S$ be a subsemigroup of $\tr{X}$. Let $Y$ be a non-empty subset of $X$ such that $\alpha|_Y\in\tr{Y}$ for all $\alpha\in S$. Let $S'=\gset{\alpha|_Y}{\alpha\in S}$. Then
	\begin{enumerate}
		\item $S'$ is a subsemigroup of $\tr{Y}$.
		\item If $S$ is commutative, then $S'$ is commutative.
		\item If $\alpha\in S$ is an idempotent, then $\alpha|_Y$ is an idempotent.
		\item If $S$ contains a unique idempotent, then $S'$ contains a unique idempotent.
	\end{enumerate}
\end{lemma}

\begin{proof}
	We begin by noticing that, since $\alpha|_Y\in\tr{Y}$ for all $\alpha\in S$, then we have that $\alpha|_Y\beta|_Y=\parens{\alpha\beta}|_Y$ for all $\alpha,\beta\in S$.
	
	\medskip
	
	\textbf{Part 1.} It is clear that $S'\subseteq\tr{Y}$. Moreover, we have $\alpha|_Y\beta|_Y=\parens{\alpha\beta}|_Y\in S'$ for all $\alpha,\beta\in S$. Thus $S'$ is a subsemigroup of $\tr{X}$.
	
	\medskip
	
	\textbf{Part 2.} Suppose that $S$ is commutative. Let $\alpha,\beta\in S$. We have that $\alpha\beta=\beta\alpha$ and, consequently, we have that $\alpha|_Y\beta|_Y=\parens{\alpha\beta}|_Y=\parens{\beta\alpha}|_Y=\beta|_Y\alpha|_Y$.
	
	\medskip
	
	\textbf{Part 3.} Suppose that there exists $\alpha\in S$ such that $\alpha^2=\alpha$. Then $\alpha|_Y\alpha|_Y=\parens{\alpha\alpha}|_Y=\alpha|_Y$ and, consequently, $\alpha|_{Y }$ is an idempotent.
	
	\medskip
	
	\textbf{Part 4.} Suppose that $S$ contains a unique idempotent. Let $e\in S$ be that idempotent. It follows from part 3 that $e|_Y$ is an idempotent. We want to show that $e|_Y$ is the unique idempotent of $S'$. Let $\alpha\in S$ be such that $\alpha|_Y$ is an idempotent. We are going to see that $\alpha|_Y=e|_Y$. We know that there exists $n\in\mathbb{N}$ such that $\alpha^n$ is an idempotent, which implies that $\alpha^n=e$ (because $e$ is the unique idempotent of $S$). Since $\alpha|_Y\alpha|_Y=\alpha|_Y$, then we have $\alpha|_Y=\parens{\alpha|_Y}^n=\parens{\alpha^n}|_Y=e|_Y$. Thus $e|_Y$ is the unique idempotent of $S'$.
\end{proof}

	In the course of the paper we will often prove results in $\ptr{X}$ by using the corresponding results in $\tr{Y}$, where $Y$ is going to be a convenient set. The idea is to choose a maximum-order commutative subsemigroup of $\ptr{X}$ of a certain type and then construct a commutative subsemigroup of $\tr{Y}$ of the same type and of the same size. The results concerning $\tr{Y}$ allow us to determine an upper bound for the size of the subsemigroups of $\ptr{X}$ and, later, they help us characterize the maximum-order commutative subsemigroup of $\ptr{X}$ we chose. With this in mind, we start explaining how we can obtain a full transformation semigroup from a partial transformation semigroup.
	
	Let $\new$ be a new symbol not in $X$ and let $\newtr{X}=X\cup\set{\new}$. For each $\beta\in\ptr{X}$ we define a full transformation in $\tr{\newtr{X}}$, which we denote by $\newtr{\beta}$, the following way: for all $x\in \newtr{X}=X\cup\set{\new}$
	\begin{displaymath}
		x\newtr{\beta}=\begin{cases}
			x\beta &\text{if } x\in\dom\beta, \\
			\new &\text{if } x\in \newtr{X}\setminus\dom\beta.
		\end{cases}
	\end{displaymath}
	We observe that, in particular, we have $\new\newtr{\beta}=\new$. Moreover, for each subsemigroup $S$ of $\ptr{X}$ we define $\newtr{S}=\gset{\newtr{\beta}\in\tr{\newtr{X}}}{\beta\in S}$.

\begin{proposition}\label{P(X): proposition P(X) --> T(Y)}
	Let $S$ be a subsemigroup of $\ptr{X}$. Then $\newtr{S}$ is a subsemigroup of $\tr{\newtr{X}}$ isomorphic to $S$.
\end{proposition}

\begin{proof}
	First, we are going to establish that $\newtr{\beta}\newtr{\gamma}=\newtr{\parens{\beta\gamma}}$ for all $\beta,\gamma\in S$. Let $\beta,\gamma\in S$. We have three possible cases.
	
	\smallskip
	
	\textit{Case 1:} Assume that $x\in\dom\beta\gamma$. Then $x\in\dom\beta$ and $x\beta\in\dom\gamma$, which implies that 
	\begin{align*}
		x\newtr{\beta}\newtr{\gamma} &=\parens{x\newtr{\beta}}\newtr{\gamma}\\
		&=\parens{x\beta}\newtr{\gamma}  &\bracks{\text{because } x\in\dom\beta}\\
		&=\parens{x\beta}\gamma &\bracks{\text{because } x\beta\in\dom\gamma}\\
		&=x\beta\gamma\\
		&=x\newtr{\parens{\beta\gamma}}. &\bracks{\text{because } x\in\dom\beta\gamma}
	\end{align*}
	
	\smallskip
	
	\textit{Case 2:} Assume that $x\in \newtr{X}\setminus\dom\beta\gamma$ and $x\in\dom\beta$. Then $x\beta\in \newtr{X}\setminus\dom\gamma$. We have
	\begin{align*}
		x\newtr{\beta}\newtr{\gamma} &=\parens{x\newtr{\beta}}\newtr{\gamma}\\
		&=\parens{x\beta}\newtr{\gamma}  &\bracks{\text{because } x\in\dom\beta}\\
		&=\new &\bracks{\text{because } x\beta\in \newtr{X}\setminus\dom\gamma}\\
		&=x\newtr{\parens{\beta\gamma}}. &\bracks{\text{because } x\in \newtr{X}\setminus\dom\beta\gamma}
	\end{align*}
	
	\smallskip
	
	\textit{Case 3:} Assume that $x\in \newtr{X}\setminus\dom\beta\gamma$ and $x\in \newtr{X}\setminus\dom\beta$. Hence
	\begin{align*}
		x\newtr{\beta}\newtr{\gamma} &=\parens{x\newtr{\beta}}\newtr{\gamma}\\
		&=\new\newtr{\gamma}  &\bracks{\text{because } x\in \newtr{X}\setminus\dom\beta}\\
		&=\new &\bracks{\text{because } \new\newtr{\gamma}=\new}\\
		&=x\newtr{\parens{\beta\gamma}}. &\bracks{\text{because } x\in \newtr{X}\setminus\dom\beta\gamma}
	\end{align*}
	
	\smallskip
	
	This concludes the proof that $\newtr{\beta}\newtr{\gamma}=\newtr{\parens{\beta\gamma}}$ for all $\beta,\gamma\in S$.
	
	Let $\Psi: S \to \tr{\newtr{X}}$ be the map defined by $\beta\Psi=\newtr{\beta}$ for all $\beta\in S$. It is clear that $\parens{\beta\gamma}\Psi= \newtr{\parens{\beta\gamma}} =\newtr{\beta}\newtr{\gamma} =\parens{\beta\Psi}\parens{\gamma\Psi}$ for all $\beta,\gamma\in S$, which implies that $\Psi$ is a homomorphism. Hence $\newtr{S}=S\Psi$ is a subsemigroup of $\tr{\newtr{X}}$.
	
	Now we are going to check that $\Psi$ is also injective. Let $\beta,\gamma\in S$ be such that $\beta\Psi=\gamma\Psi$. Then $\newtr{\beta}=\newtr{\gamma}$ and, consequently, for all $x\in \newtr{X}$ we have
	\begin{displaymath}
		x\in \dom\beta \iff x\newtr{\beta}\neq\new \iff x\newtr{\gamma}\neq\new \iff x\in \dom\gamma.
	\end{displaymath}
	Hence $\dom\beta=\dom\gamma$. Furthermore, for all $x\in \dom\beta=\dom\gamma$ we have $x\beta=x\newtr{\beta}=x\newtr{\gamma}=x\gamma$. Thus $\beta=\gamma$ and $\Psi$ is injective.
	
	We have that $\Psi$ is an injective homomorphism. Thus $S$ is isomorphic to $S\Psi=\newtr{S}$.
\end{proof}

\subsection{Alphabets and words}

An \defterm{alphabet} $X$ is a non-empty set whose elements are called \defterm{letters}. A \defterm{word} over $X$ is a finite sequence of letters of $X$; that is, a word is a sequence of the form $x_1x_2\cdots x_n$, where for all $i\in\Xn$ we have that $x_i\in X$ is a letter. The \defterm{length} of a word corresponds to the number of letters of that word (and so the length of the word $x_1x_2\cdots x_n$ is $n$). The word with no letters, which has length $0$, is called the \defterm{empty word} and it is usually denoted by $\varepsilon$. When we have two words $x_1x_2\cdots x_n$ and $y_1y_2\cdots y_m$, we can use the operation of \defterm{concatenation}, which is associative, to form a new word $x_1x_2\cdots x_ny_1y_2\cdots y_m$. The set of all words over $X$ (including the empty word $\varepsilon$) is $X^*$, which, when equipped with the operation of concatenation, forms a monoid (and $\varepsilon$ is its identity). A \defterm{prefix} of a word $w\in X^*$ is another word $u\in X^*$ such that $w=uv$ for some word $v\in X^*$. This means that the prefix of a word $x_1x_2\cdots x_n$ is either $\varepsilon$ or a word of the form $x_1x_2\cdots x_m$ for some $m\in\Xn$.


 \section{The largest commutative (full and partial) transformation semigroups of idempotents}\label{sec: largest comm smg idemp T(X)}
 
 Recall that $X$ denotes a finite set. The aim of this section is to establish that the maximum size of a commutative subsemigroup of idempotents of $\tr{X}$ (respectively, $\ptr{X}$) is $2^{\abs{X}-1}$ (respectively, $2^{\abs{X}}$). We also characterize the commutative subsemigroups of idempotents of $\tr{X}$ of maximum size, and we prove that there is exactly one commutative subsemigroup of idempotents of $\ptr{X}$ of maximum size --- namely $\idemp{\psym{X}}$, the set of idempotents of $\psym{X}$.

 We begin by characterizing the largest commutative subsemigroups of idempotents of $\tr{X}$. With this goal in mind, for each $x\in X$ we define
 \begin{equation}\label{T(X): Gamma^X_i definition}
 	\commidemp{X}{x}=\gset{\alpha\in \tr{X}}{x\alpha=x \text{ and } y\alpha\in\set{x,y} \text{ for all } y\in X\setminus\set{x}}.
 \end{equation}
 
 \begin{proposition}\label{T(X): Gamma^X_i comm smg idemp}
 	For each $x\in X$, we have that $\commidemp{X}{x}$ is a commutative subsemigroup of idempotents of $\tr{X}$ of size $2^{\abs{X}-1}$.
 \end{proposition}
 
 \begin{proof}
 	Let $x\in X$. For each $\alpha\in\commidemp{X}{x}$ we have $x\alpha=x$ and $y\alpha\in\set{x,y}$ for all $y\in X\setminus\set{x}$. Since there is exactly one possibility for the image of $x$ and exactly two possibilities for the image of every element of $X\setminus\set{x}$, we have that $\abs{\commidemp{X}{x}}=2^{\abs{X\setminus\set{x}}}=2^{\abs{X}-1}$.
 	
 	Now we establish that $\commidemp{X}{x}$ is a subsemigroup of $\tr{X}$. Let $\alpha,\beta \in \commidemp{X}{x}$. Then $x\alpha=x\beta=x$ and $y\alpha,y\beta\in\set{x,y}$ for all $y\in X\setminus\set{x}$. Hence we have $x\alpha\beta=x\beta=x$ and $y\alpha\beta=\parens{y\alpha}\beta\in\set{x\beta,y\beta}\subseteq\set{x,y}$ for all $y\in X\setminus\set{x}$. Therefore $\commidemp{X}{x}$ is a subsemigroup of $\tr{X}$.
 	
 	Let $y\in X$. In the four cases below we prove that $\commidemp{X}{x}$ is commutative.
 	
 	\smallskip
 	
 	\textit{Case 1:} Assume that $y\alpha=y\beta=x$. Then $y\alpha\beta=x\beta=x=x\alpha=y\beta\alpha$.
 	
 	\smallskip
 	
 	\textit{Case 2:} Assume that $y\alpha=x$ and $y\beta=y$. Then $y\alpha\beta=x\beta=x=y\alpha=y\beta\alpha$.
 	
 	\smallskip
 	
 	\textit{Case 3:} Assume that $y\alpha=y$ and $y\beta=x$. Then $y\alpha\beta=y\beta=x=x\alpha=y\beta\alpha$.
 	
 	\smallskip
 	
 	\textit{Case 4:} Assume that $y\alpha=y\beta=y$. Then $y\alpha\beta=y\beta=y=y\alpha=y\beta\alpha$.
 	
 	\smallskip
 	
 	Additionally, for all $y\in X$ we have
 	\begin{displaymath}
 		y\alpha^2 = (y\alpha)\alpha = \braces*{
 			\begin{array}{@{}l@{}ll@{}}
 				y\alpha && \text{(if $y\alpha = y$)} \\[1mm]
 				x\alpha &{}= x& \text{(if $y\alpha = x$)} \\
 			\end{array}
 		} = y\alpha,
 	\end{displaymath}
 	and so $\alpha$ is an idempotent. Therefore $\commidemp{X}{x}$ is a semigroup of idempotents.
 \end{proof}
 
 Now our objective is to show that the largest commutative subsemigroups of idempotents of $\tr{X}$ are precisely the semigroups $\commidemp{X}{x}$ (where $x\in X$), which have size $2^{\abs{X}-1}$. We will prove this result by induction on the size of $X$ (Theorem~\ref{T(X): maximum size comm smg idemp}). In order to use the induction step we need to be able to take a commutative subsemigroup $S$ of idempotents of $\tr{X}$ and use it to construct a commutative subsemigroup of idempotents of $\tr{Y}$, for some $Y\subsetneq X$. It follows from Lemma~\ref{T(X): lemma induction} that it is enough to prove the existence of a non-empty proper subset $Y$ of $X$ such that $\alpha|_Y\in\tr{Y}$ for all $\alpha \in S$. In Lemma~\ref{T(X): existence of I, b|_(X/I)} we will prove that this set exists whenever $S\nsubseteq\sym{X}$. Moreover, Lemma~\ref{T(X): lemma alpha|_I cycle} will be helpful in proving that, when we are dealing with commutative semigroups of idempotents, we can assume that $Y$ has size $\abs{X}-1$.
 
 \begin{lemma}\label{T(X): ab=ba => xb€Im a for all x€Im a}
 	Let $\alpha,\beta\in\tr{X}$ be such that $\alpha\beta=\beta\alpha$. Then $x\beta\in\im\alpha$ for all $x\in\im\alpha$.
 \end{lemma}
 
 \begin{proof}
 	Let $x\in\im \alpha$. Then there exists $y\in X$ such that $y\alpha=x$. We have $x\beta=y\alpha\beta=y\beta\alpha\in\im\alpha$.
 \end{proof}
 
 We mentioned earlier that the lemma below is going to be used (in the proof of Theorem~\ref{T(X): maximum size comm smg idemp}) to obtain a set $Y\subsetneq X$ and a commutative subsemigroup of idempotents of $\tr{Y}$ from a commutative subsemigroup of idempotents of $\tr{X}$. However, Lemma~\ref{T(X): existence of I, b|_(X/I)} can be employed more generally, in the sense that it can be applied to any commutative subsemigroup of $\tr{X}$ to obtain a set $Y\subsetneq X$ and a commutative subsemigroup of $\tr{Y}$. This will be useful later in Theorem~\ref{T(X): maximum size comm smg}, which gives the maximum size of a commutative subsemigroup of $\tr{X}$ (when $\abs{X}\leqslant 6$).
 
 \begin{lemma}\label{T(X): existence of I, b|_(X/I)}
 	Let $S$ be a commutative subsemigroup of $\tr{X}$. Suppose that $S\nsubseteq \sym{X}$. Then there exists a non-empty proper subset $I$ of $X$ such that $\beta|_{X\setminus I}\in\tr{X\setminus I}$ for all $\beta\in S$.
 \end{lemma}
 
 \begin{proof}
 	It follows from the fact that $S\nsubseteq \sym{X}$ that there exists $\alpha\in S$ such that $\alpha\notin \sym{X}$. Let $I=X\setminus\im\alpha\subseteq X$. We have $\im\alpha\neq\emptyset$ and $\im\alpha\neq X$ (because $\alpha\notin\sym{X}$), which implies that $I\neq X$ and $I\neq\emptyset$. 
 	
 	Let $\beta\in S$. Then $\alpha\beta=\beta\alpha$. It follows from Lemma~\ref{T(X): ab=ba => xb€Im a for all x€Im a} that for all $x\in X\setminus I=\im\alpha$ we have $x\beta\in \im\alpha=X\setminus I$. Hence $\beta|_{X\setminus I}\in\tr{X\setminus I}$.
 \end{proof}
 
 Lemma~\ref{T(X): existence of I, b|_(X/I)} motivates the following definition: for any commutative subsemigroup $S$ of $\tr{X}$, let
 \begin{equation}\label{T(X): C(S,X) definition}
 	\setclass{S}{X}=\gset{I}{\emptyset\neq I\subsetneq X \text{ and } \beta|_{X\setminus I}\in\tr{X\setminus I} \text{ for all } \beta\in S}.
 \end{equation}
 We observe that the class $\setclass{S}{X}$ might be empty. However, if $S\nsubseteq\sym{X}$, then Lemma~\ref{T(X): existence of I, b|_(X/I)} implies that $\setclass{S}{X}$ is not empty.
 
 \begin{lemma}\label{T(X): lemma alpha|_I cycle}
 	Let $S$ be a commutative subsemigroup of $\tr{X}$ such that $\setclass{S}{X}\neq\emptyset$ and let $I\in\setclass{S}{X}$ be of minimum size. If $\abs{I}\geqslant 2$, then there exists $\alpha \in S$ such that $\alpha|_I\in\sym{I}$ is a product of disjoint cycles of the same length and that length is at least $2$.
 \end{lemma}
 
 \begin{proof}
 	Suppose that $\abs{I}\geqslant 2$. Let $y\in I$. We have that $\set{y}\subsetneq I$ and $\emptyset\neq\set{y}\subsetneq X$. Moreover, since $I$ is an element of $\setclass{S}{X}$ of minimum size, then $\set{y}\notin\setclass{S}{X}$. Hence there exists $\alpha\in S$ such that $\parens{X\setminus\set{y}}\alpha\nsubseteq X\setminus\set{y}$. This implies the existence of $x\in X\setminus\set{y}$ such that $x\alpha\notin X\setminus\set{y}$, that is, such that $x\alpha=y$. Due to the fact that $I\in\setclass{S}{X}$, we have that $\parens{X\setminus I}\alpha\subseteq X\setminus I\subseteq X\setminus\set{y}$. Then the fact that $x\alpha\notin X\setminus\set{y}$ implies that $x\alpha\notin \parens{X\setminus I}\alpha$ and, consequently, that $x\in I$. Hence $y=x\alpha\in I\alpha$.
 	
 	We divide the remaining of the proof into three parts: first we are going to prove that $\alpha|_I\in\sym{I}$, then we are going to prove that all cycles in $\alpha|_I$ have length at most $2$ and finally we are going to prove that all cycles in $\alpha|_I$ have the same length.
 	
 	\medskip
 	
 	\textbf{Part 1.} The aim of this part is to prove that $\alpha|_I\in\sym{I}$. In order to do this, we show that $I\alpha=I$, so that $\alpha|_I$ is a bijection. Let $J=I\setminus I\alpha$. The fact that $y\in I\cap I\alpha$ implies that $J\subsetneq I$ and so $\abs{J}<\abs{I}$.
 	
 	Let $\beta\in S$ and $z\in X\setminus J$. Then $z\in X\setminus I$ or $z\in I\alpha$.
 	
 	\smallskip
 	
 	\textit{Case 1:} Assume that $z\in X\setminus I$. Since $I\in\setclass{S}{X}$, then we have $z\beta\in \parens{X\setminus I}\beta\subseteq X\setminus I\subseteq X\setminus J$.
 	
 	\smallskip
 	
 	\textit{Case 2:} Assume that $z\in I\alpha$. Then there exists $i\in I$ such that $z=i\alpha$. It follows from the fact that $S$ is commutative that $z\beta=i\alpha\beta=i\beta\alpha$. If $i\beta\in X\setminus I$, then $z\beta\in \parens{X\setminus I}\alpha\subseteq X\setminus I\subseteq X\setminus J$ (because $I\in\setclass{S}{X}$). If $i\beta \in I$, then $z\beta\in I\alpha$, which implies that $z\beta \in X\setminus J$.
 	
 	\smallskip
 	
 	Since $\beta$ and $z$ are arbitrary elements of $S$ and $X\setminus J$, respectively, then we can conclude that $\parens{X\setminus J}\beta\subseteq X\setminus J$ for all $\beta\in S$, that is, $\beta|_{X\setminus J}\in\tr{X\setminus J}$ for all $\beta \in S$. As a consequence of the minimality of the size of $I$ we must have $J\notin\setclass{S}{X}$. Hence $J=\emptyset$ (because $J\subsetneq I\subsetneq X$) and, consequently, $I\subseteq I\alpha$. Since $\abs{I\alpha}\leqslant\abs{I}$, then $I=I\alpha$, which concludes the proof of part 1.
 	
 	\medskip
 	
 	\textbf{Part 2.} In the previous part we proved that $\alpha|_I\in\sym{I}$. This implies that $\alpha|_I$ can be written as the product of disjoint cycles. The aim of this part is to see that none of those cycles have length $1$. We prove this by establishing that $z\alpha\neq z$ for all $z\in I$. Let $Y=\gset{z\in I}{z\alpha\neq z}$. We have $\abs{Y}\leqslant\abs{I}$.
 	
 	Let $\beta\in S$ and $z\in X\setminus Y$. Then we have $z\in X\setminus I$ or $z\alpha=z$.
 	
 	\smallskip
 	
 	\textit{Case 1:} Assume that $z\in X\setminus I$. It follows from the fact that $I\in\setclass{S}{X}$ that $z\beta\in\parens{X\setminus I}\beta\subseteq X\setminus I\subseteq X\setminus Y$.
 	
 	\smallskip
 	
 	\textit{Case 2:} Assume that $z\alpha=z$. Since $S$ is commutative, then we have that $\parens{z\beta}\alpha=z\alpha\beta=z\beta$. Hence $z\beta\in X\setminus Y$.
 	
 	\smallskip
 	
 	Since $\beta$ and $z$ are arbitrary elements of $S$ and $X\setminus Y$, respectively, then we can conclude that $\beta|_{X\setminus Y}\in\tr{X\setminus Y}$ for all $\beta \in S$. We also have $Y\subseteq I\subsetneq X$. Moreover, the fact that $x\in I$ and $x\alpha=y\neq x$ implies that $x\in Y$ and, consequently, that $Y\neq\emptyset$. Hence $Y\in\setclass{S}{X}$ and it follows from the minimality of the size of $I$ that $\abs{Y}\geqslant\abs{I}$. Therefore $\abs{Y}=\abs{I}$ and, consequently, $Y=I$. Thus $z\alpha\neq z$ for all $z\in I$, which implies that none of the cycles in $\alpha|_I$ has length $1$.
 	
 	\medskip
 	
 	\textbf{Part 3.} Finally, we are going to see that all the cycles in $\alpha|_{I}$ have the same length. Let $k$ be the maximum length of a cycle in $\alpha|_{I}$ and let $K=\gset{z\in I}{z \text{ belongs to a cycle in } \alpha|_{I} \text{ of } \allowbreak \text{length } k}$. We have $\abs{K}\leqslant\abs{I}$.
 	
 	Let $\beta\in S$ and $z\in X\setminus K$.
 	
 	\smallskip
 	
 	\textit{Case 1:} Assume that $z\in X\setminus I$. It follows from the fact that $I\in\setclass{S}{X}$ that $z\beta\in\parens{X\setminus I}\beta\subseteq X\setminus I\subseteq X\setminus K$.
 	
 	\smallskip
 	
 	\textit{Case 2:} Assume that $z\beta\in X\setminus I$. Then we can immediately conclude that $z\beta\in X\setminus K$.
 	
 	\smallskip
 	
 	\textit{Case 3:} Assume that $z\in I$ and $z\beta\in I$. Let $l$ be the length of the cycle in $\alpha|_I$ to which $z$ belongs. Since $z\in X\setminus K$, then that cycle has length at most $k-1$, that is, $l<k$. We have that $z\beta=\parens{z\alpha^l}\beta=\parens{z\beta}\alpha^l\in I\alpha^l=I$ (since $S$ is commutative and $I\alpha=I$, by part 1 of the proof). Hence $z\beta$ belongs to a cycle in $\alpha|_I$ whose length is at most $l$. Since $l<k$, then we can conclude that $z\beta\in X\setminus K$.
 	
 	\smallskip

 	
 	Since $\beta$ and $z$ are arbitrary elements of $S$ and $X\setminus K$, respectively, then this means we just proved that $\beta|_{X\setminus K}\in\tr{X\setminus K}$ for all $\beta \in S$. Since we also have $K\neq\emptyset$ and $K\subseteq I\subsetneq X$, then we can conclude that $K\in\setclass{S}{X}$. Furthermore, as a consequence of the minimality of the size of $I$ we have $\abs{K}\geqslant\abs{I}$. Thus $\abs{K}=\abs{I}$ and, consequently, $K=I$. This implies that all the elements of $I$ lie in a cycle in $\alpha|_I$ of maximum length. Therefore all cycles in $\alpha|_I$ have the same length.
 \end{proof}
 
 At last, we can characterize the maximum-order commutative subsemigroups of idempotents of $\tr{X}$.
 
 \begin{theorem}\label{T(X): maximum size comm smg idemp}
 	The maximum size of a commutative subsemigroup of idempotents of $\tr{X}$ is $2^{\abs{X}-1}$. Moreover, the maximum-order commutative subsemigroups of idempotents of $\tr{X}$ are precisely the semigroups $\commidemp{X}{x}$, where $x\in X$.
 \end{theorem}
 
 \begin{proof}
 	We are going to prove, by induction on the size of $X$, that the largest commutative subsemigroups of idempotents of $\tr{X}$ are precisely the semigroups $\commidemp{X}{x}$, where $x\in X$. We note that, by Proposition~\ref{T(X): Gamma^X_i comm smg idemp}, these semigroups are commutative subsemigroups of idempotents of $\tr{X}$ of size $2^{\abs{X}-1}$.
 	
 	Suppose that $\abs{X}=1$. Then $\tr{X}=\set{\id{X}}=\commidemp{X}{\id{X}}$ is a commutative semigroup of idempotents of size $1=2^{\abs{X}-1}$.
 	
 	Suppose that $\abs{X}\geqslant 2$ and assume that for any set $Y$ such that $\abs{Y}=\abs{X}-1$ we have that the largest commutative subsemigroup of idempotents of $\tr{Y}$ are precisely the semigroups $\commidemp{Y}{x}$, where $x\in Y$. (This is the induction hypothesis).
 	
 	Let $S$ be a largest commutative subsemigroup of idempotents of $\tr{X}$. Since for each $x\in X$ we have that $\commidemp{X}{x}$ is a commutative subsemigroup of idempotents of $\tr{X}$ of size $2^{\abs{X}-1}$ (by Proposition~\ref{T(X): Gamma^X_i comm smg idemp}), then we have $\abs{S}\geqslant 2^{\abs{X}-1}$.
 	
 	As a consequence of the fact that $\abs{X}\geqslant 2$, we have $\abs{S}\geqslant 2^{\abs{X}-1}\geqslant 2$ and, since $S$ is a semigroup of idempotents, then $S$ contains an idempotent distinct from $\id{X}$ --- the unique idempotent of $\sym{X}$. Hence $S\nsubseteq\sym{X}$ and, consequently, Lemma~\ref{T(X): existence of I, b|_(X/I)} guarantees that the class $\setclass{S}{X}$ is not empty. Let $I$ be a smallest set in $\setclass{S}{X}$.
 	
 	Assume, with the aim of obtaining a contradiction, that $\abs{I}\geqslant 2$. Then Lemma~\ref{T(X): lemma alpha|_I cycle} guarantees the existence of $\alpha\in S$ such that $\alpha|_I\in \sym{I}$ is a product of (disjoint) cycles whose length is at least $2$. Let $i\in I$. Since $\alpha|_I\in\sym{I}$, then there exists $j\in I$ such that $j\alpha=i$. It follows from the fact that $\alpha$ is an idempotent that $i=j\alpha=j\alpha^2=\parens{j\alpha}\alpha=i\alpha=i\alpha|_I$. Therefore $\parens{i}$ is a cycle (of length $1$) in $\alpha|_I$, which is a contradiction.

 	
 	
 	
 	Therefore $\abs{I}=1$. Assume that $I=\set{i}$. We have that $\beta|_{X\setminus\set{i}}\in\tr{X\setminus\set{i}}$ for all $\beta\in S$. Let $S'=\gset{\beta|_{X\setminus\set{i}}}{\beta\in S}$. It follows from Lemma~\ref{T(X): lemma induction}, and the fact that $S$ is a commutative subsemigroup of idempotents of $\tr{X}$, that $S'$ is a commutative subsemigroup of idempotents of $\tr{X\setminus\set{i}}$. Since $\abs{X\setminus\set{i}}=\abs{X}-1$, then, by the induction hypothesis, the largest commutative subsemigroups of idempotents of $\tr{X\setminus\set{i}}$ are the semigroups (of size $2^{\abs{X\setminus\set{i}}-1}$) $\commidemp{X\setminus\set{i}}{x}$, where $x\in X\setminus\set{i}$. Consequently, we have that $\abs{S'}\leqslant 2^{\abs{X\setminus\set{i}}-1}=2^{\abs{X}-2}$.
 	
 	For each $\gamma\in S'$ let $S_\gamma=\gset{\beta\in S}{\beta|_{X\setminus\set{i}}=\gamma}$. It is straightforward to see that $\set{S_\gamma}_{\gamma\in S'}$ forms a partition of $S$. Then we have $\abs{S}=\sum_{\gamma\in S'}\abs{S_\gamma}$.

 	We are going to show that $\abs{S_\gamma}\leqslant 2$ for all $\gamma\in S'$. Let $\gamma\in S'$. If $\abs{S_\gamma}\leqslant 1$, then the result follows. Now assume that $\abs{S_\gamma}\geqslant 2$. Then there exist distinct $\beta_1,\beta_2\in S_\gamma$. We have that $\beta_1|_{X\setminus\set{i}}=\gamma=\beta_2|_{X\setminus\set{i}}$. Consequently, we must have $i\beta_1\neq i\beta_2$. Hence $i\beta_1\neq i$ or $i\beta_2\neq i$. Assume, without loss of generality, that $i\beta_1\neq i$. Then we have that
 	\begin{align*}
 		\parens{i\beta_2}\beta_1&=\parens{i\beta_1}\beta_2 &\bracks{\text{since } \beta_1,\beta_2\in S, \text{ which is commutative}}\\
 		&=\parens{i\beta_1}\beta_2|_{X\setminus\set{i}} &\bracks{\text{since } i\beta_1\in X\setminus\set{i}}\\
 		&=\parens{i\beta_1}\beta_1|_{X\setminus\set{i}} &\bracks{\text{since } \beta_1|_{X\setminus\set{i}}=\beta_2|_{X\setminus\set{i}}}\\
 		&=\parens{i\beta_1}\beta_1\\
 		&=i\beta_1 &\bracks{\text{since } \beta_1 \text{ is an idempotent}}\\
 		&\neq i\beta_2\\
 		&=\parens{i\beta_2}\beta_2. &\bracks{\text{since } \beta_2 \text{ is an idempotent}}
 	\end{align*}
 	Due to the fact that $\beta_1|_{X\setminus\set{i}}=\beta_2|_{X\setminus\set{i}}$, we must have $i\beta_2=i$.
 	
 	We just proved that, given two distinct transformations $\beta_1,\beta_2\in S_\gamma$, we must have $i\beta_1=i$ or $i\beta_2=i$. Since $\beta_1|_{X\setminus\set{i}}=\beta_2|_{X\setminus\set{i}}$, then we can conclude that $\abs{S_\gamma}\leqslant 2$.
 	
 	Therefore
 	\begin{displaymath}
 		2^{\abs{X}-1}\leqslant \abs{S}=\sum_{\gamma\in S'}\abs{S_\gamma} \leqslant \sum_{\gamma\in S'}2 =\abs{S'}\cdot 2 \leqslant 2^{\abs{X}-2}\cdot 2 =2^{\abs{X}-1}
 	\end{displaymath}
 	and, consequently, $\abs{S}=2^{\abs{X}-1}$ and $\abs{S'}=2^{\abs{X}-2}=2^{\abs{X\setminus\set{i}}-1}$. According to the induction hypothesis, $2^{\abs{X\setminus\set{i}}-1}$ is the maximum size of a commutative subsemigroup of idempotents of $\tr{X\setminus\set{i}}$. Hence (by the induction hypothesis) $S'=\commidemp{X\setminus\set{i}}{x}$ for some $x\in X\setminus\set{i}$.
 	
 	Our objective is to demonstrate that $S=\commidemp{X}{x}$. We observe that, as a consequence of the fact that $\abs{S}=2^{\abs{X}-1}=\abs{\commidemp{X}{x}}$, it is enough to prove that $S\subseteq \commidemp{X}{x}$. Hence we just need to verify that for each $\beta\in S$ we have $x\beta=x$ and $y\beta\in\set{x,y}$ for all $y\in X\setminus\set{x}$.
 	
 	It follows from the fact that $S'=\commidemp{X\setminus\set{i}}{x}$ that, for each $\beta\in S$, we have that $x\beta=x\beta|_{X\setminus\set{i}}=x$ and $y\beta=y\beta|_{X\setminus\set{i}}\in\set{x,y}$ for all $y\in X\setminus\set{x,i}$. Consequently, we only need to verify that $i\beta\in\set{x,i}$ for all $\beta\in S$.
 	
 	First we observe that, since $\sum_{\gamma\in S'}\abs{S_\gamma}=\abs{S}=2^{\abs{X}-1}$ and $\abs{S'}=2^{\abs{X}-2}$ and $\abs{S_\gamma}\leqslant 2$ for all $\gamma\in S'$, then we have that $\abs{S_\gamma}=2$ for all $\gamma\in S'$.
 	
 	It follows from the fact that $S'=\commidemp{X\setminus\set{i}}{x}$ that there exists $\gamma\in S$ such that $\parens{X\setminus\set{i}}\gamma|_{X\setminus\set{i}}=\set{x}$. Furthermore, since $\abs[\big]{S_{\gamma|_{X\setminus\set{i}}}}=2$, then (by what we established earlier in the proof) there exists $\gamma'\in S_{\gamma|_{X\setminus\set{i}}}$ such that $i\gamma'=i$. We have that $\gamma'|_{X\setminus\set{i}}=\gamma|_{X\setminus\set{i}}$, which implies that $\parens{X\setminus\set{i}}\gamma'=\set{x}$. Hence $\im \gamma'=\set{x,i}$. Therefore, by Lemma~\ref{T(X): ab=ba => xb€Im a for all x€Im a}, we have that $i\beta\in\im\gamma'=\set{x,i}$ for all $\beta\in S$.
 	
 	Thus $S\subseteq\commidemp{X}{x}$, which concludes the proof.
 \end{proof}

 The last result of the section concerns the maximum-order commutative subsemigroups of idempotents of $\ptr{X}$.
 
 For all $Y\subseteq X$ we have that $\id{Y}\in\psym{X}$ is an idempotent. Moreover, the set of idempotents of $\psym{X}$ is $\idemp{\psym{X}}=\gset{\id{Y}}{Y\subseteq X}\subseteq\ptr{X}$, which is a subsemigroup of $\psym{X}$ and, consequently, a subsemigroup of $\ptr{X}$. Furthermore, since $\psym{X}$ is an inverse semigroup, its idempotents commute, which implies that $\idemp{\psym{X}}$ is a commutative subsemigroup of idempotents of $\ptr{X}$. In addition, we can easily see that $\abs{\idemp{\psym{X}}}=\abs{\mathbb{P}\parens{X}}=2^{\abs{X}}$ (where $\mathbb{P}\parens{X}$ is the power set).
 
 In the next corollary we are going to show that $\idemp{\psym{X}}$ is precisely the unique maximum-order commutative subsemigroup of idempotents of $\ptr{X}$, which makes $2^{\abs{X}}$ the maximum size of a commutative subsemigroup of idempotents of $\ptr{X}$.
 
 \begin{corollary}\label{P(X): largest comm smg idemp}
 	The maximum size of a commutative subsemigroup of idempotents of $\ptr{X}$ is $2^{\abs{X}}$. Moreover, the unique maximum-order commutative subsemigroup of idempotents of $\ptr{X}$ is $\idemp{\psym{X}}$.
 \end{corollary}
 
 \begin{proof}
 	\textbf{Part 1.} The aim of this part is to show that the maximum size of a commutative subsemigroup of idempotents of $\ptr{X}$ is $2^{\abs{X}}$. Let $S$ be a commutative subsemigroup of idempotents of $\ptr{X}$. It follows from Proposition~\ref{P(X): proposition P(X) --> T(Y)} that $\newtr{S}$ is a subsemigroup of $\tr{\newtr{X}}$. Additionally, Proposition~\ref{P(X): proposition P(X) --> T(Y)} also states that $\newtr{S}$ is isomorphic to $S$. Hence $\abs{\newtr{S}}=\abs{S}$ and, since $S$ is commutative and all its elements are idempotents, then $\newtr{S}$ is commutative and all its elements are idempotents. Thus $\newtr{S}$ is a commutative subsemigroup of idempotents of $\tr{\newtr{X}}$ and, consequently, Theorem~\ref{T(X): maximum size comm smg idemp} ensures that
 	\begin{displaymath}
 		\abs{S}=\abs{\newtr{S}}\leqslant 2^{\abs{\newtr{X}}-1}=2^{\abs{X}}.
 	\end{displaymath}
 	
 	We just proved that the maximum size of a commutative subsemigroup of idempotents of $\ptr{X}$ is at most $2^{\abs{X}}$. Furthermore, we know that there exists at least one commutative subsemigroup of idempotents of $\ptr{X}$ of size $2^{\abs{X}}$ --- namely $\idemp{\psym{X}}$. Therefore the maximum size of a commutative subsemigroup of idempotents of $\ptr{X}$ is $2^{\abs{X}}$.
 	
 	\medskip
 	
 	\textbf{Part 2.} The aim of this part is to establish that the only commutative subsemigroup of idempotents of $\ptr{X}$ of order $2^{\abs{X}}$ is $\idemp{\psym{X}}$. Let $S$ be a commutative subsemigroup of idempotents of $\ptr{X}$ such that $\abs{S}=2^{\abs{X}}$. It follows from Proposition~\ref{P(X): proposition P(X) --> T(Y)} that $\newtr{S}$ is a commutative subsemigroup of idempotents of $\tr{\newtr{X}}$ of size $\abs{\newtr{S}}=\abs{S}=2^{\abs{X}}=2^{\abs{\newtr{X}}-1}$. Hence Theorem~\ref{T(X): maximum size comm smg idemp} implies that $\newtr{S}=\commidemp{\newtr{X}}{y}$ for some $y\in \newtr{X}$. This implies that $y$ is the unique element of $\newtr{X}$ such that $y\newtr{\beta}=y$ for all $\beta\in S$. It follows from the fact that $\new\newtr{\beta}=\new$ for all $\beta\in S$ that $y=\new$. Then we have $x\newtr{\beta}\in\set{x,\new}$ for all $\beta\in S$ and $x\in \newtr{X}\setminus\set{\new}=X$ and, consequently, for all $\beta\in S$ and $x\in \newtr{X}\setminus\set{\new}=X$ we have either $x\in \newtr{X}\setminus\dom\beta$, or $x\in\dom\beta$ and $x\beta=x\newtr{\beta}=x=x\mskip 1.5mu\id{\dom\beta}$. Therefore $\beta=\id{\dom\beta}$ and, thus,
 	\begin{displaymath}
 		S\subseteq\gset{\id{W}}{W\subseteq X}=\idemp{\psym{X}}
 	\end{displaymath}
 	and, since $\abs{S}=2^{\abs{X}}=\abs{\idemp{\psym{X}}}$, then we can conclude that $S=\idemp{\psym{X}}$.
 \end{proof}
 
 \section{The largest commutative (full and partial) transformation semigroups with a unique idempotent}\label{sec: largest comm smg 1 idemp T(X)}
 
 Recall that $X$ denotes a finite set. In this section we investigate commutative subsemigroups of $\tr{X}$ (respectively, $\ptr{X}$) that contain exactly one idempotent. Our goal is to find the maximum size of these semigroups and describe the ones that achieve that size. We will see that, when $\abs{X}\leqslant 3$, the largest commutative subsemigroups of $\tr{X}$ that contain exactly one idempotent are groups of size $\abs{X}$; when $\abs{X}=4$, they are either groups or null semigroups and have size $\abs{X}=\maxnull{\abs{X}}$; and, when $\abs{X}\geqslant 5$, they are null semigroups of size $\maxnull{\abs{X}}$. A corollary of this result is that the maximum-order commutative nilpotent subsemigroups of $\tr{X}$ have size $\maxnull{\abs{X}}$ and are all null semigroups (which was proved directly by Cain, Malheiro and the present author in \cite{Commutative_nilpotent_transformation_semigroups_paper}). Moreover, we will see that the maximum size of a commutative subsemigroup of $\ptr{X}$ with a unique idempotent is $\maxnull{\abs{X}+1}$ and that, when $\abs{X}\leqslant 2$, the subsemigroups that achieve that size are either groups or null semigroups and, when $\abs{X}\geqslant 3$, they are all null semigroups. A corollary of this result is that the maximum-order commutative nilpotent subsemigroups of $\ptr{X}$ have size $\maxnull{\abs{X}+1}$ and are all null semigroups.

 We start by characterizing the largest commutative subsemigroups of $\tr{X}$ with a unique idempotent. Proving this result relies on knowledge of the maximum-order abelian subgroups of $\sym{X}$ (Theorem~\ref{T(X): maximum size abelian subgroup of S(X)}) and the maximum-order null subsemigroups of $\tr{X}$ (Theorem~\ref{null semigroups of maximum size}). Below we supply the background information we need to prove the results of this section.
 
 First we introduce the theorem that describes the largest abelian subgroups of $\sym{X}$ in terms of products of cyclic groups. Recall that $C_n=\langle x \mid x^n=1\rangle$ is the \defterm{cyclic group} of order $n$.
 
 \begin{theorem}[{\cite[Burns, Goldsmith]{Symmetric_group}}]\label{T(X): maximum size abelian subgroup of S(X)}
 	Suppose that $\abs{X}\geqslant 2$. Then the maximum size of an abelian subgroup of $\sym{X}$ is
 	\begin{displaymath}
 		\begin{cases}
 			3^k& \text{if } \abs{X}=3k,\\
 			4\cdot 3^{k-1}& \text{if } \abs{X}=3k+1,\\
 			2\cdot 3^k& \text{if } \abs{X}=3k+2.
 		\end{cases} 
 	\end{displaymath}
 	Moreover, the maximum-order abelian subgroups of $\sym{X}$ are isomorphic to
 	\begin{displaymath}
 		\begin{cases}
 			C_3^k& \text{if } \abs{X}=3k,\\
 			C_4\times C_3^{k-1} \text{ or } C_2\times C_2\times C_3^{k-1}& \text{if } \abs{X}=3k+1,\\
 			C_2\times C_3^k& \text{if } \abs{X}=3k+2.
 		\end{cases} 
 	\end{displaymath}
 \end{theorem}
 
 Now we discuss maximum-order null subsemigroups of $\tr{X}$. In \cite{Null_semigroups} Cameron et al. introduced two functions $\xi,\alpha:\mathbb{N}\rightarrow\mathbb{N}$ which, for each $n\in\mathbb{N}$, are defined in the following way
 \begin{displaymath}
 	\maxnull{n}=\max\gset[\big]{t^{n-t}}{t\in\Xn}
 \end{displaymath}
 and
 \begin{displaymath}
 	\alphanull{n}=\max\gset[\big]{t\in\Xn}{t^{n-t}=\maxnull{n}}.
 \end{displaymath}
 For the values of $\maxnull{n}$ and $\alphanull{n}$ for $n\in\X{10}$ see Table~\ref{T(X): table xi alpha}.
 
 \begin{table}[hbt]
 	\centering
 	\begin{tabular}{rrr}
 		\toprule
 		$n$ & $(n)\alpha$ & $(n)\xi$      \\
 		\midrule
 		$1$ & $1$         & $1$           \\
 		$2$ & $2$         & $1$           \\
 		$3$ & $2$         & $2$           \\
 		$4$ & $2$         & $4$           \\
 		$5$ & $3$         & $9$           \\
 		$6$ & $3$         & $27$          \\
 		$7$ & $3$         & $81$          \\
 		$8$ & $4$         & $256$         \\
 		$9$ & $4$         & $1024$        \\
 		$10$ & $4$         & $4096$        \\
 		$11$ & $4$         & $16384$       \\
 		$12$ & $5$         & $78125$       \\
 		$13$ & $5$         & $390625$      \\
 		$14$ & $5$         & $1953125$     \\
 		$15$ & $6$         & $10077696$    \\
 		$16$ & $6$         & $60466176$    \\
 		$17$ & $6$         & $362797056$   \\
 		$18$ & $6$         & $2176782336$  \\
 		$19$ & $7$         & $13841287201$ \\
 		$20$ & $7$         & $96889010407$ \\
 		\bottomrule
 	\end{tabular}
 	\caption{Values of $\maxnull{n}$ and $\alphanull{n}$ for $n\in\X{10}$ \cite{Null_semigroups}.}
 	\label{T(X): table xi alpha}
 \end{table}
 
 The next lemma provides some inequalities satisfied by the function $\xi$ described above.
 
 \begin{lemma}[{\cite[Lemma 2.4]{Null_semigroups}}]\label{inequalities xi}
 	We have $(1)\xi=(2)\xi$ and \allowbreak $(n)\xi<(n+1)\xi$ for all $n\geqslant 2$.
 \end{lemma}
 
 Theorem~\ref{maximum size null semigroup} shows that the size of a largest null subsemigroup of $\tr{X}$ depends on the function $\xi$ and Theorem~\ref{null semigroups of maximum size} uses the function $\alpha$ to characterize all the null subsemigroups of $\tr{X}$ of maximum size.
 
 \begin{theorem} [{\cite[Theorem 4.4]{Null_semigroups}}]\label{maximum size null semigroup}
 	The maximum size of a null subsemigroup of $\tr{X}$ is $(|X|)\xi$.
 \end{theorem}
 
 \begin{theorem} [{\cite[Subsection 4.1]{Null_semigroups}}]\label{null semigroups of maximum size}
 	Let $S$ be a null subsemigroup of $\tr{X}$and let $t=\alphanull{\abs{X}}$. We have that $\abs{S}=\maxnull{\abs{X}}$ if and only if at least one of the following conditions is satisfied:
 	\begin{enumerate}
 		\item There exist pairwise distinct $x_1,\ldots,x_t\in X$ such that
 		\begin{displaymath}
 			S=\gset[\big]{\beta \in \tr{X}}{\set{x_1,\ldots,x_t}\beta=\set{x_1} \text{ and } \im\beta\subseteq\set{x_1,\ldots,x_t}}.
 		\end{displaymath}
 		
 		\item $\abs{X}=2$ and $S=\set{\id{X}}$.
 	\end{enumerate}
 \end{theorem}
 
 We are going to adopt the notation $\nulltr{X}{x_1}{x_t}$ to designate the maximum-order null semigroup described in part 1 of Theorem~\ref{null semigroups of maximum size}; that is,
 \begin{displaymath}
 	\nulltr{X}{x_1}{x_t}=\gset[\big]{\beta \in \tr{X}}{\set{x_1,\ldots,x_t}\beta=\set{x_1} \text{ and } \im\beta\subseteq\set{x_1,\ldots,x_t}},
 \end{displaymath}
 where $t=\alphanull{\abs{X}}$ and $x_1,\ldots,x_t\in X$ are pairwise distinct.

 A certain notation for idempotents of $\tr{X}$ was introduced in \cite{Commuting_graph_T_X}, which we now describe: if $\set{A_i}_{i=1}^n$ is a partition of $X$, where $n\in\mathbb{N}$, and $x_i\in A_i$ for all $i\in\Xn$, then we denote by
 \begin{displaymath}
 	e=\chain{A_1, x_1}\chain{A_2, x_2}\cdots\chain{A_n, x_n}
 \end{displaymath}
 the idempotent such that $A_ie=\set{x_i}$ for all $i\in\Xn$.
 
 We note that all idempotents $e$ of $\tr{X}$ can be written using that notation. In fact, if $\im e=\set{x_1,\ldots,x_n}$ and $A_i=\set{x_i}e^{-1}$ for all $i\in\Xn$, then $\set{A_i}_{i=1}^n$ is a partition of $X$ and, since $e$ is an idempotent, we have $x_i\in \set{x_i}e^{-1}=A_i$ for all $i\in\Xn$. Then we can write $e$ using the notation introduced above.
 
 Below we exhibit a method to identify the transformations that commute with a given idempotent. A different form of this result is also present in \cite[Lemma 2.2]{Characterization_idempotents_1}.
 
 \begin{lemma}[{\cite[Lemma 2.2]{Commuting_graph_T_X}}]\label{T(X): tr commutes with idemp} 
 	Let $e=\chain{A_1, x_1}\chain{A_2, x_2}\cdots\chain{A_n, x_n}$ be an idempotent of $\tr{X}$ and let $\beta\in\tr{X}$. Then $e\beta=\beta e$ if and only if for all $i\in\Xn$ there exists $j\in\Xn$ such that $x_i\beta=x_j$ and $A_i\beta\subseteq A_j$.
 \end{lemma}

 The technique we will use to obtain the maximum size of a commutative subsemigroup of $\tr{X}$ with a unique idempotent, and to identify the semigroups that achieve that size, is based on the one used in \cite[Theorems 3.7 and 3.12]{Commutative_nilpotent_transformation_semigroups_paper} to determine the maximum size of a commutative nilpotent subsemigroup of $\tr{X}$ and the maximum-order commutative nilpotent subsemigroup of $\tr{X}$. This result will come as a corollary of the one we prove in this section.
 
 In \cite[Theorem 3.7]{Commutative_nilpotent_transformation_semigroups_paper} the authors constructed a tree from a commutative nilpotent transformation semigroup whose zero has rank $1$, modified it and proved that the resulting tree was one corresponding to a null semigroup. In this section we will construct, in a similar way, a tree from a commutative transformation semigroup whose unique idempotent is not the identity. The modifications we apply on the tree are more complex than the ones used in \cite[Theorem 3.7]{Commutative_nilpotent_transformation_semigroups_paper}, but at the end we are also capable of obtaining a tree of a null semigroup.
 
 Below we describe how to obtain a special partition of $X$ from a commutative transformation semigroup with a unique idempotent. This partition is an adaptation (and also an extension) of the $S$-partition defined in \cite[Definition 3.3]{Commutative_nilpotent_transformation_semigroups_paper} for commutative nilpotent subsemigroups of $\tr{X}$ whose zero (the unique idempotent) has rank 1. This partition is the starting point for obtaining a tree from a semigroup.
 
 \begin{definition}[$S$-partition]
 	Let $S$ be a commutative subsemigroup of $\tr{X}$ with a unique idempotent $e\in S$. Given a partition $\{A_j\}_{j=0}^k$ of $X$, we say that $\{A_j\}_{j=0}^k$ is an \textit{$S$-partition of $X$} if
 	\begin{align*} &A_0=\im e \\
 		&A_j=\left\{x\in X\setminus \bigcup_{l=0}^{j-1}A_l: x\beta\in\bigcup_{l=0}^{j-1} A_l \textrm{ for all } \beta \in S\right\}, \textrm{ } j=1,\ldots,k.\end{align*}
 \end{definition}
 
 Note that, from construction, given a commutative subsemigroup $S$ of $\tr{X}$ with a unique idempotent there is at most one $S$-partition of $X$. We will prove in Proposition~\ref{T(X): S-partition} below that an $S$-partition always exists, but first we illustrate the definition with an example.
 
 \begin{example} \label{example S-partition}
 	We consider the semigroup $\Tr{7}$ of full transformations over $\{1,2,3,4,5,6,7\}$. Let $S$ be the subsemigroup of $\Tr{7}$ formed by the following transformations:
 	\begin{align*}
 		&\begin{pmatrix} 1&2&3&4&5&6&7 \\ 1&7&4&4&4&4&7\end{pmatrix} && \begin{pmatrix}1&2&3&4&5&6&7 \\ 4&1&7&7&7&7&1\end{pmatrix} &\begin{pmatrix}1&2&3&4&5&6&7 \\ 7&4&1&1&1&1&4\end{pmatrix}\\
 		&\begin{pmatrix} 1&2&3&4&5&6&7 \\ 1&7&4&4&4&3&7\end{pmatrix} &&&\begin{pmatrix}1&2&3&4&5&6&7 \\ 7&3&1&1&1&1&4\end{pmatrix}\\
 		&\begin{pmatrix} 1&2&3&4&5&6&7 \\ 1&7&4&4&4&5&7\end{pmatrix} &&&\begin{pmatrix}1&2&3&4&5&6&7 \\ 7&5&1&1&1&1&4\end{pmatrix}
 	\end{align*}
 	
 	Notice that the top-leftmost transformation is the (unique) idempotent of the semigroup. It is straightforward to verify that $S$ is a commutative semigroup.
 	
 	We are going to determine the $S$-partition of $\set{1,2,3,4,5,6,7}$. The set $A_0$ is equal to the image of the idempotent of $S$, which implies that $A_0=\set{1,4,7}$. The set $A_1$ is formed by all the elements of $\set{1,2,3,4,5,6,7}\setminus A_0=\set{2,3,5,6}$ whose image, in all the transformations of $S$, belongs to $A_0$, that is, whose image is either $1$, $4$ or $7$. Those elements are precisely $3$ and $5$. Hence $A_1=\set{3,5}$. The set $A_2$ is formed by all the elements of $\set{1,2,3,4,5,6,7}\setminus\parens{A_0\cup A_1}$ whose image, in all the transformations of $S$, belongs to $A_0\cup A_1=\set{1,3,4,5,7}$. The image of $4$ and $6$ in the transformations of $S$ always belongs to $A_0\cup A_1=\set{1,3,4,5,7}$, and so $A_2$ comprises the remaining elements of $X$. Since $A_0\cup A_1\cup A_2=\set{1,2,3,4,5,6,7}$, then $\set{A_j}_{j=0}^2$ is the $S$-partition of $\set{1,2,3,4,5,6,7}$.
 \end{example}
 
 In order to prove that it is always possible to construct an $S$-partition from a commutative transformation semigroup $S$ with a unique idempotent, we first need to introduce the following lemma.
 
 \begin{lemma}\label{T(X): exists x not in union of images}
 	Let $S$ be a commutative subsemigroup of $\tr{X}$ with a unique idempotent. If $S\nsubseteq\sym{X}$, then $\bigcup_{\beta\in S} \im \beta \subsetneq X$.
 \end{lemma}
 
 \begin{proof}
 	Let $e$ be the unique idempotent of $S$. Since $S\nsubseteq\sym{X}$, then there exists $\alpha\in S$ such that $\alpha\notin\sym{X}$. Then $\im\alpha\subsetneq X$. Due to the fact that $S$ is finite, we have that there exists $m\in\mathbb{N}$ such that $\alpha^m$ is an idempotent. Since $e$ is the unique idempotent of $S$, then $\alpha^m=e$. Therefore $\im e=\im\alpha^m\subseteq\im\alpha\subsetneq X$.
 	
 	Suppose, with the aim of obtaining a contradiction, that $\bigcup_{\beta\in S} \im \beta = X$. 
 	
 	We have that $X\setminus\im e\neq \emptyset$. Let $x_1 \in X\setminus\im e$. There exist $\beta_1 \in S$ and $x_2 \in X$ such that $x_1=x_2\beta_1$. Since $e\beta_1=\beta_1 e$ and $x_2\beta_1=x_1\in X\setminus\im e$, then Lemma~\ref{T(X): ab=ba => xb€Im a for all x€Im a} implies that $x_2\in X\setminus\im e$. Continuing in this way, construct a sequence $(x_n)_{n\in\mathbb{N}}$ of elements of $X\setminus\im e$ and a sequence $(\beta_n)_{n\in\mathbb{N}}$ of elements of $S$ that satisfy $x_n=x_{n+1}\beta_n$ for all $n\in\mathbb{N}$. Since $X$ is finite, then there exist $i<j$ such that $x_i=x_j$ and there exists $k\in\mathbb{N}$ such that $(\beta_{j-1}\cdots\beta_{i})^k$ is an idempotent. Hence $(\beta_{j-1}\cdots\beta_{i})^k=e$. Moreover, we have that
 	\begin{displaymath}
 		x_i=x_{i+1}\beta_i=x_{i+2}\beta_{i+1}\beta_i=\ldots=x_j\beta_{j-1}\cdots\beta_{i}=x_i\beta_{j-1}\cdots\beta_{i}.
 	\end{displaymath}
 	Consequently,
 	\begin{displaymath}
 		x_i=x_i\beta_{j-1}\cdots\beta_{i}=x_i(\beta_{j-1}\cdots\beta_{i})^2=\ldots=x_i(\beta_{j-1}\cdots\beta_{i})^k=x_ie\in\im e,
 	\end{displaymath}
 	which is a contradiction.
 	
 	Therefore $\bigcup_{\beta\in S} \im \beta \subsetneq X$.
 \end{proof}

 \begin{proposition}\label{T(X): S-partition}
 	Let $S$ be a commutative subsemigroup of $\tr{X}$ with a unique idempotent. Then there exists an $S$-partition of $X$. 
 \end{proposition}
 
 \begin{proof}
 	Let $e\in S$ be the unique idempotent of $S$. 
 	
 	We are going to prove the result by induction on the size of $X$.
 	
 	Suppose that $\abs{X}=1$. Then $S=\tr{X}=\set{e}$ and $X=\im e$. Thus $\set{\im e}$ is an $S$-partition of $X$.
 	
 	Suppose that $\abs{X}\geqslant 2$ and assume that the result is valid for any set of size $\abs{X}-1$.
 	
 	If $S\subseteq\sym{X}$, then the unique idempotent of $S$ is $\id{X}$. Hence $e=\id{X}$ and, consequently, $\im e=\im\id{X}=X$. Thus $\set{\im e}$ is an $S$-partition of $X$.
 	
 	Now assume that $S\nsubseteq\sym{X}$. Then, by Lemma~\ref{T(X): exists x not in union of images}, there exists $t\in X\setminus\bigcup_{\beta\in S}\im \beta$, which implies that $\beta|_{X\setminus\set{t}}\in\tr{X\setminus\set{t}}$ for all $\beta\in S$. It follows from Lemma~\ref{T(X): lemma induction} that $S'=\gset{\beta|_{X\setminus\set{t}}}{\beta\in S}$ is a commutative subsemigroup of $\tr{X\setminus\set{t}}$ whose unique idempotent is $e|_{X\setminus\set{t}}$. Therefore, by the induction hypothesis, $X\setminus\set{t}$ admits an $S'$-partition $\set{A_j}_{j=0}^{k}$, where $A_0=\im e|_{X\setminus\set{t}}$ and, for all $j\in\{1,\ldots,k\}$,
 	\begin{align*}
 		A_j&=\gset[\Bigg]{x\in (X\setminus \set{t})\setminus \bigcup_{l=0}^{j-1}A_l}{x\beta\in\bigcup_{l=0}^{j-1} A_l \textrm{ for all } \beta \in S'}\\
 		&=\gset[\Bigg]{x\in (X\setminus \set{t})\setminus \bigcup_{l=0}^{j-1}A_l}{x\beta|_{X\setminus\set{t}}\in\bigcup_{l=0}^{j-1} A_l \textrm{ for all } \beta \in S}\\
 		&=\gset[\Bigg]{x\in (X\setminus \{t\})\setminus \bigcup_{l=0}^{j-1}A_l}{x\beta\in\bigcup_{l=0}^{j-1} A_l \textrm{ for all } \beta \in S}.
 	\end{align*}
 	
 	We observe that we have $\im e\setminus\set{te}\subseteq\im e|_{X\setminus\set{t}}\subseteq\im e$. We are going to see that $te\in\im e|_{X\setminus\set{t}}$, which implies that $\im e|_{X\setminus\set{t}}=\im e$. Since $t\in X\setminus\bigcup_{\beta\in S}\im \beta\subseteq X\setminus\im e$, then $te\in \parens{X\setminus\set{t}}\cap\im e$. Hence $te=\parens{te}e\in \parens{X\setminus\set{t}}e=\im e|_{X\setminus\set{t}}$ and, consequently, we must have $\im e|_{X\setminus\set{t}}=\im e$. Thus $A_0=\im e|_{X\setminus\set{t}}=\im e$.

 	
 	From the definition of $t$, we have $t\in X\setminus\im e$ and $t\beta\in X\setminus\set{t}=\bigcup_{j=0}^k A_j$ for all $\beta\in S$. Let
 	\begin{displaymath}
 		r=\min\gset[\Bigg]{j\in\set{1,\ldots,k+1}}{t\beta\in\bigcup_{l=0}^{j-1}A_l \textrm{ for all } \beta \in S}.
 	\end{displaymath} 
 	
 	We want to construct an $S$-partition of $X$ from the $S'$-partition $\{A_j\}_{j=0}^k$ of $X\setminus\{t\}$. We will either create a new set $A_{k+1}$ formed exclusively by $t$, or add $t$ to one of the existing sets of $\{A_j\}_{j=0}^k$. The way we extend the partition of $X\setminus\{t\}$ depends on the value of $r$ defined above and is chosen so that the new partition is an $S$-partition of $X$. We consider two cases.
 	
 	\smallskip
 	
 	\textit{Case 1:} Suppose that $r=k+1$. This implies that there exists $\beta \in S$ such that $t\beta \notin \bigcup_{l=0}^{k-1}A_l$. Consequently, 
 	\begin{displaymath}
 		A_j=\gset[\Bigg]{x\in X\setminus \bigcup_{l=0}^{j-1}A_l}{x\beta\in\bigcup_{l=0}^{j-1} A_l \textrm{ for all } \beta \in S}
 	\end{displaymath}
 	for all $j\in\X{k}$. Let $A_{k+1}=\set{t}$. Then
 	\begin{displaymath}
 		A_{k+1}=\gset[\Bigg]{x\in X\setminus \bigcup_{l=0}^{k}A_l}{x\beta\in\bigcup_{l=0}^{k} A_l \textrm{ for all } \beta \in S}
 	\end{displaymath}
 	and $\set{A_j}_{j=0}^{k+1}$ is an $S$-partition of $X$.
 	
 	\smallskip
 	
 	\textit{Case 2:} Suppose that $r\leqslant k$. Let $B_r=A_r\cup\{t\}$ and $B_j=A_j$ for all $j\in\set{0,\ldots,k}\setminus\set{r}$. We also have $t\beta \notin \bigcup_{l=0}^{r-2}B_l$ for some $\beta \in S$. Then
 	\begin{displaymath}
 		B_j=\gset[\Bigg]{x\in X\setminus \bigcup_{l=0}^{j-1}B_l}{x\beta\in\bigcup_{l=0}^{j-1} B_l \textrm{ for all } \beta \in S}
 	\end{displaymath}
 	for all $j\in\X{k}$. Thus $\set{B_j}_{j=0}^{k}$ is an $S$-partition of $X$.
 \end{proof}
 
 Below we describe how to construct a labelled tree from a commutative subsemigroup of $\tr{X}$ with a unique idempotent. For an illustration of how to construct a tree from a specific semigroup see Example~\ref{T(X): example tree}.
 
 \begin{definition}[Tree of a semigroup]\label{T(X): definition tree}
 	Let $S$ be a commutative subsemigroup of $\tr{X}$ with a unique idempotent and assume that $n=\abs{X}$. Let $\{A_j\}_{j=0}^{k}$ be the $S$-partition of $X$.
 	
 	In order to obtain a labelled tree from $S$, we need to order the elements of $X$ in a convenient way and then use that order to associate each transformation of $S$ to a word of length $n$ over $X$. The tree of $S$ is constructed from those words.
 	
 	We reorder the elements of $X$ in a way such that the elements of $A_j$ appear before the elements of $A_{j+1}$ for all $j\in\{0,\ldots,k-1\}$. Assume that, after reordering, the elements of $X$ are sequenced in the following way: $x_1,\ldots,x_n$. Each transformation $\beta\in S$ determines the word $w_\beta$ of length $n$ over $X$ whose $i$-th letter is $x_i\beta$. Let $W_S=\{w_{\beta}: \beta\in S\}$ be the set of words determined by (the transformations of) $S$, whose size is $\abs{S}$.
 	
 	The \defterm{tree of $S$}, denoted by $T_S$, is a labelled tree whose vertex set is the set of prefixes of the words belonging to $W_S$, that is, the set of vertices is $\gset{u\in X^*}{uv \in W_S \textrm{ for some } v\in X^*}$. Each arc of the tree is labelled with a letter from the alphabet $X$ and, given two vertices $u$ and $v$, we have an arc from $u$ to $v$ labelled by the letter $x$ if and only if $ux=v$.
 \end{definition}
 
 The result below provides some basic properties regarding trees of semigroups.
 
 \begin{lemma}\label{T(X): properties tree of S}
 	Let $S$ be a commutative subsemigroup of $\tr{X}$ with a unique idempotent and assume that $n=\abs{X}$. Let $x_1,\ldots,x_n$ be the order of the elements of $X$ used to construct $T_S$. Then
 	\begin{enumerate}
 		\item The vertex $\varepsilon$ is the root of $T_S$.
 		
 		\item The number of leaves of $T_S$ is $\abs{S}$.
 	\end{enumerate}
 \end{lemma}
 
 \begin{proof}
 	It follows from the way we defined $T_S$ that the vertex $\varepsilon$ is the only vertex of $T_S$ whose indegree is zero. Thus vertex $\varepsilon$ is the root of $T_S$.
 	
 	Additionally, we have that
 	\begin{align*}
 		u \text{ is a leaf of } T_S &\iff \text{the outdegree of } u \text{ is zero}\\
 		&\iff u\in W_S,
 	\end{align*}
 	where $W_S$ corresponds to the set of words over $X$ determined by (the transformations of) $S$. Hence the number of leaves of $T_S$ is $\abs{W_S}=\abs{S}$.
 \end{proof}

 Next we define some terms that will be used frequently when we use trees of semigroups.
 
 \begin{definition}
 	Let $S$ be a commutative subsemigroup of $\tr{X}$ with a unique idempotent and assume that $n=\abs{X}$. Let $x_1,\ldots,x_n$ be the order of the elements of $X$ used to construct the tree of $S$.
 	\begin{enumerate}
 		\item If a vertex of $T_S$ has outdegree at least $2$, that is, if a vertex of $T_S$ has at least two arcs starting in it, then we say that a \defterm{branching} occurs. If $s$ is the outdegree of the vertex, then we say that we have a branching with $s$ arcs.
 		
 		\item Let $i\in\Xn$. We say that the arcs of $T_S$ whose starting vertex is a word of length $i-1$ and ending vertex is a word of length $i$ form the \defterm{level $x_i$} of $T_S$.
 		
 		\item Let $i\in\Xn$. We say that the level $x_i$ is a \defterm{branching level} if there is at least one branching at level $x_i$ (that is, if there exists a vertex that is a word of length $i-1$ whose outdegree is at least $2$).
 		
 		\item Let $i\in\Xn$. We say that the level $x_i$ is a \defterm{linear level} if no branching occurs at level $x_i$ (that is, if all the vertices that are words of length $i-1$ have outdegree $1$).
 		
 		\item If level $x_1$ is linear, then the beginning of the tree is a single path whose length is equal to the number of linear levels occurring at the beginning of the tree (that is, the number of linear levels that appear before the first branching level). We call that path the \defterm{trunk of $T_S$}.
 	\end{enumerate}
 \end{definition}
 
 \begin{example}\label{T(X): example tree}
 	The aim of this example is to construct $T_S$, where $S$ is the semigroup from Example~\ref{example S-partition}. We saw that $\set{A_j}_{j=0}^2$ is the $S$-partition of $\set{1,2,3,4,5,6,7}$, where $A_0=\set{1,4,7}$, $A_1=\set{3,5}$ and $A_2=\set{2,6}$.
 	
 	We want to choose a sequence of the elements of $\set{1,2,3,4,5,6,7}$ such that the elements of $A_0$ are the first to appear in that sequence (in any order), followed by the elements of $A_1$ (in any order) and the last elements are the ones belonging to $A_3$ (in any order). A possible way of ordering the elements is $1,4,7,3,5,2,6$.
 	
 	In order to facilitate obtaining words from the transformations of $S$, we are going to write the transformations of $S$ in a different form so that, in their first row, the elements of $\set{1,2,3,4,5,6,7}$ appear in the order $1,4,7,3,5,2,6$.
 	\begin{align*}
 		&\begin{pmatrix} 1&4&7&3&5&2&6 \\ 1&4&7&4&4&7&4\end{pmatrix} &&\begin{pmatrix}1&4&7&3&5&2&6 \\ 4&7&1&7&7&1&7\end{pmatrix} &\begin{pmatrix}1&4&7&3&5&2&6 \\ 7&1&4&1&1&4&1\end{pmatrix}\\
 		&\begin{pmatrix} 1&4&7&3&5&2&6 \\ 1&4&7&4&4&7&3\end{pmatrix} &&&\begin{pmatrix}1&4&7&3&5&2&6 \\ 7&1&4&1&1&3&1\end{pmatrix}\\
 		&\begin{pmatrix} 1&4&7&3&5&2&6 \\ 1&4&7&4&4&7&5\end{pmatrix} &&&\begin{pmatrix}1&4&7&3&5&2&6 \\ 7&1&4&1&1&5&1\end{pmatrix}
 	\end{align*}
 	
 	The words we construct from the transformations of $S$ (using the order $1,4,7,3,5,2,6$) can be obtained simply by reading the second row (from left to right) of the transformations written above. The set of words we get is
 	\begin{align*}
 		W_S &=\set{1474474, 1474473, 1474475, 4717717, 7141141, 7141131, 7141151}\\
 		&\subseteq \set{1,2,3,4,5,6,7}^*,
 	\end{align*}
 	which allows us to construct the tree $T_S$, which is represented in Figure~\ref{T(X): figure, tree of S}.
 	
 	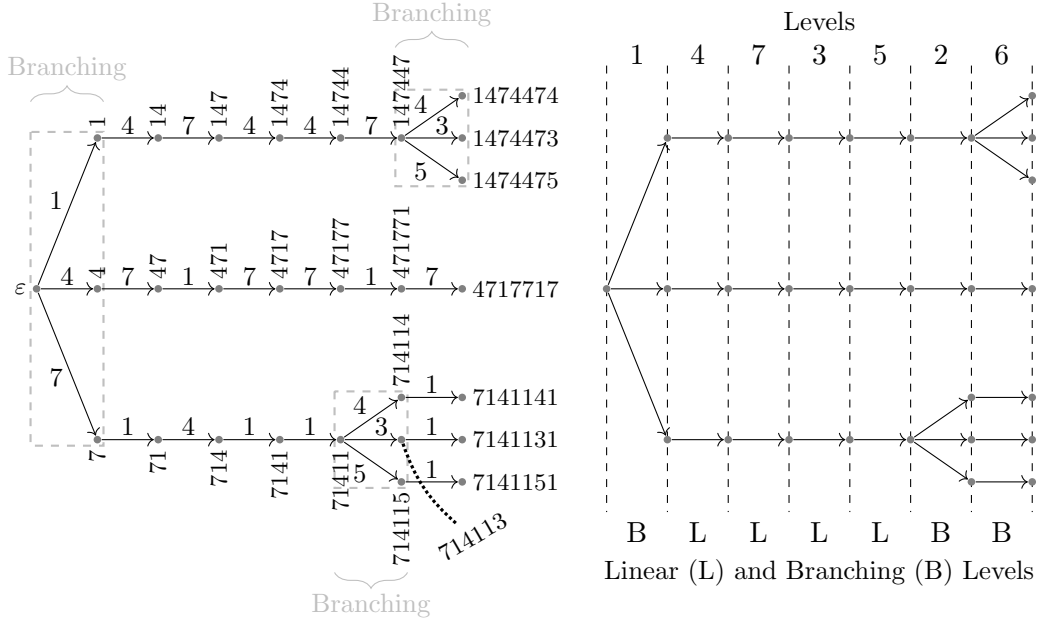
\begin{figure}[hbt]
 		\begin{center}
 			
 			\begin{tikzpicture}[x=8mm,y=8mm]
 				
 				\draw[color=lightgray, dashed, thick](-0.1,2.6) rectangle (1.1,-2.6);
 				\draw[color=lightgray, dashed, thick](5.9,3.3) rectangle (7.1,1.7);
 				\draw[color=lightgray, dashed, thick](4.9,-1.7) rectangle (6.1,-3.3);
 				
 				\draw[decorate, decoration={brace, amplitude=6pt}, color=lightgray] (-0.1,3.1) -- (1.1,3.1);
 				\node at (0.5,3.65) {\small \textcolor{lightgray}{Branching}};
 				\draw[decorate, decoration={brace, amplitude=6pt}, color=lightgray] (5.9,4) -- (7.1,4);
 				\node at (6.5,4.55) {\small \textcolor{lightgray}{Branching}};
 				\draw[decorate, decoration={brace, amplitude=6pt, mirror}, color=lightgray] (4.9,-4.7) -- (6.1,-4.7);
 				\node at (5.5,-5.25) {\small \textcolor{lightgray}{Branching}};
 				
 				\begin{scope}[
 					every node/.style={treenode},
 					]
 					\node (root) at (0,0) {};
 					\node (1) at (1,2.5) {};
 					\node (14) at (2,2.5) {};
 					\node (147) at (3,2.5) {};
 					\node (1474) at (4,2.5) {};
 					\node (14744) at (5,2.5) {};
 					\node (147447) at (6,2.5) {};
 					\node (1474474) at (7,3.2) {};
 					\node (1474473) at (7,2.5) {};
 					\node (1474475) at (7,1.8) {};
 					\node (4) at (1,0) {};
 					\node (47) at (2,0) {};
 					\node (471) at (3,0) {};
 					\node (4717) at (4,0) {};
 					\node (47177) at (5,0) {};
 					\node (471771) at (6,0) {};
 					\node (4717717) at (7,0) {};
 					\node (7) at (1,-2.5) {};
 					\node (71) at (2,-2.5) {};
 					\node (714) at (3,-2.5) {};
 					\node (7141) at (4,-2.5) {};
 					\node (71411) at (5,-2.5) {};
 					\node (714114) at (6,-1.8) {};
 					\node (714113) at (6,-2.5) {};
 					\node (714115) at (6,-3.2) {};
 					\node (7141141) at (7,-1.8) {};
 					\node (7141131) at (7,-2.5) {};
 					\node (7141151) at (7,-3.2) {};
 				\end{scope}
 				
 				\node[anchor=east] at (root) {\footnotesize $\varepsilon$};
 				\node[anchor=west, rotate=90] at (1) {\footnotesize $1$};
 				\node[anchor=west, rotate=90] at (14) {\footnotesize $14$};
 				\node[anchor=west, rotate=90] at (147) {\footnotesize $147$};
 				\node[anchor=west, rotate=90] at (1474) {\footnotesize $1474$};
 				\node[anchor=west, rotate=90] at (14744) {\footnotesize $14744$};
 				\node[anchor=west, rotate=90] at (147447) {\footnotesize $147447$};
 				\node[anchor=west] at (1474474) {\footnotesize $1474474$};
 				\node[anchor=west] at (1474473) {\footnotesize $1474473$};
 				\node[anchor=west] at (1474475) {\footnotesize $1474475$};
 				\node[anchor=west, rotate=90] at (4) {\footnotesize $4$};
 				\node[anchor=west, rotate=90] at (47) {\footnotesize $47$};
 				\node[anchor=west, rotate=90] at (471) {\footnotesize $471$};
 				\node[anchor=west, rotate=90] at (4717) {\footnotesize $4717$};
 				\node[anchor=west, rotate=90] at (47177) {\footnotesize $47177$};
 				\node[anchor=west, rotate=90] at (471771) {\footnotesize $471771$};
 				\node[anchor=west] at (4717717) {\footnotesize $4717717$};
 				\node[anchor=east, rotate=90] at (7) {\footnotesize $7$};
 				\node[anchor=east, rotate=90] at (71) {\footnotesize $71$};
 				\node[anchor=east, rotate=90] at (714) {\footnotesize $714$};
 				\node[anchor=east, rotate=90] at (7141) {\footnotesize $7141$};
 				\node[anchor=east, rotate=90] at (71411) {\footnotesize $71411$};
 				\node[anchor=west, rotate=90] at (714114) {\footnotesize $714114$};
 				\node[anchor=east, rotate=90] at (714115) {\footnotesize $714115$};
 				\node[anchor=west] at (7141141) {\footnotesize $7141141$};
 				\node[anchor=west] at (7141131) {\footnotesize $7141131$};
 				\node[anchor=west] at (7141151) {\footnotesize $7141151$};
 				
 				\node[anchor=west, rotate=30] (label) at (6.5,-4.5) {\footnotesize $714113$};
 				\draw[very thick, densely dotted] (714113) edge[bend right=15] (label);

 				\begin{scope}[
 					->,
 					every node/.style={edgelabel},
 					]
 					\draw (root) -- node[anchor = south east] {$1$} (1);
 					\draw (1) -- node[above] {$4$} (14);
 					\draw (14) -- node[above] {$7$} (147);
 					\draw (147) -- node[above] {$4$} (1474);
 					\draw (1474) -- node[above] {$4$} (14744);
 					\draw (14744) -- node[above] {$7$} (147447);
 					\draw (147447) -- node[anchor=south east] {$4$} (1474474);
 					\draw (147447) -- node[pos=.7, above] {$3$} (1474473);
 					\draw (147447) -- node[anchor=north east] {$5$} (1474475);
 					\draw (root) -- node[above] {$4$} (4);
 					\draw (4) -- node[above] {$7$} (47);
 					\draw (47) -- node[above] {$1$} (471);
 					\draw (471) -- node[above] {$7$} (4717);
 					\draw (4717) -- node[above] {$7$} (47177);
 					\draw (47177) -- node[above] {$1$} (471771);
 					\draw (471771) -- node[above] {$7$} (4717717);
 					\draw (root) -- node[anchor=north east] {$7$} (7);
 					\draw (7) -- node[above] {$1$} (71);
 					\draw (71) -- node[above] {$4$} (714);
 					\draw (714) -- node[above] {$1$} (7141);
 					\draw (7141) -- node[above] {$1$} (71411);
 					\draw (71411) -- node[anchor=south east] {$4$} (714114);
 					\draw (71411) -- node[pos=.7, above] {$3$} (714113);
 					\draw (71411) -- node[anchor=north east] {$5$} (714115);
 					\draw (714114) -- node[above] {$1$} (7141141);
 					\draw (714113) -- node[above] {$1$} (7141131);
 					\draw (714115) -- node[above] {$1$} (7141151);
 				\end{scope}
 				
 				\begin{scope}[xshift=7.5cm]
 					\foreach \x in {0,1,2,3,4,5,6,7} {
 						\draw[black,dashed] (\x,3.7) --  (\x,-3.7);
 					}
 					\foreach \x/\t in {0/B,1/L,2/L,3/L,4/L,5/B,6/B} {
 						\node[columnlabel] at ({\x+.5},-3.7) {\t};
 					}
 					\foreach \x/\t in {0/1,1/4,2/7,3/3,4/5,5/2,6/6} {
 						\node[columnlabel] at ({\x+.5},4.2) {\t};
 					}
 					\node at (3.5,4.45) {\small Levels};
 					\node at (3.5,-4.7) {\small Linear (L) and Branching (B) Levels};
 					
 					\begin{scope}[
 						every node/.style={treenode},
 						]
 						\node (root) at (0,0) {};
 						\node (1) at (1,2.5) {};
 						\node (14) at (2,2.5) {};
 						\node (147) at (3,2.5) {};
 						\node (1474) at (4,2.5) {};
 						\node (14744) at (5,2.5) {};
 						\node (147447) at (6,2.5) {};
 						\node (1474474) at (7,3.2) {};
 						\node (1474473) at (7,2.5) {};
 						\node (1474475) at (7,1.8) {};
 						\node (4) at (1,0) {};
 						\node (47) at (2,0) {};
 						\node (471) at (3,0) {};
 						\node (4717) at (4,0) {};
 						\node (47177) at (5,0) {};
 						\node (471771) at (6,0) {};
 						\node (4717717) at (7,0) {};
 						\node (7) at (1,-2.5) {};
 						\node (71) at (2,-2.5) {};
 						\node (714) at (3,-2.5) {};
 						\node (7141) at (4,-2.5) {};
 						\node (71411) at (5,-2.5) {};
 						\node (714114) at (6,-1.8) {};
 						\node (714113) at (6,-2.5) {};
 						\node (714115) at (6,-3.2) {};
 						\node (7141141) at (7,-1.8) {};
 						\node (7141131) at (7,-2.5) {};
 						\node (7141151) at (7,-3.2) {};
 					\end{scope}
 					
 					\begin{scope}[
 						->,
 						every node/.style={edgelabel},
 						]
 						\draw (root) -- (1);
 						\draw (1) -- (14);
 						\draw (14) -- (147);
 						\draw (147) -- (1474);
 						\draw (1474) -- (14744);
 						\draw (14744) -- (147447);
 						\draw (147447) -- (1474474);
 						\draw (147447) -- (1474473);
 						\draw (147447) -- (1474475);
 						\draw (root) -- (4);
 						\draw (4) -- (47);
 						\draw (47) -- (471);
 						\draw (471) -- (4717);
 						\draw (4717) -- (47177);
 						\draw (47177) -- (471771);
 						\draw (471771) -- (4717717);
 						\draw (root) -- (7);
 						\draw (7) -- (71);
 						\draw (71) -- (714);
 						\draw (714) -- (7141);
 						\draw (7141) -- (71411);
 						\draw (71411) -- (714114);
 						\draw (71411) -- (714113);
 						\draw (71411) -- (714115);
 						\draw (714114) -- (7141141);
 						\draw (714113) -- (7141131);
 						\draw (714115) -- (7141151);
 					\end{scope}
 				\end{scope}
 			\end{tikzpicture}
 			\caption[Tree of $S$.]{\emph{Tree of $S$.} The image on the left is the tree of $S$ with the arcs and vertices labelled. The branchings of $T_S$ are also identified. The branchings of the tree are the ones associated with the vertices $\varepsilon$, $147447$ and $71411$, and are marked by dashed rectangles. Moreover, since level $1$ (the first level of $T_S$) is a branching level, then $T_S$ contains no trunk. In the image on the right, the levels of $T_S$ are indicated at the top of the tree, and at the bottom are distinguished the linear and branching levels of the tree.}
 			\label{T(X): figure, tree of S}
 		\end{center}
 	\end{figure}
 \end{example}
 
 Next we introduce a lemma which adds information about the labels of the arcs of the tree of a semigroup. Moreover, part 2 of the lemma shows what the presence of a branching implies for the semigroup.
 
 \begin{lemma}\label{T(X): labels of T_S}
 	Let $S$ be a commutative subsemigroup of $\tr{X}$ with a unique idempotent and assume that $n=\abs{X}$. Let $x_1,\ldots,x_n$ be the order of the elements of $X$ used to construct $T_S$. Let $i\in\Xn$. Then
 	\begin{enumerate}
 		\item If $x$ is the label of some arc of level $x_i$, then there exists $\beta\in S$ such that $x=x_i\beta$. Furthermore, if $i\leqslant\abs{\im e}$, then $x\in \set{x_1,\ldots,x_{\abs{\im e}}}$ and if $i>\abs{\im e}$, then $x\in\set{x_1,\ldots,x_{i-1}}$.
 		
 		\item If level $x_i$ contains a branching with $s$ arcs whose labels are $x_{i_1},\ldots,x_{i_s}$, then there exist $\beta_1,\ldots,\beta_s\in S$ such that $\beta_1,\ldots,\beta_s$ are equal in $\set{x_1,\ldots,x_{i-1}}$ and $x_{i_m}=x_i\beta_m$ for all $m\in\X{s}$.
 	\end{enumerate}
 \end{lemma}
 
 \begin{proof}
 	\textbf{Part 1.} Let $x$ be the label of an arc of level $x_i$. Then there exists $\beta \in S$ such that $x$ corresponds to the $i$-th letter of the word $w_\beta$ determined by $\beta$, which is equal to $x_i\beta$. Hence $x=x_i\beta$.
 	
 	Let $\set{A_j}_{j=0}^k$ be the $S$-partition of $X$. We observe that the order $x_1,\ldots,x_n$ was obtained by making the elements of $A_j$ appear before the elements of $A_{j+1}$ for all $j\in\X{k-1}$. Consequently, in the sequence $x_1,\ldots,x_n$, the first $\abs{A_0}=\abs{\im e}$ elements are precisely the elements of $A_0=\im e$. Thus $\im e=\set{x_1,\ldots,x_{\abs{\im e}}}$.
 	
 	\smallskip
 	
 	\textit{Case 1:} Assume that $i\leqslant\abs{\im e}$. Then $x_i\in\im e$. We have that $e$ is an idempotent and $\beta e=e\beta$. Hence Lemma~\ref{T(X): tr commutes with idemp} implies that $x=x_i\beta\in\im e=\set{x_1,\ldots,x_{\abs{\im e}}}$.
 	
 	\smallskip
 	
 	\textit{Case 2:} Assume that $i>\abs{\im e}$. Then $x_i\in X\setminus\im e=X\setminus A_0=\bigcup_{j=1}^k A_j$. Let $m\in\X{k}$ be such that $x_i\in A_m$. Then $x=x_i\beta\in\bigcup_{j=0}^{m-1} A_j$ and, consequently, $x$ precedes $x_i$ in the sequence $x_1,\ldots,x_n$. Thus $x\in\set{x_1,\ldots,x_{i-1}}$.
 	
 	\medskip
 	
 	\textbf{Part 2.} Suppose that there exists a branching with $s\geqslant 2$ arcs at level $x_i$ whose labels are $x_{i_1},\ldots,x_{i_s}$.
 	
 	Let $w$ be the starting vertex (which is a word of length $i-1$) of the $s$ arcs that form the branching. Then the ending vertices of those $s$ arcs are $wx_{i_1},\ldots,wx_{i_s}$. Hence, for each $m\in\X{s}$, there exists $\beta_m\in S$ such that $wx_{i_m}$ is a prefix of the word $w_{\beta_m}$ determined by $\beta_m$.
 	
 	For each $m\in\X{s}$ and $j\in\X{i-1}$ we have
 	\begin{displaymath}
 		x_j\beta_m = j \text{-th letter of } w_{\beta_m} = j \text{-th letter of } wx_{i_m} = j \text{-th letter of } w.
 	\end{displaymath}
 	This implies that $x_j\beta_1=\cdots=x_j\beta_s$ for all $j\in\X{i-1}$, that is, $\beta_1,\ldots,\beta_s$ are equal in $\set{x_1,\ldots,x_{i-1}}$. Additionally, for all $m\in\X{s}$ we have
 	\begin{displaymath}
 		x_i\beta_m = i \text{-th letter of } w_{\beta_m} = i \text{-th letter of } wx_{i_m} = x_{i_m}. \qedhere
 	\end{displaymath}
 \end{proof}
 
 Now that we know how to construct a tree from a commutative transformation semigroup with a unique idempotent, we need some results that will allow us to modify this tree and obtain a new one, which will be the tree of a null semigroup.
 
 The first result will allow us to show later that the tree of a commutative transformation semigroup with one idempotent contains a subgraph which is the tree of a group.
 
 \begin{proposition}\label{T(X): unique idemp is id_X => group in S(X)}
 	Let $S$ be a subsemigroup of $\tr{X}$ with a unique idempotent. If that idempotent is $\id{X}$, then $S$ is a subgroup of $\sym{X}$.
 \end{proposition}
 
 \begin{proof}
 	Suppose that the unique idempotent of $S$ is $\id{X}$.
 	
 	If $S=\set{\id{X}}$, then $S\subseteq\sym{X}$ and $S$ is a group.
 	
 	Now assume that $S\neq\set{\id{X}}$. Let $\alpha\in S\setminus\set{\id{X}}$. Since ($\tr{X}$ and, consequently,) $S$ is finite, then there exists $m\in\mathbb{N}$ such that $\alpha^m$ is an idempotent. Hence $\alpha^m=\id{X}$ (because $\id{X}$ is the unique idempotent of $S$). Consequently, $X=\im\id{X}=\im\alpha^m\subseteq\im\alpha\subseteq X$, which implies that $\alpha\in\sym{X}$. In addition, we have that $m\geqslant 2$ (because $\alpha\neq\id{X}$) and, thus, $\alpha\parens{\alpha^{m-1}}=\id{X}=\parens{\alpha^{m-1}}\alpha$. Therefore $\alpha^{-1}=\alpha^{m-1}\in S$.
 	
 	Since $\alpha$ is an arbitrary element of $S\setminus\set{\id{X}}$, we can conclude that $S$ is a subgroup of $\sym{X}$.
 \end{proof}
 
 The following lemma will allow us to perform the first modification of the tree of a semigroup.
 
 \begin{lemma}\label{T(X): unique idemp, upper bound abelian subgroup S(X)}
 	Let $G$ be an abelian subgroup of $\sym{X}$. Then $\abs{G}\leqslant\maxnull{\abs{X}+1}$. Moreover, if $\abs{X}\geqslant 5$, then $\abs{G}< \maxnull{\abs{X}}$.
 \end{lemma}
 
 \begin{proof}
 	Suppose that $\abs{X}\leqslant 4$. Then, by Theorem~\ref{T(X): maximum size abelian subgroup of S(X)}, we have $\abs{G} \leqslant \abs{X}$. We consider four cases.
 	
 	\smallskip
 	
 	\textit{Case 1:} Assume that $\abs{X}=1$. Then $\abs{G}=1=\maxnull{2}=\maxnull{\abs{X}+1}$.
 	
 	\smallskip
 	
 	\textit{Case 2:} Assume that $\abs{X}=2$. Then $\abs{G}\leqslant 2=\maxnull{3}=\maxnull{\abs{X}+1}$.
 	
 	\smallskip
 	
 	\textit{Case 3:} Assume that $\abs{X}=3$. Then $\abs{G}\leqslant 3<4=\maxnull{4}=\maxnull{\abs{X}+1}$.
 	
 	\smallskip
 	
 	\textit{Case 4:} Assume that $\abs{X}=4$. Then $\abs{G}\leqslant 4<9=\maxnull{5}=\maxnull{\abs{X}+1}$.
 	
 	\smallskip
 	
 	Now suppose that $\abs{X}\geqslant 5$. We divide the proof into three cases.
 	
 	\smallskip
 	
 	
 	
 	
 	\textit{Case 1:} Assume that $\abs{X}=3k$ for some $k\in\mathbb{Z}$. We have $k\geqslant 2$ (since $\abs{X}\geqslant 5$). This implies that $2k-3\geqslant 1$ and, consequently, Theorem~\ref{T(X): maximum size abelian subgroup of S(X)} implies that
 	\begin{displaymath}
 		\abs{G}\leqslant 3^k<3^k\cdot 3^{2k-3} =3^{3k-3} =3^{\abs{X}-3}\leqslant \parens{\abs{X}}\xi. 
 	\end{displaymath}
 	
 	\smallskip
 	
 	\textit{Case 2:} Assume that $\abs{X}=3k+1$ for some $k\in\mathbb{Z}$. Then $k\geqslant 2$ (since $\abs{X}\geqslant 5$). It follows from Theorem~\ref{T(X): maximum size abelian subgroup of S(X)} that
 	\begin{displaymath}
 		\abs{G}=4\cdot 3^{k-1}< 4\cdot 8^{k-1}=2^2\cdot \parens{2^3}^{k-1}=2^{\parens{3k+1}-2}=2^{\abs{X}-2}\leqslant\parens{\abs{X}}\xi.
 	\end{displaymath}
 	
 	\smallskip
 	
 	\textit{Case 3:} Assume that $\abs{X}=3k+2$ for some $k\in\mathbb{Z}$. Then $k\geqslant 1$ (since $\abs{X}\geqslant 5$) and, consequently, $2k-1\geqslant 1$. Hence, by Theorem~\ref{T(X): maximum size abelian subgroup of S(X)}, we have
 	\begin{displaymath}
 		\abs{G}\leqslant 2\cdot 3^k <3^{2k-1} \cdot 3^k =3^{\parens{3k+2}-3}=3^{\abs{X}-3}\leqslant\parens{\abs{X}}\xi.
 	\end{displaymath}
 	
 	\smallskip
 	
 	We just showed that, when $\abs{X}\geqslant 5$, we have $\abs{G}<\maxnull{\abs{X}}$. Moreover, Lemma~\ref{inequalities xi}, and the fact that $\abs{X}\geqslant 5$, imply that $\maxnull{\abs{X}}<\maxnull{\abs{X}+1}$, which concludes the proof.
 \end{proof}
 
 Lemma~\ref{i_1,...,i_s<i, i_2,...,i_s linear levels} is the last lemma we need to modify the tree of a semigroup. This lemma provides some properties of the tree of a semigroup that relate the notions of branching and linear level. This result will be important later to show that the resulting tree (after all the modifications) has enough linear levels to be the tree of a null semigroup. In order to prove Lemma~\ref{i_1,...,i_s<i, i_2,...,i_s linear levels} we require another lemma (Lemma~\ref{T(X): main lemma}), which explains how commutativity restricts the structure of the maps of a commutative transformation semigroup with a unique idempotent and whose proof relies on the concept of $S$-partition.
 
 \begin{lemma}\label{T(X): main lemma}
 	Let $S$ be a commutative subsemigroup of $\tr{X}$ with a unique idempotent and let $\set{A_j}_{j=0}^{k}$ be the $S$-partition of $X$. Let $i\in\X{k}$ and define $A=\bigcup_{j=0}^{i-1}A_j$. Let $x\in A_i$ and $\beta_1,\ldots,\beta_m \in S$ be such that $\beta_1|_A=\cdots=\beta_m|_A$. Then $(x\beta_1)\gamma=\cdots=(x\beta_m)\gamma$ for all $\gamma \in S$.
 \end{lemma}
 
 \begin{proof}
 	Let $\gamma \in S$ and $l,t\in\X{m}$. Since $x\in A_i$, then $x\gamma \in \bigcup_{j=0}^{i-1}A_j=A$. Hence, since $S$ is commutative, we have
 	\begin{displaymath}
 		(x\beta_l)\gamma=(x\gamma)\beta_l=(x\gamma)\beta_l|_A=(x\gamma)\beta_t|_A =(x\gamma)\beta_t=(x\beta_t)\gamma. \qedhere
 	\end{displaymath}
 \end{proof}
 
 \begin{lemma}\label{i_1,...,i_s<i, i_2,...,i_s linear levels}
 	Let $S$ be a commutative subsemigroup of $\tr{X}$ whose unique idempotent is $e\in\tr{X}\setminus\set{\id{X}}$ and assume that $n=\abs{X}$. Let $x_1,\ldots,x_n$ be the order of the elements of $X$ used to construct $T_S$. If there exists $i\in\set{\abs{\im e}+1,\ldots,n}$ such that $T_S$ contains a branching at level $x_i$ with $s\geqslant 2$ arcs whose labels are $x_{i_1},\ldots, x_{i_s}$ (where $i_1<i_2<\cdots<i_s$), then $i_s<i$, the levels $x_{i_2},\ldots,x_{i_s}$ are linear and $x_{i_2},\ldots,x_{i_s}\in \set{x_{\abs{\im e}+1},\ldots,x_n}$.
 \end{lemma}
 
 \begin{proof}
 	Suppose that there exists $i\in\{\abs{\im e}+1,\ldots,n\}$ such that $T_S$ contains a branching at level $x_i$. Assume that there are $s\geqslant 2$ arcs in that branching and that their labels are $x_{i_1},\ldots, x_{i_s}$ (where $i_1<i_2<\cdots<i_s$). Let $\set{A_j}_{j=0}^k$ be the $S$-partition of $X$. We have that, in the sequence $x_1,\ldots,x_n$, the elements of $A_j$ appear before the elements of $A_{j+1}$ for all $j\in\X{k-1}$. Hence $A_0=\im e=\set{x_1,\ldots,x_{\abs{\im e}}}$.
 	
 	Since $i>\abs{\im e}$, then $x_i\in X\setminus A_0=\bigcup_{j=1}^k A_j$, which implies the existence of $l\in\set{1,\ldots,k}$ such that $x_i\in A_l$. Because of the way we ordered the elements of $X$, we have that the elements of $\bigcup_{j=0}^{l-1}A_j$ precede the elements of $A_l$, which implies that the elements of $\bigcup_{j=0}^{l-1}A_j$ precede $x_i$, that is, $\bigcup_{j=0}^{l-1}A_j\subseteq \set{x_1,\ldots,x_{i-1}}$.
 	
 	Furthermore, the existence of a branching at level $x_i$ with $s$ arcs, whose labels are $x_{i_1},\ldots,x_{i_s}$, implies, by part 2 of Lemma~\ref{T(X): labels of T_S}, the existence of $\beta_1,\ldots,\beta_s\in S$ such that $\beta_1,\ldots,\beta_s$ are equal in $\set{x_1,\ldots,x_{i-1}}$ and $x_{i_m}=x_i\beta_m$ for all $m\in\X{s}$. Then, since $\bigcup_{j=0}^{l-1}A_j\subseteq \set{x_1,\ldots,x_{i-1}}$, we also have that $\beta_1,\ldots,\beta_s$ are equal in $\bigcup_{j=0}^{l-1}A_j$.

 	First, we are going to prove that $i_s<i$. Since $i>\abs{\im e}$, then part 1 of Lemma~\ref{T(X): labels of T_S} guarantees that the labels of the arcs of level $x_i$ belong to $\set{x_1,\ldots,x_{i-1}}$. In particular, we have $x_{i_s}\in\set{x_1,\ldots,x_{i-1}}$ and, consequently, $i_s<i$.
 	
 	
 	Now we want to see that the levels $x_{i_2},\ldots,x_{i_s}$ are linear. Let $m\in\set{2,\ldots,s}$. Let $u$ be a vertex that is a word of length $i_m-1$. Then $u$ is the starting vertex of some arc of level $x_{i_m}$. Let $x$ be the $i_1$-th letter of $u$ (notice that $i_1\leqslant i_2-1\leqslant i_m-1$, the length of $u$). Choose one of the arcs whose starting vertex is $u$ and assume that $x'$ is its label. Then the ending vertex of the arc we chose is $ux'$. We have that there exists $\beta\in S$ such that $ux'$ is a prefix of the word $w_{\beta}$ determined by $\beta$. As a consequence of the fact that $u$ is a word of length $i_m-1\geqslant i_2-1\geqslant i_1$ we have that
 	\begin{gather*}
 		x_{i_1}\beta = i_1 \text{-th letter of } w_\beta = i_1 \text{-th letter of } ux' = i_1 \text{-th letter of } u = x\\
 		\shortintertext{and}
 		x_{i_m}\beta = i_m \text{-th letter of } w_\beta = i_m \text{-th letter of } ux'=x'.
 	\end{gather*}
 	Additionally, it follows from Lemma~\ref{T(X): main lemma}, as well as the fact that $x_i\in A_l$ and $\beta_1$ and $\beta_m$ are equal in $\bigcup_{j=0}^{l-1} A_j$, that $\parens{x_i\beta_1}\beta=\parens{x_i\beta_m}\beta$. Consequently, we have
 	\begin{displaymath}
 		x=x_{i_1}\beta=\parens{x_i\beta_1}\beta=\parens{x_i\beta_m}\beta=x_{i_m}\beta=x'.
 	\end{displaymath}
 	Therefore the only arc with starting vertex $u$ is the one with label $x$. Thus $u$ has outdegree $1$.
 	
 	We just proved that all the starting vertices of the arcs of level $x_{i_m}$ have outdegree $1$. Thus the level $x_{i_m}$ is linear. Since $m$ is an arbitrary element of $\set{2,\ldots,s}$, then the levels $x_{i_2},\ldots,x_{i_s}$ are all linear.
 	
 	Finally we are going to prove that $x_{i_2},\ldots,x_{i_s}\in \set{x_{\abs{\im e}+1},\ldots,x_n}$. We consider two cases.
 	
 	\smallskip
 	
 	\textit{Case 1:} Assume that $x_{i_1}\in \set{x_{\abs{\im e}+1},\ldots,x_n}$. Then we immediately obtain that $x_{i_2},\ldots,x_{i_s}\in \set{x_{\abs{\im e}+1},\ldots,x_n}$ (because $i_1<i_2<\cdots<i_s$).
 	
 	\smallskip
 	
 	\textit{Case 2:} Assume that $x_{i_1}\in\set{x_1,\ldots,x_{\abs{\im e}}}=\im e$. Let $m\in\X{s}$ be such that $x_{i_m}\in\set{x_1,\ldots,x_{\abs{\im e}}}=\im e$. Then there exist $y,y'\in X$ such that $ye=x_{i_1}$ and $y'e=x_{i_m}$. In addition, we also know that there exists $t\in\mathbb{N}$ such that $\beta_m^t$ is an idempotent, which implies that $\beta_m^t=e$ (since $e$ is the unique idempotent of $S$). Furthermore, we have that
 	\begin{align*}
 		x_{i_1}\beta_m &=x_i\beta_1\beta_m &\bracks{\text{because } x_i\beta_1=x_{i_1}}\\
 		&=x_i\beta_m\beta_1 &\bracks{\text{because } \beta_1,\beta_m\in S, \text{ which is commutative}}\\
 		&=x_{i_m}\beta_1 &\bracks{\text{because } x_i\beta_m=x_{i_m}}\\
 		&=x_{i_m}\beta_m. &\bracks{\text{because } \beta_1 \text{ is equal to } \beta_m \text{ in } \set{x_1,\ldots,x_{i-1}} \text{ and } i_m\leqslant i_s<i}
 	\end{align*}
 	Consequently, we have
 	\begin{displaymath}
 		x_{i_1}=ye=\parens{ye}e=x_{i_1}e=x_{i_1}\beta_m^t=x_{i_m}\beta_m^t=x_{i_m}e=\parens{y'e}e=y'e=x_{i_m}.
 	\end{displaymath}
 	Thus $\set{x_{i_1},\ldots,x_{i_m}}\cap\im e=\set{x_{i_1}}$ and, consequently, we have $\set{x_{i_2},\ldots,x_{i_m}}\subseteq X\setminus\im e=\set{x_{\abs{\im e}+1},\ldots,x_n}$. 
 \end{proof}
 
 As a consequence of Lemma~\ref{i_1,...,i_s<i, i_2,...,i_s linear levels}, we have that a branching with $s$ arcs is associated to $s$ levels that precede it: the first one can either be a linear or a branching level and the last $s-1$ are all linear levels.
 
 We can finally show how the concept of tree of a semigroup can be used to prove that, for each commutative subsemigroup of $\tr{X}$ whose unique idempotent is not the identity, there is a null subsemigroup of $\tr{X}$ of the same size. Moreover, this result implies, together with Theorem~\ref{maximum size null semigroup}, that the maximum size of these semigroups is $\maxnull{\abs{X}}$.

 \begin{theorem}\label{|S|=|N|, idemp distinct id_X}
 	Let $S$ be a commutative subsemigroup of $\tr{X}$ with a unique idempotent. If that idempotent is distinct from $\id{X}$, then there exists a null subsemigroup $N$ of $\tr{X}$ such that $|S|=|N|$.
 \end{theorem}
 
 \begin{proof}
 	The idea of the proof is to construct the tree of $S$, modify it and obtain a new one which will be the labelled tree of a null semigroup of size $|S|$. (For an illustration of how the proof applies to a particular semigroup, see Example~\ref{example proof}.) Throughout this proof we will perform two modifications on the tree of $S$, after which we will relabel the arcs and rename the vertices of the final tree. In this process we define several trees. In order to make the proof easier to follow, we introduce the diagram below, which provides a scheme of the proof and a way to distinguish the several trees that we will use in it.
 	\begin{equation}\label{T(X): diagram proof tree}
 		\begin{tikzpicture}[baseline=(labelhere)]
 			\coordinate (labelhere) at (0,-3);
 			
 			\node at (0,0) {$T_M$};
 			\draw[->, decorate, decoration=snake] (0.5,0) -- (1.5,0);
 			\node[columnlabel] at (1,-0.4) {\begin{minipage}{1.5cm}\centering\small
 					Removing one arc from the trunk
 			\end{minipage}};
 			\node at (2,0) {$T'_M$};
 			
 			\node at (0,-3) {$T_S$};
 			\draw[->, decorate, decoration=snake] (0.5,-3) -- (1.5,-3);
 			\node[columnlabel] at (1,-3.4) {\begin{minipage}{1.5cm}\centering\small
 					Replacing $T_G$ by $T'_M$
 			\end{minipage}};
 			\node at (2,-3) {$T_1$};
 			\draw[->, decorate, decoration=snake] (2.5,-3) -- (3.5,-3);
 			\node[columnlabel] at (3,-3.4) {\begin{minipage}{1.5cm}\centering\small
 					Moving the linear levels to the trunk
 			\end{minipage}};
 			\node at (4,-3) {$T_2$};
 			\draw[->, decorate, decoration=snake] (4.5,-3) -- (5.5,-3);
 			\node[columnlabel] at (5,-3.4) {\begin{minipage}{1.5cm}\centering\small
 					Adding labels to arcs and renaming vertices
 			\end{minipage}};
 			\node at (6,-3) {$T_N$};
 		\end{tikzpicture}
 	\end{equation}
 	
 	Let $n=|X|$. Assume that $\set{A_j}_{j=0}^k$ is the $S$-partition of $X$ and that the order of the elements of $X$ used to construct $T_S$ is $x_1,\ldots,x_n$. Let $e\in S$ be the unique idempotent of $S$ and assume that $e\neq\id{X}$. We notice that, due to the way we organized the elements of $X$, the first $\abs{A_0}=\abs{\im e}$ elements of $X$ belong to $A_0=\im e$. Thus $\im e=\set{x_1,\ldots,x_{\abs{\im e}}}$.
 	
 	Let $G=\gset{\beta|_{\im e}}{\beta \in S}$. It follows from Lemma~\ref{T(X): tr commutes with idemp} that $\beta|_{\im e}\in\tr{\im e}$ for all $\beta\in S$. Furthermore, the fact that $S$ is a commutative subsemigroup of $\tr{X}$ whose unique idempotent is $e$ implies, by Lemma~\ref{T(X): lemma induction}, that $G$ is a commutative subsemigroup of $\tr{\im e}$ whose unique idempotent is $e|_{\im e}=\id{\im e}$. Hence Proposition~\ref{T(X): unique idemp is id_X => group in S(X)} guarantees that $G$ is an abelian subgroup of $\sym{\im e}$.
 	
 	We consider $T_G$, the tree of $G$, which we construct using the order $x_1,\ldots,x_{\abs{\im e}}$ of the elements of $\im e$. We are going to see that $T_G$ corresponds to the subgraph of $T_S$ located at levels $x_1,\ldots,x_{\abs{\im e}}$ of $T_S$. For each $\beta\in S$ (respectively, $\beta\in G$) let $w_\beta$ (respectively, $w'_\beta$) be the word over $X^*$ (respectively, $\parens{\im e}^*\subseteq X^*$) determined by $\beta$. Let $W_S=\gset{w_\beta}{\beta\in S}$ and $W_G=\gset{w'_\beta}{\beta\in G}$. The length of the words of the sets $W_S$ and $W_G$ is $n$ and $\abs{\im e}$, respectively.
 	
 	
 	The vertex set of $T_G$ is the set of prefixes of the words belonging to $W_G$ and the vertex set of the subgraph of $T_S$ located at levels $x_1,\ldots,x_{\abs{\im e}}$ is the set formed by the prefixes of length at most $\abs{\im e}$ of the words belonging to $W_S$. In order to show that these two vertex sets are equal, it is enough to prove that $W_G$ corresponds to the set of prefixes of length $\abs{\im e}$ of the words belonging to $W_S$. In fact, since $\im e=\set{x_1,\ldots,x_{\abs{\im e}}}$, then for all $w\in X^*$ we have that
 	\begin{align*}
 		w\in W_G &\iff \text{there exists } \beta\in S \text{ such that } w=w'_{\beta|_{\im e}}\\
 		&\iff \left\{\begin{minipage}{8.5cm} the length of $w$ is $\abs{\im e}$ and there exists $\beta\in S$ such that for all $i\in\X{\abs{\im e}}$ the $i$-th letter of $w$ is $x_i\beta|_{\im e}$\end{minipage}\right.\\
 		&\iff \left\{\begin{minipage}{8.5cm} the length of $w$ is $\abs{\im e}$ and there exists $\beta\in S$ such that for all $i\in\X{\abs{\im e}}$ the $i$-th letter of $w$ is $x_i\beta$\end{minipage}\right.\\
 		&\iff \left\{\begin{minipage}{8.5cm} there exists $\beta\in S$ such that $w$ is a prefix of length $\abs{\im e}$ of $w_\beta$\end{minipage}\right.\\
 		&\iff w \text{ is a prefix of length } \abs{\im e} \text{ of a word belonging to } W_S, 
 	\end{align*}
 	which proves the desired equality.
 	
 	We have that the set of words used to construct $T_G$ comprises the words used to construct $T_S$ whose length is at most $\abs{\im e}$. Thus, it follows from the way we defined the tree of a semigroup that $T_G$ is the subgraph of $T_S$ located at levels $x_1,\ldots,x_{\abs{\im e}}$.
 	
 	This result motivates the first modification of $T_S$, which consists on replacing the subgraph $T_G$ of $T_S$ by another tree with the same number of leaves as $T_G$. Before we do that, we delete the labels of all the arcs and the names of all the vertices of the tree $T_S$ (since we will not need them for the rest of the proof). In what follows we explain how to obtain the new tree meant to replace $T_G$.
 	
 	Let $Y=\set{x_1,\ldots,x_{\abs{\im e}+1}}=\im e \cup\set{x_{\abs{\im e}+1}}$ and $t=\alphanull{\abs{Y}}=\alphanull{\abs{\im e}+1}$. (We observe that, since $e\neq\id{X}$, then $\abs{\im e} <\abs{X}=n$.) We have that $\nulltr{Y}{x_1}{x_t}$ is a null subsemigroup of $\tr{Y}$ whose zero is the transformation $f$ over $Y$ such that $\im f=\set{x_1}$. It follows from Lemma~\ref{T(X): unique idemp, upper bound abelian subgroup S(X)} and Theorem~\ref{null semigroups of maximum size} that $\abs{G}\leqslant\maxnull{\abs{\im e}+1}=\maxnull{\abs{Y}}=\abs{\nulltr{Y}{x_1}{x_t}}$. Hence there exists $M\subseteq\nulltr{Y}{x_1}{x_t}$ such that $f\in M$ and $\abs{M}=\abs{G}$. As a consequence of the fact that $\nulltr{Y}{x_1}{x_t}$ is a null subsemigroup of $\tr{Y}$ we have that $M$ is also a null subsemigroup of $\tr{Y}$. Let $\set{B_j}_{j=0}^m$ be the $M$-partition of $Y$. We consider any order of the elements of $Y$ where the elements of $B_j$ appear before the elements of $B_{j+1}$ for all $j\in\X{m-1}$, and we use it to construct the tree $T_M$. In what follows we describe $T_M$. 
 	
 	First we will see that $T_M$ has a trunk whose length is at least $t$. We have that $B_0=\im f=\set{x_1}$. Furthermore, for all $x\in B_1$ and $\beta\in M\subseteq \nulltr{Y}{x_1}{x_t}$ we have that $x\beta\in B_0=\set{x_1}$. (We observe that we must have $m\geqslant 1$ due to the fact that $\abs{Y}\geqslant 2$.) This implies that, for all $i\in\X{\abs{B_0}+\abs{B_1}}$, the $i$-th letter of all the words determined by the transformations of $M$ is $x_1$. Hence $x_1, x_1^2, \ldots, x_1^{\abs{B_0}+\abs{B_1}}$ are prefixes of all these words and, consequently, for each $i\in\X{\abs{B_0}+\abs{B_1}}$ we have that $x_1^i$ is the only vertex of length $i$. In the tree $T_M$, this translates into a path of length $\abs{B_0}+\abs{B_1}$ that begins at the vertex $\varepsilon$ (the root of the tree) and ends at the vertex $x_1^{\abs{B_0}+\abs{B_1}}$, and where all the arcs have label $x_1$. Thus the first $\abs{B_0}+\abs{B_1}$ levels of $T_M$ are linear and, consequently, $T_M$ contains a trunk whose length is at least $\abs{B_0}+\abs{B_1}$. Moreover, for all $\beta\in M\subseteq \nulltr{Y}{x_1}{x_t}$ we have that $\set{x_2,\ldots,x_t}\beta=\set{x_1}$, which implies that $x_2,\ldots,x_t\in B_1$. Thus $\abs{B_1}\geqslant t-1$ and, since $\abs{B_0}=1$, we can conclude that the length of the trunk of $T_M$ is at least $t$.
 	
 	Now we will see that (if $T_M$ contains branchings, then) any branching of $T_M$ contains at most $t$ arcs. Let $i\in\X{\abs{\im e}+1}$ be such that level $x_i$ of $T_M$ contains a branching. For all $\beta\in M\subseteq\nulltr{Y}{x_1}{x_t}$ we have $x_i\beta\in\set{x_1,\ldots,x_t}$, which implies that the labels of the arcs of level $x_i$ must belong to $\set{x_1,\ldots,x_t}$. In particular, the labels of the arcs of any branching at level $x_i$ belong to $\set{x_1,\ldots,x_t}$, which implies that any branching at level $x_i$ has at most $t$ arcs. Since $x_i$ is an arbitrary branching level of $T_M$, we can conclude that any branching of $T_M$ has at most $t$ arcs.
 	
 	Since the trunk of $T_M$ precedes any branching of $T_M$, then the former two paragraphs allow us to conclude that each branching of $T_M$ is preceded by at least $t$ linear levels.
 	
 	Just like we did with $T_S$, we can now remove the labels from all the arcs of $T_M$, as well as the names of its vertices, which will not be necessary for the rest of the proof.
 	
 	At this moment we have enough to perform the first modification of the tree of $S$. Like we mentioned before, we are going to replace the subgraph $T_G$ of $T_S$ by a tree whose number of leaves is equal to the number of leaves of the tree $T_G$. It follows from part 2 of Lemma~\ref{T(X): properties tree of S} that the number of leaves of $T_G$ and $T_M$ are $\abs{G}$ and $\abs{M}$, respectively. Due to the fact that $\abs{G}=\abs{M}$, then $T_M$ is a good choice to replace $T_G$. However, $T_M$ has one more level than $T_G$: $T_M$ has $\abs{\im e}+1$ levels and $T_G$ has $\abs{\im e}$ levels. In order to solve this discrepancy, we remove one arc from the trunk of $T_M$ and obtain a new tree --- which we denote by $T'_M$ --- that has $\abs{G}=\abs{M}$ leaves, $\abs{\im e}$ levels and a trunk of length at least $t-1$. (We note that $t=\alphanull{\abs{Y}}=\alphanull{\abs{\im e}+1}\geqslant 2$ because $2^{\abs{Y}-2}=2^{\abs{\im e}-1}\geqslant 1=1^{\abs{Y}-1}$.) In $T_S$, we replace $T_G$ by this new tree. Since $T_G$ and $T'_M$ have the same number of leaves, then this replacement does not cause any problems. We are going to denote the tree we obtain from $T_S$, after the first modification, by $T_1$. We notice that, due to the fact that $e\neq\id{X}$, we have that $\abs{\im e}<n$. Hence $T_G$ is not equal to $T_S$ and $T_S$ contains more levels than $T_G$. More specifically, the last level of $T_S$ (the one where the leaves are) is not a level of $T_G$. Thus, replacing $T_G$ with $T'_M$ in $T_S$, and obtaining $T_1$, does not change the number of leaves of the tree and, consequently, $T_1$ has $\abs{S}$ leaves.

 	Now we are ready to do the second modification. We consider all the linear levels of $T_1$ that do not correspond to the trunk of the tree. Assume that there are $r$ linear levels in the tree $T_1$, $r'$ of which are the linear levels outside of the trunk. Then $r$ is equal to the sum of $r'$ and the number of arcs in the trunk of $T_1$. We are going to move those $r'$ linear levels to the trunk of the tree, that is, we are going to eliminate all the arcs that correspond to those levels, and we are going to add $r'$ arcs to the trunk of the tree (that is, we are adding $r'$ linear levels to the trunk). Of course, if $T_1$ has all its linear levels in the trunk, then we do not need to perform any changes in the tree. Note that, since all the starting vertices of the arcs belonging to the linear levels have outdegree $1$, then eliminating linear levels does not cause any problems in the tree. This entire process does not change either the number of leaves of the tree, or the number of linear and branching levels of the tree. Furthermore, these transformations do not create new branchings and maintain the number of arcs of the existing ones. This means that each branching of the resulting tree --- which we denote by $T_2$ --- was also a branching of $T_1$ (and it has the same number of arcs). We also note that the tree $T_2$ has a trunk with $r$ arcs and all its linear levels are the ones associated with its trunk.
 	
 	Before we show that it is possible to obtain a null semigroup from $T_2$ we need to demonstrate that the length of the trunk of $T_2$ is an upper bound for the number of arcs of each branching of $T_2$. Since $T_2$ was obtained from $T_1$ by moving all the linear levels to its trunk, then it is enough to demonstrate that the number of linear levels of $T_1$ is an upper bound for the number of arcs of each branching of $T_1$. Assume that there is a branching in $T_1$ with $s$ arcs. We have two possible cases, depending on the location of that branching in the tree $T_1$. We recall that the tree $T'_M$ is a subgraph of $T_1$. Hence the branching could be in the tree $T'_M$ or outside the tree $T'_M$.
 	
 	\smallskip
 	
 	\textit{Case 1:} Assume that the branching is not in $T'_M$. Then the branching comes from the original tree $T_S$ and it is located at one of the levels $x_{\abs{\im e}+1},\ldots,x_n$ of $T_S$. (Recall that $T_1$ was obtained from $T_S$ by replacing $T_G$ --- located at levels $x_1,\ldots,x_{\abs{\im e}}$ --- with $T'_M$ and, consequently, no changes were made at the levels $x_{\abs{\im e}+1},\ldots,x_n$.) Assume that, in $T_S$, the labels of the arcs of the selected branching are $x_{i_1},\ldots,x_{i_s}$ and assume that $i_1<i_2<\cdots<i_s$. Then, by Lemma~\ref{i_1,...,i_s<i, i_2,...,i_s linear levels}, we have that levels $x_{i_2},\ldots,x_{i_s}$ of $T_S$ are linear and $x_{i_2},\ldots,x_{i_s}\in\set{x_{\abs{\im e}+1},\ldots,x_n}$. Since the process of modifying $T_S$ and obtaining $T_1$ does not make any changes in the levels $x_{\abs{\im e}+1},\ldots,x_n$, then we can conclude that $T_1$ contains $s-1$ linear levels that are not levels of $T'_M$. Additionally, $T_1$ contains a trunk (which is the trunk of $T'_M$), which implies that there is at least one more linear level in $T_1$ (which is located in $T'_M$ and, consequently, is distinct from the previous $s-1$ linear levels). Thus $T_1$ contains at least $s$ linear levels, the number of arcs of the branching we chose.
 	
 	\smallskip
 	
 	\textit{Case 2:} Assume that the branching is in $T'_M$. Since $T'_M$ is obtained from $T_M$ by removing the first arc of its trunk, then the branching is also in $T_M$. We proved earlier that every branching of $T_M$ is preceded by at least as many linear levels as arcs of that branching. Hence the branching we are considering, which has $s$ arcs, is preceded by at least $s$ linear levels in $T_M$. Therefore the branching is preceded by at least $s-1$ linear levels in $T'_M$ (and, consequently, it is preceded by $s-1$ linear levels in $T_1$). In order to see that $T_1$ contains at least one more linear level, we consider the following two sub-cases.
 	
 	\smallskip
 	
 	\textsc{Sub-case 1:} Assume that all the levels of $T_1$ that are not levels of $T'_M$ are linear. Then $T_1$ contains one linear level which is not part of $T'_M$, that is, the linear level is distinct from the $s-1$ linear levels previously mentioned.
 	
 	\smallskip
 	
 	\textsc{Sub-case 2:} Assume that among the levels of $T_1$ that are not levels of $T'_M$ there is a branching level. If we present an argument similar to the one of case 1, then we can conclude that there exists at least one linear level that is not a level of $T'_M$ and, consequently, there exists one linear level in $T_1$ distinct from the $s-1$ linear levels previously mentioned.
 	
 	\smallskip
 	
 	It follows from cases 1 and 2 that $T_1$ has at least $s$ linear levels. Moreover, since we considered an arbitrary branching of $T_1$, then we can conclude that the number of linear levels of $T_1$ is an upper bound for the number of arcs of any branching of $T_1$. Consequently, the length of the trunk of $T_2$ (and, consequently, the number of linear levels of $T_2$) is an upper bound for the number of arcs of each branching. Since the length of the trunk of $T_2$ is $r$, then we can conclude that each branching of $T_2$ has at most $r$ arcs.
 	
 	For the remainder of the proof we will show how to extract a null semigroup from the tree $T_2$. The first thing we need to do is to add labels to the arcs of $T_2$ and rename its vertices (we recall that we removed these at the beginning of the proof). We do this in a way that guarantees that $T_2$ is the tree of a (null) semigroup.
 	
 	We start by labelling the arcs. All the $r$ arcs belonging to the trunk of the tree are labelled by $x_1$. We now consider the starting vertices of the arcs that do not belong to the trunk of the new tree. We want to label these arcs using exclusively elements from $\set{x_1,\ldots,x_r}$. If we have a vertex with outdegree $1$ then we label the corresponding arc by $x_1$. Now assume that we have a vertex with outdegree $s\geqslant 2$. Then we have a branching at that vertex and, since each branching of $T_2$ has at most $r$ arcs, we must have $s\leqslant r$ and, consequently, we label the arcs of this branching by $x_1,\ldots,x_s$.
 	
 	Finally, we rename the vertices. We want the vertices to be the prefixes of the words associated with the leaves, which should be words of length $n$. Hence the root of the tree needs to be the word $\varepsilon$. We also want to guarantee that, given two vertices $u$ and $v$, there is an arc labelled by $x$ from $u$ to $v$ if and only if $v=ux$. Hence the vertices that are not the root must be given by $wx$, where $x$ is the label of the only arc that ends at the vertex we are considering and $w$ is the starting vertex of that arc.
 	
 	Let $Z$ be the set of words formed by the labels of the leaves of $T_2$. Note that the trees $T_S$, $T_1$ and $T_2$ have the same number of leaves, which is equal to $\abs{S}$. Then we have $\abs{S}$ words, all of which have length $n$. Using again the order $x_1,\ldots,x_n$ of the elements of $X$, we are going to obtain from each word of $Z$ a transformation of $\tr{X}$. Let $w=w_1\cdots w_n\in Z$ (where $w_1,\ldots,w_n \in \set{x_1,\ldots,x_n}$). Then $w$ determines the transformation $\beta\in \tr{X}$ such that $x_i\beta=w_i$. Let $N$ be the set formed by the transformations obtained from $Z$. We want to prove that $N$ is a null semigroup. First, we notice that $x_1^n\in Z$. Hence the constant map $f$ with image $\set{x_1}$ belongs to $N$. Let $\beta,\gamma\in N$ and $x\in X$. Since the labels of the arcs of the new tree belong to $\set{x_1,\ldots,x_r}$, then $Z\subseteq \set{x_1,\ldots,x_r}^*$ and, consequently, $x\beta\in\set{x_1,\ldots,x_r}$. However, at the trunk of the tree $T_2$, the arcs are all labelled $x_1$, which implies that $x_1^r$ is a prefix of all the words in $Z$. Therefore $\set{x_1,\ldots,x_r}\gamma=\set{x_1}$ and, as a consequence, $x\beta\gamma=x_1$. Thus $\beta\gamma=f$.
 	
 	Therefore $N$ is a null subsemigroup of $\tr{X}$ such that $\abs{N}=\abs{Z}=\abs{S}$. Moreover, we observe that, when we added labels to the arcs of the tree $T_2$ and named its vertices, we obtained the tree $T_N$ (constructed using the order $x_1,\ldots,x_n$ of the elements of $X$).
 \end{proof}
 
 \begin{example}\label{example proof}
 	The present example serves as a way to show how the proof of Theorem~\ref{|S|=|N|, idemp distinct id_X} works. We will use the semigroup $S$ described in Example~\ref{example S-partition} to do this. Moreover, in Example~\ref{T(X): example tree} we constructed $T_S$ (see Figure~\ref{T(X): figure, tree of S}).
 	
 	Let $e$ be the idempotent of $S$, that is,
 	\begin{displaymath}
 		e=\begin{pmatrix} 1&2&3&4&5&6&7 \\ 1&7&4&4&4&4&7\end{pmatrix}
 	\end{displaymath}
 	and let
 	\begin{displaymath}
 		G=\gset{\beta|_{\im e}}{\beta\in S}=\gset{\beta_{\set{1,4,7}}}{\beta\in S}=\set*{\begin{pmatrix}
 				1&4&7\\1&4&7
 			\end{pmatrix},
 			\begin{pmatrix}
 				1&4&7\\4&7&1
 			\end{pmatrix},
 			\begin{pmatrix}
 				1&4&7\\7&1&4
 		\end{pmatrix}}.
 	\end{displaymath}
 	
 	If we use the sequence $1,4,7$ to construct $T_G$, then the tree we obtain is the one inside the dashed rectangle in Figure~\ref{T(X): figure, T_G in T_S}, that is, the subgraph of $T_S$ located at levels $1$, $4$ and $7$.	
 	
 	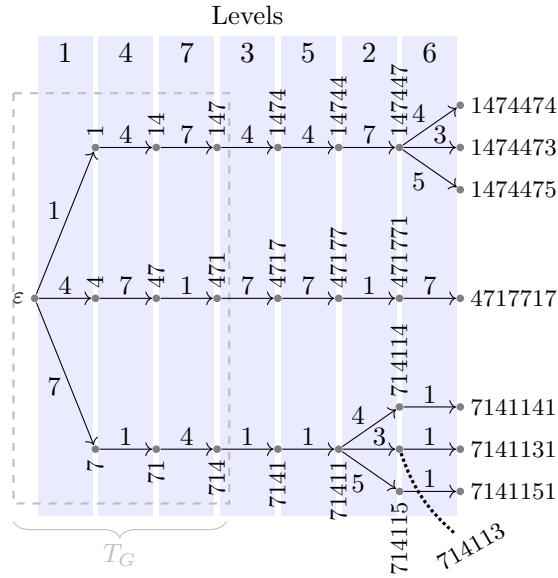
\begin{figure}[h!tb]
 		\begin{center}
 			
 			\begin{tikzpicture}[x=8mm,y=8mm]
 				
 				\fill[color=blue!8](0.05,4.35) rectangle (0.95,-3.6);
 				\fill[color=blue!8](1.05,4.35) rectangle (1.95,-3.6);
 				\fill[color=blue!8](2.05,4.35) rectangle (2.95,-3.6);
 				\fill[color=blue!8](3.05,4.35) rectangle (3.95,-3.6);
 				\fill[color=blue!8](4.05,4.35) rectangle (4.95,-3.6);
 				\fill[color=blue!8](5.05,4.35) rectangle (5.95,-3.6);
 				\fill[color=blue!8](6.05,4.35) rectangle (6.95,-3.6);
 				
 				\foreach \x/\t in {0/1,1/4,2/7,3/3,4/5,5/2,6/6} {
 					\node[columnlabel] at ({\x+.5},4.4) {\t};
 				}
 				\node at (3.5,4.7) {\small Levels};
 				
 				\draw[color=lightgray, dashed, thick](-0.35,3.4) rectangle (3.2,-3.4);
 				
 				\draw[decorate, decoration={brace, amplitude=6pt, mirror}, color=lightgray] (-0.35,-3.7) -- (3.15,-3.7);
 				
 				
 				\node at (1.4,-4.25) {\small \textcolor{lightgray}{$T_G$}};
 				
 				\begin{scope}[
 					every node/.style={treenode},
 					]
 					\node (root) at (0,0) {};
 					\node (1) at (1,2.5) {};
 					\node (14) at (2,2.5) {};
 					\node (147) at (3,2.5) {};
 					\node (1474) at (4,2.5) {};
 					\node (14744) at (5,2.5) {};
 					\node (147447) at (6,2.5) {};
 					\node (1474474) at (7,3.2) {};
 					\node (1474473) at (7,2.5) {};
 					\node (1474475) at (7,1.8) {};
 					\node (4) at (1,0) {};
 					\node (47) at (2,0) {};
 					\node (471) at (3,0) {};
 					\node (4717) at (4,0) {};
 					\node (47177) at (5,0) {};
 					\node (471771) at (6,0) {};
 					\node (4717717) at (7,0) {};
 					\node (7) at (1,-2.5) {};
 					\node (71) at (2,-2.5) {};
 					\node (714) at (3,-2.5) {};
 					\node (7141) at (4,-2.5) {};
 					\node (71411) at (5,-2.5) {};
 					\node (714114) at (6,-1.8) {};
 					\node (714113) at (6,-2.5) {};
 					\node (714115) at (6,-3.2) {};
 					\node (7141141) at (7,-1.8) {};
 					\node (7141131) at (7,-2.5) {};
 					\node (7141151) at (7,-3.2) {};
 				\end{scope}
 				
 				\node[anchor=east] at (root) {\footnotesize $\varepsilon$};
 				\node[anchor=west, rotate=90] at (1) {\footnotesize $1$};
 				\node[anchor=west, rotate=90] at (14) {\footnotesize $14$};
 				\node[anchor=west, rotate=90] at (147) {\footnotesize $147$};
 				\node[anchor=west, rotate=90] at (1474) {\footnotesize $1474$};
 				\node[anchor=west, rotate=90] at (14744) {\footnotesize $14744$};
 				\node[anchor=west, rotate=90] at (147447) {\footnotesize $147447$};
 				\node[anchor=west] at (1474474) {\footnotesize $1474474$};
 				\node[anchor=west] at (1474473) {\footnotesize $1474473$};
 				\node[anchor=west] at (1474475) {\footnotesize $1474475$};
 				\node[anchor=west, rotate=90] at (4) {\footnotesize $4$};
 				\node[anchor=west, rotate=90] at (47) {\footnotesize $47$};
 				\node[anchor=west, rotate=90] at (471) {\footnotesize $471$};
 				\node[anchor=west, rotate=90] at (4717) {\footnotesize $4717$};
 				\node[anchor=west, rotate=90] at (47177) {\footnotesize $47177$};
 				\node[anchor=west, rotate=90] at (471771) {\footnotesize $471771$};
 				\node[anchor=west] at (4717717) {\footnotesize $4717717$};
 				\node[anchor=east, rotate=90] at (7) {\footnotesize $7$};
 				\node[anchor=east, rotate=90] at (71) {\footnotesize $71$};
 				\node[anchor=east, rotate=90] at (714) {\footnotesize $714$};
 				\node[anchor=east, rotate=90] at (7141) {\footnotesize $7141$};
 				\node[anchor=east, rotate=90] at (71411) {\footnotesize $71411$};
 				\node[anchor=west, rotate=90] at (714114) {\footnotesize $714114$};
 				\node[anchor=east, rotate=90] at (714115) {\footnotesize $714115$};
 				\node[anchor=west] at (7141141) {\footnotesize $7141141$};
 				\node[anchor=west] at (7141131) {\footnotesize $7141131$};
 				\node[anchor=west] at (7141151) {\footnotesize $7141151$};
 				
 				\node[anchor=west, rotate=30] (label) at (6.5,-4.5) {\footnotesize $714113$};
 				\draw[very thick, densely dotted] (714113) edge[bend right=15] (label);

 				\begin{scope}[
 					->,
 					every node/.style={edgelabel},
 					]
 					\draw (root) -- node[anchor = south east] {$1$} (1);
 					\draw (1) -- node[above] {$4$} (14);
 					\draw (14) -- node[above] {$7$} (147);
 					\draw (147) -- node[above] {$4$} (1474);
 					\draw (1474) -- node[above] {$4$} (14744);
 					\draw (14744) -- node[above] {$7$} (147447);
 					\draw (147447) -- node[anchor=south east] {$4$} (1474474);
 					\draw (147447) -- node[pos=.7, above] {$3$} (1474473);
 					\draw (147447) -- node[anchor=north east] {$5$} (1474475);
 					\draw (root) -- node[above] {$4$} (4);
 					\draw (4) -- node[above] {$7$} (47);
 					\draw (47) -- node[above] {$1$} (471);
 					\draw (471) -- node[above] {$7$} (4717);
 					\draw (4717) -- node[above] {$7$} (47177);
 					\draw (47177) -- node[above] {$1$} (471771);
 					\draw (471771) -- node[above] {$7$} (4717717);
 					\draw (root) -- node[anchor=north east] {$7$} (7);
 					\draw (7) -- node[above] {$1$} (71);
 					\draw (71) -- node[above] {$4$} (714);
 					\draw (714) -- node[above] {$1$} (7141);
 					\draw (7141) -- node[above] {$1$} (71411);
 					\draw (71411) -- node[anchor=south east] {$4$} (714114);
 					\draw (71411) -- node[pos=.7, above] {$3$} (714113);
 					\draw (71411) -- node[anchor=north east] {$5$} (714115);
 					\draw (714114) -- node[above] {$1$} (7141141);
 					\draw (714113) -- node[above] {$1$} (7141131);
 					\draw (714115) -- node[above] {$1$} (7141151);
 				\end{scope}
 			\end{tikzpicture}
 			\caption{Tree of $S$ (with tree of $G$ highlighted).}
 			\label{T(X): figure, T_G in T_S}
 		\end{center}
 	\end{figure}
 	
 	Let $Y=\im e \cup \set{3}=\set{1,4,7,3}$. We have that $\alphanull{\abs{Y}}=\alphanull{4}=2$ and $N^Y_{1,4}$ is a null semigroup of size $\maxnull{\abs{Y}}=\maxnull{4}=4$. Now we construct a null subsemigroup of $N^Y_{1,4}$ of size $\abs{G}=3$. Let
 	\begin{displaymath}
 		M=\set*{\begin{pmatrix}
 				1&4&7&3\\1&1&1&1
 			\end{pmatrix},
 			\begin{pmatrix}
 				1&4&7&3\\1&1&1&2
 			\end{pmatrix},
 			\begin{pmatrix}
 				1&4&7&3\\1&1&2&1
 		\end{pmatrix}}
 	\end{displaymath}
 	be that semigroup. In Figure~\ref{T(X): figure, T_M} we have the tree of $M$ constructed using the sequence $1,4,7,3$ of the elements of $Y$. We observe that $T_M$ contains a trunk of length $2$. The tree $T'_M$, obtained from $T_M$ by removing the first arc of its trunk, corresponds to the one located at the first three levels of the rightmost tree of Figure~\ref{T(X): figure, T_S --> T_1}.
 	
 	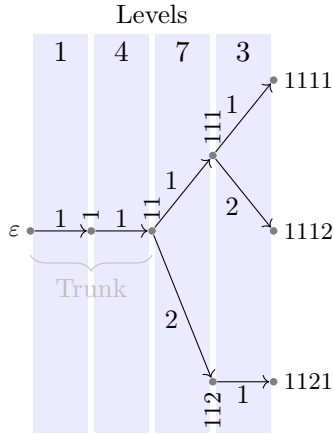
\begin{figure}[h!bt]
 		\begin{center}
 			
 			\begin{tikzpicture}[x=8mm,y=8mm]
 				
 				\fill[color=blue!8](0.05,3.25) rectangle (0.95,-3.35);
 				\fill[color=blue!8](1.05,3.25) rectangle (1.95,-3.35);
 				\fill[color=blue!8](2.05,3.25) rectangle (2.95,-3.35);
 				\fill[color=blue!8](3.05,3.25) rectangle (3.95,-3.35);
 				
 				\draw[decorate, decoration={brace, amplitude=6pt, mirror}, color=lightgray] (0,-0.4) -- (2,-0.4);
 				
 				
 				\node at (1,-0.95) {\small \textcolor{lightgray}{Trunk}};
 				
 				\foreach \x/\t in {0/1,1/4,2/7,3/3} {
 					\node[columnlabel] at ({\x+.5},3.3) {\t};
 				}
 				\node at (2,3.6) {\small Levels};
 				
 				\begin{scope}[
 					every node/.style={treenode},
 					]
 					\node (root) at (0,0) {};
 					\node (1) at (1,0) {};
 					\node (11) at (2,0) {};
 					\node (111) at (3,1.25) {};
 					\node (1111) at (4,2.5) {};
 					\node (1112) at (4,0) {};
 					\node (112) at (3,-2.5) {};
 					\node (1121) at (4,-2.5) {};
 				\end{scope}
 				
 				\node[anchor=east] at (root) {\footnotesize $\varepsilon$};
 				\node[anchor=west, rotate=90] at (1) {\footnotesize $1$};
 				\node[anchor=west, rotate=90] at (11) {\footnotesize $11$};
 				\node[anchor=west, rotate=90] at (111) {\footnotesize $111$};
 				\node[anchor=west] at (1111) {\footnotesize $1111$};
 				\node[anchor=west] at (1112) {\footnotesize $1112$};
 				\node[anchor=east, rotate=90] at (112) {\footnotesize $112$};
 				\node[anchor=west] at (1121) {\footnotesize $1121$};
 				
 				\begin{scope}[
 					->,
 					every node/.style={edgelabel},
 					]
 					\draw (root) -- node[above] {$1$} (1);
 					\draw (1) -- node[above] {$1$} (11);
 					\draw (11) -- node[anchor=south east] {$1$} (111);
 					\draw (111) -- node[anchor=south east] {$1$} (1111);
 					\draw (111) -- node[anchor=north east] {$2$} (1112);
 					\draw (11) -- node[anchor=north east] {$2$} (112);
 					\draw (112) -- node[below] {$1$} (1121);
 				\end{scope}
 			\end{tikzpicture}
 			\caption{Tree of $M$.}
 			\label{T(X): figure, T_M}
 		\end{center}
 	\end{figure}
 	
 	Now we are going to perform some modifications in tree of $S$ in order to obtain a new tree. The first modification we do is replacing, in the tree $T_S$, the tree $T_G$ by the tree $T'_M$. This can be seen in Figure~\ref{T(X): figure, T_S --> T_1}: the tree on the left is $T_S$ and the tree on the right is the resulting tree, which we denote by $T_1$.
 	
 	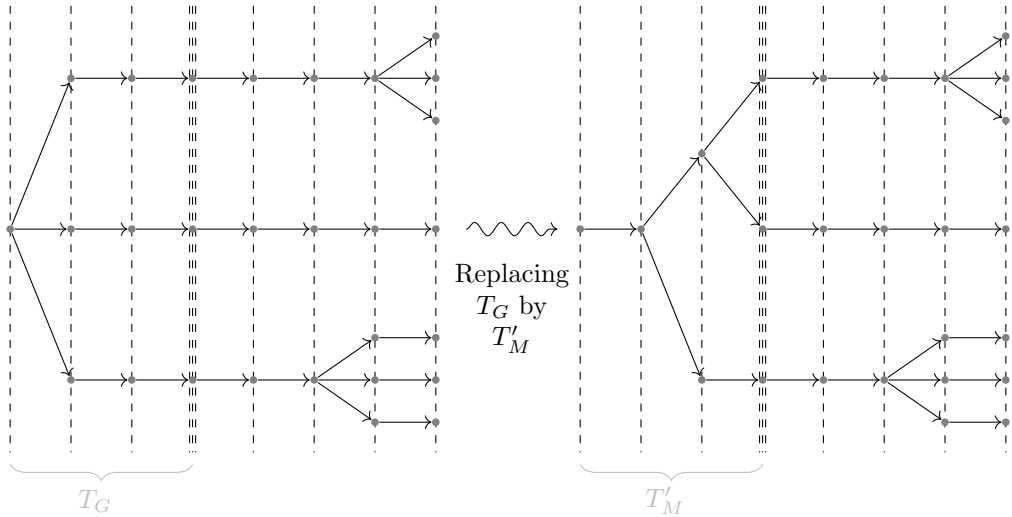
\begin{figure}[h!tb]
 		\begin{center}
 			
 			\begin{tikzpicture}[x=8mm,y=8mm]
 				
 				\foreach \x in {0,1,2,3,4,5,6,7} {
 					\draw[black,dashed] (\x,3.7) --  (\x,-3.7);
 				}
 				\draw[black,dashed] (2.95,3.7) --  (2.95,-3.7);
 				\draw[black,dashed] (3.05,3.7) --  (3.05,-3.7);

 				\draw[decorate, decoration={brace, amplitude=6pt, mirror}, color=lightgray] (0,-3.9) -- (3,-3.9);
 				
 				\node at (1.4,-4.5) {\small \textcolor{lightgray}{$T_G$}};
 				
 				\begin{scope}[
 					every node/.style={treenode},
 					]
 					\node (root) at (0,0) {};
 					\node (1) at (1,2.5) {};
 					\node (14) at (2,2.5) {};
 					\node (147) at (3,2.5) {};
 					\node (1474) at (4,2.5) {};
 					\node (14744) at (5,2.5) {};
 					\node (147447) at (6,2.5) {};
 					\node (1474474) at (7,3.2) {};
 					\node (1474473) at (7,2.5) {};
 					\node (1474475) at (7,1.8) {};
 					\node (4) at (1,0) {};
 					\node (47) at (2,0) {};
 					\node (471) at (3,0) {};
 					\node (4717) at (4,0) {};
 					\node (47177) at (5,0) {};
 					\node (471771) at (6,0) {};
 					\node (4717717) at (7,0) {};
 					\node (7) at (1,-2.5) {};
 					\node (71) at (2,-2.5) {};
 					\node (714) at (3,-2.5) {};
 					\node (7141) at (4,-2.5) {};
 					\node (71411) at (5,-2.5) {};
 					\node (714114) at (6,-1.8) {};
 					\node (714113) at (6,-2.5) {};
 					\node (714115) at (6,-3.2) {};
 					\node (7141141) at (7,-1.8) {};
 					\node (7141131) at (7,-2.5) {};
 					\node (7141151) at (7,-3.2) {};
 				\end{scope}
 				
 				\begin{scope}[
 					->,
 					every node/.style={edgelabel},
 					]
 					\draw (root) -- (1);
 					\draw (1) -- (14);
 					\draw (14) -- (147);
 					\draw (147) -- (1474);
 					\draw (1474) -- (14744);
 					\draw (14744) -- (147447);
 					\draw (147447) -- (1474474);
 					\draw (147447) -- (1474473);
 					\draw (147447) -- (1474475);
 					\draw (root) -- (4);
 					\draw (4) -- (47);
 					\draw (47) -- (471);
 					\draw (471) -- (4717);
 					\draw (4717) -- (47177);
 					\draw (47177) -- (471771);
 					\draw (471771) -- (4717717);
 					\draw (root) -- (7);
 					\draw (7) -- (71);
 					\draw (71) -- (714);
 					\draw (714) -- (7141);
 					\draw (7141) -- (71411);
 					\draw (71411) -- (714114);
 					\draw (71411) -- (714113);
 					\draw (71411) -- (714115);
 					\draw (714114) -- (7141141);
 					\draw (714113) -- (7141131);
 					\draw (714115) -- (7141151);
 				\end{scope}

 				\draw[->, decorate, decoration=snake] (7.5,0) -- (9,0);
 				\node[columnlabel] at (8.25,-0.4) {\begin{minipage}{1.5cm}\centering\small
 						Replacing $T_G$ by $T'_M$
 				\end{minipage}};

 				\begin{scope}[xshift=7.5cm]
 					\foreach \x in {0,1,2,3,4,5,6,7} {
 						\draw[black,dashed] (\x,3.7) --  (\x,-3.7);
 					}
 					\draw[black,dashed] (2.95,3.7) --  (2.95,-3.7);
 					\draw[black,dashed] (3.05,3.7) --  (3.05,-3.7);

 					\draw[decorate, decoration={brace, amplitude=6pt, mirror}, color=lightgray] (0,-3.9) -- (3,-3.9);
 					
 					\node at (1.4,-4.5) {\small \textcolor{lightgray}{$T'_M$}};
 					
 					\begin{scope}[
 						every node/.style={treenode},
 						]
 						\node (1) at (0,0) {};
 						\node (11) at (1,0) {};
 						\node (111) at (2,1.25) {};
 						\node (1111) at (3,2.5) {};
 						\node (1112) at (3,0) {};
 						\node (112) at (2,-2.5) {};
 						\node (1121) at (3,-2.5) {};
 						\node (1474) at (4,2.5) {};
 						\node (14744) at (5,2.5) {};
 						\node (147447) at (6,2.5) {};
 						\node (1474474) at (7,3.2) {};
 						\node (1474473) at (7,2.5) {};
 						\node (1474475) at (7,1.8) {};
 						\node (4717) at (4,0) {};
 						\node (47177) at (5,0) {};
 						\node (471771) at (6,0) {};
 						\node (4717717) at (7,0) {};
 						\node (7141) at (4,-2.5) {};
 						\node (71411) at (5,-2.5) {};
 						\node (714114) at (6,-1.8) {};
 						\node (714113) at (6,-2.5) {};
 						\node (714115) at (6,-3.2) {};
 						\node (7141141) at (7,-1.8) {};
 						\node (7141131) at (7,-2.5) {};
 						\node (7141151) at (7,-3.2) {};
 					\end{scope}
 					
 					\begin{scope}[
 						->,
 						every node/.style={edgelabel},
 						]
 						\draw (1) -- (11);
 						\draw (11) -- (111);
 						\draw (111) -- (1111);
 						\draw (111) -- (1112);
 						\draw (11) -- (112);
 						\draw (112) -- (1121);
 						\draw (1111) -- (1474);
 						\draw (1474) -- (14744);
 						\draw (14744) -- (147447);
 						\draw (147447) -- (1474474);
 						\draw (147447) -- (1474473);
 						\draw (147447) -- (1474475);
 						\draw (1112) -- (4717);
 						\draw (4717) -- (47177);
 						\draw (47177) -- (471771);
 						\draw (471771) -- (4717717);
 						\draw (1121) -- (7141);
 						\draw (7141) -- (71411);
 						\draw (71411) -- (714114);
 						\draw (71411) -- (714113);
 						\draw (71411) -- (714115);
 						\draw (714114) -- (7141141);
 						\draw (714113) -- (7141131);
 						\draw (714115) -- (7141151);
 					\end{scope}
 				\end{scope}
 			\end{tikzpicture}
 			\caption{Transforming the tree $T_S$ into the tree $T_1$.}
 			\label{T(X): figure, T_S --> T_1}
 		\end{center}
 	\end{figure}
 	
 	Now we modify the tree $T_1$ and obtain a new one, which we denote by $T_2$. This modification can be seen in Figure~\ref{T(X): figure, T_1 --> T_2}. We remove the two linear levels of $T_1$ which are not in the trunk (that is, we delete the arcs which belong to the linear levels outside the trunk --- the ones in bold in the tree on the left in Figure~\ref{T(X): figure, T_1 --> T_2}), and then we add two linear levels to the trunk of the tree (that is, we add two arcs to the trunk --- the ones in bold in the tree on the right in Figure~\ref{T(X): figure, T_1 --> T_2}).
 	
 	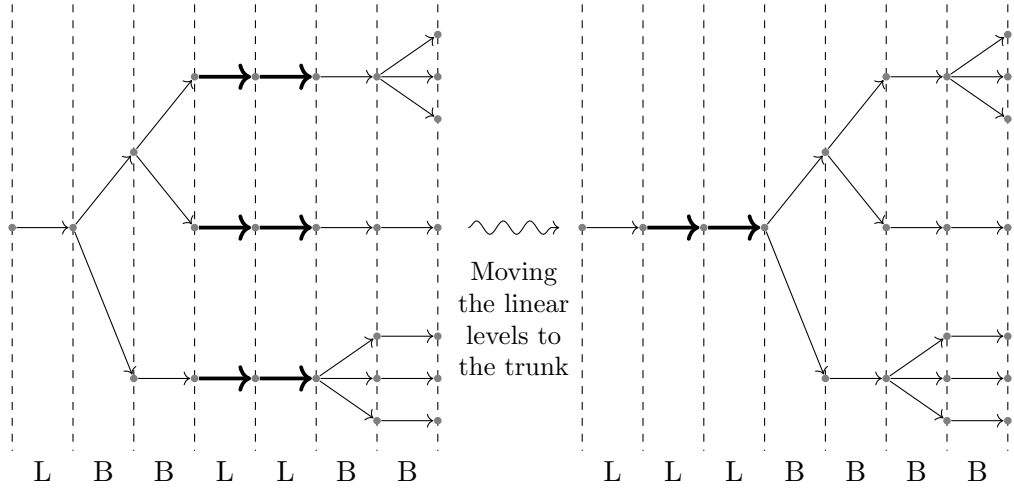
\begin{figure}[h!bt]
 		\begin{center}
 			
 			\begin{tikzpicture}[x=8mm,y=8mm]
 				
 				\foreach \x in {0,1,2,3,4,5,6,7} {
 					\draw[black,dashed] (\x,3.7) --  (\x,-3.7);
 				}
 				\foreach \x/\t in {0/L,1/B,2/B,3/L,4/L,5/B,6/B} {
 					\node[columnlabel] at ({\x+.5},-3.7) {\t};
 				}
 				
 				\begin{scope}[
 					every node/.style={treenode},
 					]
 					\node (1) at (0,0) {};
 					\node (11) at (1,0) {};
 					\node (111) at (2,1.25) {};
 					\node (1111) at (3,2.5) {};
 					\node (1112) at (3,0) {};
 					\node (112) at (2,-2.5) {};
 					\node (1121) at (3,-2.5) {};
 					\node (1474) at (4,2.5) {};
 					\node (14744) at (5,2.5) {};
 					\node (147447) at (6,2.5) {};
 					\node (1474474) at (7,3.2) {};
 					\node (1474473) at (7,2.5) {};
 					\node (1474475) at (7,1.8) {};
 					\node (4717) at (4,0) {};
 					\node (47177) at (5,0) {};
 					\node (471771) at (6,0) {};
 					\node (4717717) at (7,0) {};
 					\node (7141) at (4,-2.5) {};
 					\node (71411) at (5,-2.5) {};
 					\node (714114) at (6,-1.8) {};
 					\node (714113) at (6,-2.5) {};
 					\node (714115) at (6,-3.2) {};
 					\node (7141141) at (7,-1.8) {};
 					\node (7141131) at (7,-2.5) {};
 					\node (7141151) at (7,-3.2) {};
 				\end{scope}
 				
 				\begin{scope}[
 					->,
 					every node/.style={edgelabel},
 					]
 					\draw (1) -- (11);
 					\draw (11) -- (111);
 					\draw (111) -- (1111);
 					\draw (111) -- (1112);
 					\draw (11) -- (112);
 					\draw (112) -- (1121);
 					\draw[line width=.5mm] (1111) -- (1474);
 					\draw[line width=.5mm] (1474) -- (14744);
 					\draw (14744) -- (147447);
 					\draw (147447) -- (1474474);
 					\draw (147447) -- (1474473);
 					\draw (147447) -- (1474475);
 					\draw[line width=.5mm] (1112) -- (4717);
 					\draw[line width=.5mm] (4717) -- (47177);
 					\draw (47177) -- (471771);
 					\draw (471771) -- (4717717);
 					\draw[line width=.5mm] (1121) -- (7141);
 					\draw[line width=.5mm] (7141) -- (71411);
 					\draw (71411) -- (714114);
 					\draw (71411) -- (714113);
 					\draw (71411) -- (714115);
 					\draw (714114) -- (7141141);
 					\draw (714113) -- (7141131);
 					\draw (714115) -- (7141151);
 				\end{scope}

 				\draw[->, decorate, decoration=snake] (7.5,0) -- (9,0);
 				\node[columnlabel] at (8.25,-0.4) {\begin{minipage}{1.5cm}\centering\small
 						Moving the linear levels to the trunk
 				\end{minipage}};

 				\begin{scope}[xshift=7.5cm]
 					\foreach \x in {0,1,2,3,4,5,6,7} {
 						\draw[black,dashed] (\x,3.7) --  (\x,-3.7);
 					}
 					\foreach \x/\t in {0/L,1/L,2/L,3/B,4/B,5/B,6/B} {
 						\node[columnlabel] at ({\x+.5},-3.7) {\t};
 					}

 					\begin{scope}[
 						every node/.style={treenode},
 						]
 						\node (root) at (0,0) {};
 						\node (1) at (1,0) {};
 						\node (11) at (2,0) {};
 						\node (111) at (3,0) {};
 						\node (1111) at (4,1.25) {};
 						\node (1114) at (4,-2.5) {};
 						\node (11111) at (5,2.5) {};
 						\node (11114) at (5,0) {};
 						\node (11141) at (5,-2.5) {};
 						\node (111111) at (6,2.5) {};
 						\node (111141) at (6,0) {};
 						\node (111411) at (6,-1.8) {};
 						\node (111414) at (6,-2.5) {};
 						\node (111417) at (6,-3.2) {};
 						\node (1111111) at (7,3.2) {};
 						\node (1111114) at (7,2.5) {};
 						\node (1111117) at (7,1.8) {};
 						\node (1111411) at (7,0) {};
 						\node (1114111) at (7,-1.8) {};
 						\node (1114141) at (7,-2.5) {};
 						\node (1114171) at (7,-3.2) {};
 					\end{scope}
 					
 					\begin{scope}[
 						->,
 						every node/.style={edgelabel},
 						]
 						\draw (root) -- (1);
 						\draw[line width=.5mm] (1) -- (11);
 						\draw[line width=.5mm] (11) -- (111);
 						\draw (111) -- (1111);
 						\draw (111) -- (1114);
 						\draw (1111) -- (11111);
 						\draw (1111) -- (11114);
 						\draw (1114) -- (11141);
 						\draw (11111) -- (111111);
 						\draw (11114) -- (111141);
 						\draw (11141) -- (111411);
 						\draw (11141) -- (111414);
 						\draw (11141) -- (111417);
 						\draw (111111) -- (1111111);
 						\draw (111111) -- (1111114);
 						\draw (111111) -- (1111117);
 						\draw (111141) -- (1111411);
 						\draw (111411) -- (1114111);
 						\draw (111414) -- (1114141);
 						\draw (111417) -- (1114171);
 					\end{scope}
 				\end{scope}
 			\end{tikzpicture}
 			\caption{Transforming the tree $T_1$ into the tree $T_2$.}
 			\label{T(X): figure, T_1 --> T_2}
 		\end{center}
 	\end{figure}
 	
 	Finally, we just need to relabel the arcs and vertices of the tree $T_2$. Figure~\ref{T(X): figure, T_N} shows the labelled tree obtained from $T_2$.
 	
 	\begin{figure}[h!bt]
 		\begin{center}
 			
 			\begin{tikzpicture}[x=8mm,y=8mm]
 				
 				\begin{scope}[
 					every node/.style={treenode},
 					]
 					\node (root) at (0,0) {};
 					\node (1) at (1,0) {};
 					\node (11) at (2,0) {};
 					\node (111) at (3,0) {};
 					\node (1111) at (4,1.25) {};
 					\node (1114) at (4,-2.5) {};
 					\node (11111) at (5,2.5) {};
 					\node (11114) at (5,0) {};
 					\node (11141) at (5,-2.5) {};
 					\node (111111) at (6,2.5) {};
 					\node (111141) at (6,0) {};
 					\node (111411) at (6,-1.8) {};
 					\node (111414) at (6,-2.5) {};
 					\node (111417) at (6,-3.2) {};
 					\node (1111111) at (7,3.2) {};
 					\node (1111114) at (7,2.5) {};
 					\node (1111117) at (7,1.8) {};
 					\node (1111411) at (7,0) {};
 					\node (1114111) at (7,-1.8) {};
 					\node (1114141) at (7,-2.5) {};
 					\node (1114171) at (7,-3.2) {};
 				\end{scope}
 				
 				\node[anchor=east] at (root) {\footnotesize $\varepsilon$};
 				\node[anchor=west, rotate=90] at (1) {\footnotesize $1$};
 				\node[anchor=west, rotate=90] at (11) {\footnotesize $11$};
 				\node[anchor=west, rotate=90] at (111) {\footnotesize $111$};
 				\node[anchor=west, rotate=90] at (1111) {\footnotesize $1111$};
 				\node[anchor=west, rotate=90] at (11111) {\footnotesize $11111$};
 				\node[anchor=west, rotate=90] at (111111) {\footnotesize $111111$};
 				\node[anchor=west] at (1111111) {\footnotesize $1111111$};
 				\node[anchor=west] at (1111114) {\footnotesize $1111114$};
 				\node[anchor=west] at (1111117) {\footnotesize $1111117$};
 				\node[anchor=west, rotate=90] at (11114) {\footnotesize $11114$};
 				\node[anchor=west, rotate=90] at (111141) {\footnotesize $111141$};
 				\node[anchor=west] at (1111411) {\footnotesize $1111411$};
 				\node[anchor=east, rotate=90] at (1114) {\footnotesize $1114$};
 				\node[anchor=east, rotate=90] at (11141) {\footnotesize $11141$};
 				\node[anchor=west, rotate=90] at (111411) {\footnotesize $111411$};
 				\node[anchor=east, rotate=90] at (111417) {\footnotesize $111417$};
 				\node[anchor=west] at (1114111) {\footnotesize $1114111$};
 				\node[anchor=west] at (1114141) {\footnotesize $1114141$};
 				\node[anchor=west] at (1114171) {\footnotesize $1114171$};
 				
 				\node[anchor=west, rotate=30] (label) at (6.5,-4.5) {\footnotesize $111414$};
 				\draw[very thick, densely dotted] (111414) edge[bend right=15] (label);

 				\begin{scope}[
 					->,
 					every node/.style={edgelabel},
 					]
 					\draw (root) -- node[above] {$1$} (1);
 					\draw (1) -- node[above] {$1$} (11);
 					\draw (11) -- node[above] {$1$} (111);
 					\draw (111) -- node[anchor=south east] {$1$} (1111);
 					\draw (111) -- node[anchor=north east] {$4$} (1114);
 					\draw (1111) -- node[anchor=south east] {$1$} (11111);
 					\draw (1111) -- node[anchor=north east] {$4$} (11114);
 					\draw (1114) -- node[below] {$1$} (11141);
 					\draw (11111) -- node[above] {$1$} (111111);
 					\draw (11114) -- node[above] {$1$} (111141);
 					\draw (11141) -- node[anchor=south east] {$1$} (111411);
 					\draw (11141) -- node[pos=0.7, above] {$4$} (111414);
 					\draw (11141) -- node[anchor=north east] {$7$} (111417);
 					\draw (111111) -- node[anchor=south east] {$1$} (1111111);
 					\draw (111111) -- node[pos=0.7, above] {$4$} (1111114);
 					\draw (111111) -- node[anchor=north east] {$7$} (1111117);
 					\draw (111141) -- node[above] {$1$} (1111411);
 					\draw (111411) -- node[above] {$1$} (1114111);
 					\draw (111414) -- node[above] {$1$} (1114141);
 					\draw (111417) -- node[above] {$1$} (1114171);
 				\end{scope}
 			\end{tikzpicture}
 			\caption{Tree (of a null semigroup) obtained after all the modifications.}
 			\label{T(X): figure, T_N}
 		\end{center}
 	\end{figure}
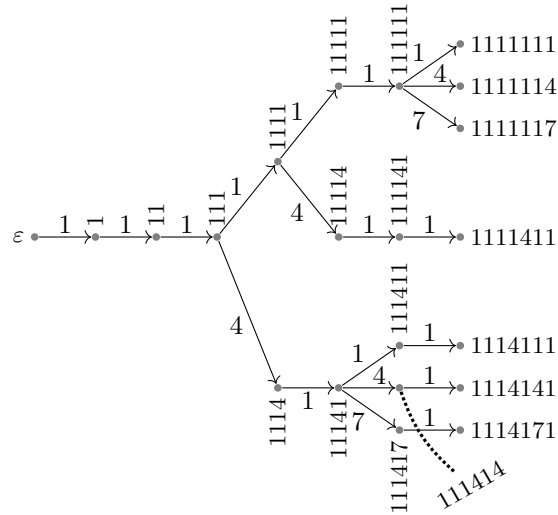

 	This new tree gives us the set of words
 	\begin{align*}
 		Z &=\set{1111111, 1111114, 1111117, 1111411, 1114111, 1114141, 1114171}\\
 		&\subseteq \set{1,2,3,4,5,6,7}^*.
 	\end{align*}
 	Using the words from $Z$ and the order $1,4,7,3,5,2,6$, we obtain the transformations below.
 	\begin{align*}
 		&\begin{pmatrix} 1&2&3&4&5&6&7 \\ 1&1&1&1&1&1&1\end{pmatrix}  &&\begin{pmatrix}1&2&3&4&5&6&7 \\ 1&1&1&1&4&1&1\end{pmatrix} &\begin{pmatrix}1&2&3&4&5&6&7 \\ 1&1&4&1&1&1&1\end{pmatrix}\\
 		&\begin{pmatrix} 1&2&3&4&5&6&7 \\ 1&1&1&1&1&4&1\end{pmatrix} &&&\begin{pmatrix}1&2&3&4&5&6&7 \\ 1&4&4&1&1&1&1\end{pmatrix}\\
 		&\begin{pmatrix} 1&2&3&4&5&6&7 \\ 1&1&1&1&1&7&1\end{pmatrix} &&&\begin{pmatrix}1&2&3&4&5&6&7 \\ 1&7&4&1&1&1&1\end{pmatrix}
 	\end{align*}
 	We can easily check that the product of any two transformations is equal to the top-leftmost transformation, which is the zero of this new semigroup. Hence we obtained a null subsemigroup of $\Tr{7}$ with the same number of elements as $S$. Additionally, we can easily verify that the tree in Figure~\ref{T(X): figure, T_N} is the tree of this null semigroup (when we use the order $1,4,7,3,5,2,6$ to construct it).
 \end{example}
 
 Our next goal is to prove that the largest commutative transformation semigroups with one idempotent are either groups or null semigroups. This is demonstrated in Theorem~\ref{T(X): largest comm smg 1 idemp}.  In order to prove it, we need one more result, which we present below.
 
 \begin{lemma}\label{T(X): |im e|=1}
 	Suppose that $\abs{X}=4$. Let $S$ be a commutative subsemigroup of $\tr{X}$ whose unique idempotent is $e \in S$. If $\abs{\im e}\in\set{2,3}$, then $\abs{S}<4$.
 \end{lemma}
 
 \begin{proof}
 	Assume, without loss of generality, that $X=\set{1,2,3,4}$.
 	
 	Let $G=\gset{\beta|_{\im e}}{\beta \in S}$. It follows from Lemma~\ref{T(X): tr commutes with idemp} that $\beta|_{\im e}\in \tr{\im e}$ for all $\beta\in S$. Hence, since $S$ is a commutative subsemigroup of $\tr{X}$ whose unique idempotent is $e$, then Lemma~\ref{T(X): lemma induction} guarantees that $G$ is a commutative subsemigroup of $\tr{\im e}$ whose unique idempotent is $e|_{\im e}=\id{\im e}$. Thus, by Proposition~\ref{T(X): unique idemp is id_X => group in S(X)}, $G$ is an abelian subgroup of $\sym{\im e}$.
 	
 	\medskip
 	
 	\textbf{Part 1.} Suppose that $\abs{\im e}=2$. Assume, without loss of generality, that $\im e=\set{1,4}$. Then there exists a partition $\set{B_1,B_4}$ of $X$ such that $e=\chain{B_1, 1}\chain{B_4, 4}$. (Note that $1\in B_1$ and $4\in B_4$.)
 	
 	Let $\set{A_j}_{j=0}^k$ be the $S$-partition of $X$. We have that $A_0=\im e=\set{1,4}$. Furthermore, $2,3\notin\im e$, which implies that $k\geqslant 1$ and, consequently, $2\in A_1$ or $3\in A_1$. Assume, without loss of generality, that $2\in A_1$. Then, for all $\beta\in S$ we have $2\beta\in\im e=\set{1,4}$. Additionally, there exists $i\in\X{k}$ such that $3\in A_i$, which implies that $3\beta\in\bigcup_{j=0}^{i-1}A_j$ for all $\beta\in S$ and, consequently, that $3\beta\neq 3$ for all $\beta\in S$.
 	
 	We have $\abs{B_1}=3$ and $\abs{B_4}=1$, or $\abs{B_1}=1$ and $\abs{B_4}=3$, or $\abs{B_1}=\abs{B_4}=2$.
 	
 	\smallskip
 	
 	\textit{Case 1:} Suppose that $\abs{B_1}=3$ and $\abs{B_4}=1$. Then $B_1=\set{1,2,3}$ and $B_4=\set{4}$ and, thus,
 	\begin{displaymath}
 		e=\begin{pmatrix}
 			1&2&3&4\\
 			1&1&1&4
 		\end{pmatrix}.
 	\end{displaymath}
 	
 	Let
 	\begin{displaymath}
 		\beta_1=\begin{pmatrix}
 			1&2&3&4\\
 			1&1&2&4
 		\end{pmatrix} \quad \text{and} \quad
 		\beta_2=\begin{pmatrix}
 			1&2&3&4\\
 			4&4&4&1
 		\end{pmatrix}.
 	\end{displaymath}
 	
 	Let $\beta\in S$. Since $\beta|_{\set{1,4}}=\beta|_{\im e}\in G\subseteq\sym{\im e}=\sym{\set{1,4}}$, then we have either $1\beta=1$ and $4\beta=4$ or $1\beta=4$ and $4\beta=1$.
 	
 	\smallskip
 	
 	\textsc{Sub-case 1:} Suppose that $1\beta=1$ and $4\beta=4$. It follows from Lemma~\ref{T(X): tr commutes with idemp} that $\set{2,3}\beta\subseteq B_1\beta\subseteq B_1=\set{1,2,3}$. We also have $2\beta\in\set{1,4}$ and $3\beta\neq 3$, which implies that $2\beta=1$ and $3\beta\in\set{1,2}$. Thus $\beta\in\set{e,\beta_1}$.
 	
 	\smallskip
 	
 	\textsc{Sub-case 2:} Suppose that $1\beta=4$ and $4\beta=1$. Then Lemma~\ref{T(X): tr commutes with idemp} implies that $\set{2,3}\beta\subseteq B_1\beta\subseteq B_4=\set{4}$. Hence $\beta=\beta_2$.
 	
 	\smallskip
 	
 	Since $\beta$ is an arbitrary element of $S$, then we can conclude that $S\subseteq\set{e,\beta_1,\beta_2}$. Therefore $\abs{S}<4$.
 	
 	\smallskip
 	
 	\textit{Case 2:} Suppose that $\abs{B_1}=1$ and $\abs{B_4}=3$. We can prove in a parallel way to case 1 that $\abs{S}<4$.
 	
 	\smallskip
 	
 	\textit{Case 3:} Suppose that $\abs{B_1}=\abs{B_4}=2$. Assume, without loss of generality, that $B_1=\set{1,2}$ and $B_4=\set{3,4}$. Then
 	\begin{displaymath}
 		e=\begin{pmatrix}
 			1&2&3&4\\
 			1&1&4&4
 		\end{pmatrix}.
 	\end{displaymath}
 	Let
 	\begin{displaymath}
 		\beta_1=\begin{pmatrix}
 			1&2&3&4\\
 			4&4&1&1
 		\end{pmatrix} \quad \text{and} \quad
 		\beta_2=\begin{pmatrix}
 			1&2&3&4\\
 			4&4&2&1
 		\end{pmatrix}.
 	\end{displaymath}
 	
 	Let $\beta\in S$. We have either $1\beta=1$ and $4\beta=4$ or $1\beta=4$ and $4\beta=1$.
 	
 	\smallskip
 	
 	\textsc{Sub-case 1:} Suppose that $1\beta=1$ and $4\beta=4$. Then, by Lemma~\ref{T(X): tr commutes with idemp}, we have that $2\beta\in B_1\beta\subseteq B_1=\set{1,2}$ and $3\beta\in B_4\beta\subseteq B_4=\set{3,4}$. In addition, we have that $2\beta\in \set{1,4}$ and $3\beta\neq 3$, which implies that $2\beta=1$ and $3\beta=4$. Thus $\beta=e$.
 	
 	\smallskip
 	
 	\textsc{Sub-case 2:} Suppose that $1\beta=4$ and $4\beta=1$. As a consequence of Lemma~\ref{T(X): tr commutes with idemp} we have that $2\beta\in B_1\beta\subseteq B_4=\set{3,4}$ and $3\beta\in B_4\beta\subseteq B_1=\set{1,2}$. Since we also have $2\beta\in \set{1,4}$, then we can conclude that $2\beta=4$. Hence $\beta\in\set{\beta_1,\beta_2}$.
 	
 	\smallskip
 	
 	We just proved that $S\subseteq\set{e,\beta_1,\beta_2}$. Therefore $\abs{S}<4$.
 	
 	\medskip
 	
 	\textbf{Part 2.} Suppose that $\abs{\im e}=3$. Let $B_1=\set{1,2}$, $B_3=\set{3}$ and $B_4=\set{4}$. Assume, without lost of generality, that
 	\begin{displaymath}
 		e=\begin{pmatrix}
 			1&2&3&4\\
 			1&1&3&4
 		\end{pmatrix}=\chain{B_1, 1}\chain{B_3, 3}\chain{B_4, 4}.
 	\end{displaymath}
 	
 	Let $\gamma\in G$ and let $\gamma'\in S$ be such that $\gamma'|_{\im e}=\gamma$. It is clear that $\gamma$ determines $\gamma'$ in $\im e=X\setminus\set{2}$. In what follows we will see that $\gamma$ also determines $\gamma'$ in $2$. Let $i\in X$ be such that $1\gamma'=1\gamma=i$. We note that, since $G\subseteq\sym{\im e}$, then $i\in\im e=\set{1,3,4}$. By Lemma~\ref{T(X): tr commutes with idemp}, we have that $2\gamma'\in B_1\gamma'\subseteq B_i$. In addition, Lemma~\ref{T(X): exists x not in union of images} implies that $\bigcup_{\beta\in S}\im \beta\subsetneq X$. Then, since $X\setminus\set{2}=\im e$, we must have $\bigcup_{\beta\in S}\im \beta=X\setminus\set{2}$. Consequently, $2\gamma'\in B_i\cap \parens{X\setminus\set{2}} = B_i\setminus\set{2}$, which is a singleton.
 	
 	We just proved that there is a one-to-one correspondence between the elements of $G$ and the elements of $S$. Thus $\abs{S}=\abs{G}$. Furthermore, Theorem~\ref{T(X): maximum size abelian subgroup of S(X)} guarantees that $\abs{G}\leqslant 3$ and, consequently, we have $\abs{S}<4$.
 \end{proof}
 
 \begin{theorem}\label{T(X): largest comm smg 1 idemp}
 	The maximum size of a commutative subsemigroup of $\tr{X}$ with a unique idempotent is
 	\begin{displaymath}
 		\begin{cases}
 			\abs{X}& \text{if } \abs{X}\leqslant 4,\\
 			\maxnull{\abs{X}}& \text{if } \abs{X}\geqslant 5.
 		\end{cases}
 	\end{displaymath}
 	Moreover, if $S$ is a maximum-order commutative subsemigroup of $\tr{X}$ with a unique idempotent, then
 	\begin{enumerate}
 		\item If $\abs{X}\leqslant 3$, then $C_{\abs{X}}\simeq S\subseteq\sym{X}$.
 		
 		\item If $\abs{X}=4$, then either $C_4\simeq S\subseteq\sym{X}$, or $ C_2\times C_2\simeq S\subseteq\sym{X}$, or $S=N^{X}_{x_1,x_2}$ for some distinct $x_1,x_2\in X$.
 		
 		\item If $\abs{X}\geqslant 5$, then $S=\nulltr{X}{x_1}{x_{\alphanull{\abs{X}}}}$, for some pairwise distinct $x_1,\ldots,x_{\alphanull{\abs{X}}}\in X$.
 	\end{enumerate}
 \end{theorem}
 
 \begin{proof}
 	We partition the class of commutative subsemigroups of $\tr{X}$ with a unique idempotent into two classes. One of the classes, which we denote by $\mathcal{C}_1$, comprises the semigroups whose unique idempotent is $\id{X}$, and the other class, which we denote by $\mathcal{C}_2$, comprises the semigroups whose unique idempotent is not $\id{X}$. We begin by determining in which one of these classes the largest commutative subsemigroups of $\tr{X}$ with a unique idempotent lie.
 	
 	We have that all abelian subgroups of $\sym{X}$ are commutative subsemigroup of $\tr{X}$ whose unique idempotent is $\id{X}$. Hence $\mathcal{C}_1$ contains all the abelian subgroups of $\sym{X}$. Furthermore, Proposition~\ref{T(X): unique idemp is id_X => group in S(X)} guarantees that all the semigroups of the class $\mathcal{C}_1$ are abelian subgroups of $\sym{X}$. Thus $\mathcal{C}_1$ is the class of abelian subgroups of $\sym{X}$.
 	
 	We have that, when $\abs{X}=1$, the class $\mathcal{C}_2$ contains no semigroups. Assume that $\abs{X}\geqslant 2$. Let $x_1,\ldots,x_{\alphanull{\abs{X}}}\in X$ be pairwise distinct elements. We have that the zero of the null semigroup $\nulltr{X}{x_1}{x_{\alphanull{\abs{X}}}}$ has rank 1, which implies that its unique idempotent is not $\id{X}$. Hence $\nulltr{X}{x_1}{x_{\alphanull{\abs{X}}}}\in\mathcal{C}_2$. It follows from Theorem~\ref{null semigroups of maximum size} that $\abs{\nulltr{X}{x_1}{x_{\alphanull{\abs{X}}}}}=\maxnull{\abs{X}}$ and, consequently, $\mathcal{C}_2$ contains semigroups of size $\maxnull{\abs{X}}$. Moreover, Theorem~\ref{|S|=|N|, idemp distinct id_X} guarantees that the size of each semigroup in $\mathcal{C}_2$ is equal to the size of some null subsemigroup of $\tr{X}$ and, since the maximum size of a null subsemigroup of $\tr{X}$ is $\maxnull{\abs{X}}$ (by Theorem~\ref{maximum size null semigroup}), then we can conclude that the maximum size of a semigroup in $\mathcal{C}_2$ is $\maxnull{\abs{X}}$.
 	
 	With this in mind, we consider the following cases, where we ascertain which classes ($\mathcal{C}_1$ or $\mathcal{C}_2$) contain maximum-order commutative subsemigroups of $\tr{X}$ with a unique idempotent.
 	
 	\smallskip
 	
 	\textit{Case 1:} Assume that $\abs{X}\leqslant 3$. If $\abs{X}=1$, then the class $\mathcal{C}_2$ is empty and, consequently, the class $\mathcal{C}_1$ contains the unique maximum-order subsemigroup of $\tr{X}$ with a unique idempotent, which is isomorphic to $C_1$. Now assume that $\abs{X}\in\set{2,3}$. We have that $2>1=\maxnull{2}$ and $3>2=\maxnull{3}$, which implies that $\abs{X}>\maxnull{\abs{X}}$. Due to the fact that $\abs{X}$ is the maximum size of a semigroup in $\mathcal{C}_1$ (by Theorem~\ref{T(X): maximum size abelian subgroup of S(X)}) and $\maxnull{\abs{X}}$ is the maximum size of a semigroup in $\mathcal{C}_2$, then we can conclude that the maximum-order commutative subsemigroups of $\tr{X}$ with a unique idempotent lie in $\mathcal{C}_1$ and, consequently, are isomorphic to $C_{\abs{X}}$ (by Theorem~\ref{T(X): maximum size abelian subgroup of S(X)}).
 	
 	\smallskip
 	
 	\textit{Case 2:} Assume that $\abs{X}=4$. By Theorem~\ref{T(X): maximum size abelian subgroup of S(X)} the maximum size of a semigroup in $\mathcal{C}_1$ is $4$, which is equal to $\maxnull{4}$, the maximum size of a semigroup in $\mathcal{C}_2$. Hence we have maximum-order commutative subsemigroups of $\tr{X}$ with a unique idempotent in $\mathcal{C}_1$ and in $\mathcal{C}_2$. It follows from Theorem~\ref{T(X): maximum size abelian subgroup of S(X)} that the ones that lie in $\mathcal{C}_1$ are either isomorphic to $C_4$ or $C_2\times C_2$. We will describe after the next case the ones that lie in $\mathcal{C}_2$.
 	
 	\smallskip
 	
 	\textit{Case 3:} Assume that $\abs{X}\geqslant 5$. Lemma~\ref{T(X): unique idemp, upper bound abelian subgroup S(X)} ensures that the size of any semigroup in $\mathcal{C}_1$ (an abelian subgroup of $\sym{X}$) is less than $\maxnull{\abs{X}}$, the maximum size of a semigroup in $\mathcal{C}_2$. Therefore all the maximum-order commutative subsemigroups of $\tr{X}$ with a unique idempotent lie in $\mathcal{C}_2$.
 	
 	\smallskip
 	
 	In order to conclude this proof, we only need to describe, when $\abs{X}\geqslant 4$, the largest semigroups in $\mathcal{C}_2$, that is, the semigroups in $\mathcal{C}_2$ of size $\maxnull{\abs{X}}$.
 	
 	Suppose that $\abs{X}\geqslant 4$. Let $S$ be a semigroup in $\mathcal{C}_2$ such that $\abs{S}=\maxnull{\abs{X}}$. Let $e\in S\setminus\set{\id{X}}$ be its unique idempotent. We are going to use the proof of Theorem~\ref{|S|=|N|, idemp distinct id_X} to prove that $S$ is a null semigroup.

 	Let $n=\abs{X}\geqslant 4$. Let $\set{A_i}_{j=0}^{k}$ be the $S$-partition of $X$ and let $x_1,\ldots,x_n$ be the order of the elements of $X$ used to construct the tree of $S$. Let $G=\gset{\beta|_{\im e}}{\beta \in S}$, which is an abelian subgroup of $\sym{\im e}$. Let $M$ be a null subsemigroup of $\nulltr{Y}{x_1}{x_t}$ such that $\abs{M}=\abs{G}$, where $Y=\set{x_1,\ldots,x_{\abs{\im e}+1}}$ and $t=\alphanull{\abs{Y}}$. Just like in the proof of Theorem~\ref{|S|=|N|, idemp distinct id_X}, we denote by $T'_M$ the tree obtained from $T_M$ by removing the first arc of the trunk, we denote by $T_1$ the tree obtained from $T_S$ by replacing $T_G$ with $T'_M$, and we denote by $T_2$ the tree obtained from $T_1$ by moving all the linear levels to its trunk. See \eqref{T(X): diagram proof tree} for a diagram showing the relationship between these trees.
 	
 	Let $N$ be the null subsemigroup of $\tr{X}$ (of size $\abs{S}$) obtained from $S$ by modifying $T_S$. It follows from the way we constructed $N$ that its zero (that is, its idempotent) has rank 1 and its image is equal to $\set{x_1}$. Furthermore, we know that there exists $r\in\mathbb{N}$ such that $\set{x_1,\ldots,x_r}\beta=\set{x_1}$ and $\im\beta\subseteq\set{x_1,\ldots,x_r}$ for all $\beta\in N$. Since $\abs{N}=\abs{S}=\maxnull{n}$, then, by Theorem~\ref{null semigroups of maximum size}, we have that $r=\alphanull{n}$ and $N=\nulltr{X}{x_1}{x_{\alphanull{n}}}$.
 	
 	
 	Now we will describe the tree of $N$ (constructed with the order $x_1,\ldots,x_n$ of the elements of $X$). The set of words determined by (the transformations of) $N$ is
 	\begin{displaymath}
 		W_N=\gset{x_1^{\alphanull{n}}y_1\cdots y_{n-\alphanull{n}}}{y_1,\ldots, y_{n-\alphanull{n}}\in\set{x_1,\ldots,x_{\alphanull{n}}}}
 	\end{displaymath}
 	and the vertex set of $T_N$ is the set of prefixes of the words of $W_N$.
 	
 	First we will verify that the starting vertices of the arcs of levels $x_{\alphanull{n}+1},\ldots,x_n$ have outdegree $\alphanull{n}$. Let $i\in\set{\alphanull{n}+1,\ldots,n}$ and let $w$ be the starting vertex of an arc of level $x_i$. Then $w$ is a word of length $i-1$. We have that $wx_1,wx_2,\ldots,wx_{\alphanull{n}}$ are precisely the vertices of length $i$ of $T_N$ that contain $w$ as a prefix. This implies that the ending vertices of the arcs whose starting vertex is $w$ are precisely $wx_1,wx_2,\ldots,wx_{\alphanull{n}}$ and, consequently, $w$ has outdegree $\alphanull{n}$. Since $w$ is an arbitrary vertex of length $i-1$ and $i$ is an arbitrary element of $\set{\alphanull{n}+1,\ldots,n}$, then we can conclude that the starting vertices of the arcs of levels $x_{\alphanull{n}+1},\ldots,x_n$ have outdegree $\alphanull{n}$, that is, a branching with $\alphanull{n}$ arcs occurs at the starting vertices of the arcs of levels $x_{\alphanull{n}+1},\ldots,x_n$. Note that this also implies that levels $x_{\alphanull{n}+1},\ldots,x_n$ are branching levels.
 	
 	Now we will see that $T_N$ contains a trunk of length $\alphanull{n}$. It is straightforward to see that $x_1,x_1^2,\ldots, x_1^{\alphanull{n}}$ are prefixes of all the words of $W_N$, which implies that for all $i\in\X{\alphanull{n}}$ the only vertex of length $i$ of $T_N$ is $x_1^i$. This implies that the subgraph of $T_N$ located at levels $x_1,\ldots,x_{\alphanull{n}}$ is a path of length $\alphanull{n}$ starting at vertex $\varepsilon$ and ending at vertex $x_1^{\alphanull{n}}$. Hence $x_1,\ldots,x_{\alphanull{n}}$ are linear and $T_N$ contains a trunk of length $\alphanull{n}$ (we note that the fact that level $x_{\alphanull{n}+1}$ is a branching level implies that the trunk of $T_N$ has at most $\alphanull{n}$ arcs).
 	
 	It follows from the previous two paragraphs that $T_N$ has $\alphanull{n}$ linear levels (namely, levels $x_1,\ldots,x_{\alphanull{n}}$) which are all associated with the trunk of $T_N$, and $T_N$ has $n-\alphanull{n}$ branching levels (namely, levels $x_{\alphanull{n}+1},\ldots,x_n$).
 	
 	
 	Notice that, since $T_N$ was obtained from $T_2$ simply by adding labels to its arcs and renaming its vertices, then the trees $T_2$ and $T_N$ have the same structure. Therefore, $T_2$ shares with $T_N$ all the properties we mentioned in the previous three paragraphs.
 	
 	We have that $T_2$ was obtained from $T_1$ by moving all its linear levels, that were not in the trunk, to the trunk of the tree (assuming that there were any linear levels outside the trunk of $T_1$). This means that, in the process of transforming the tree $T_1$ into the tree $T_2$, we do not change the number of linear levels. Therefore $T_1$ and $T_2$ have the same number of linear levels, which is equal to $\alphanull{n}$. Furthermore, this process also does not change the content of the branching levels, which implies that every branching of $T_1$ has $\alphanull{n}$ arcs.

 	Now we are going to establish that $\abs{G}<\maxnull{n}=\abs{S}$. We recall that $G$ is an abelian subgroup of $\sym{\im e}$.
 	
 	\smallskip
 	
 	\textit{Case 1:} Assume that $\abs{\im e}\leqslant 3$. Then Theorem~\ref{T(X): maximum size abelian subgroup of S(X)} implies that $\abs{G}\leqslant \abs{\im e}<4=\maxnull{4}$. Due to the fact that $n=\abs{X}\geqslant 4$, and by Lemma~\ref{inequalities xi}, we have that $\maxnull{4}\leqslant\maxnull{n}$ and, consequently, that $\abs{G}<\maxnull{n}$.
 	
 	\smallskip
 	
 	\textit{Case 2:} Assume that $\abs{\im e}=4$. Since $e\neq\id{X}$, then we have that $\abs{\im e}<n$, which implies that $n>4$. Hence, by Theorem~\ref{T(X): maximum size abelian subgroup of S(X)} and Lemma~\ref{inequalities xi} we have that $\abs{G}\leqslant 4=\maxnull{4}<\maxnull{n}$.
 	
 	\smallskip
 	
 	\textit{Case 3:} Assume that $\abs{\im e}\geqslant 5$. Since $e\neq\id{X}$, then we have that $\abs{\im e}<n$. As a consequence of Lemma~\ref{T(X): unique idemp, upper bound abelian subgroup S(X)} and Lemma~\ref{inequalities xi} we have that $\abs{G}<\maxnull{\abs{\im e}}<\maxnull{n}$.
 	
 	\smallskip
 	
 	We just proved that $\abs{G}<\maxnull{n}=\abs{S}$. Then, by part 2 of Lemma~\ref{T(X): properties tree of S}, we have that the number of leaves of $T_G$ is smaller than the number of leaves of $T_S$. Due to the fact that $T_G$ is the subgraph of $T_S$ located at the levels $x_1,\ldots,x_{\abs{\im e}}$, then we can conclude that it is not possible for the remaining levels $x_{\abs{\im e}+1},\ldots,x_n$ of $T_S$ to be all linear, that is, among the levels $x_{\abs{\im e}+1},\ldots,x_n$ of $T_S$ there exist branching levels. Let $i\in\set{\abs{\im e}+1,\ldots,n}$ be the minimal element such that level $x_i$ is a branching level of $T_S$ (that is, $x_i$ is the leftmost branching level among the levels $x_{\abs{\im e}+1},\ldots,x_n$ of $T_S$). Since that branching is not in $T_G$, then, when we replace $T_G$ by $T'_M$, and obtain the tree $T_1$, the branching is not deleted. Hence the branching is also in $T_1$ (and outside $T'_M$). We have that all the branchings of $T_1$ have $\alphanull{n}$ arcs, which implies that the branching of $T_S$ we are referring to also has $\alphanull{n}$ arcs. Furthermore, it follows from Lemma~\ref{i_1,...,i_s<i, i_2,...,i_s linear levels} that there exist at least $\alphanull{n}-1$ linear levels among the levels $x_{\abs{\im e}+1},\ldots,x_{i-1}$ of $T_S$. This means that these $\alphanull{n}-1$ linear levels are not levels of $T_G$ and, consequently, they are also not levels of $T'_M$, which implies that they remain unaltered when we modify $T_S$ to obtain $T_1$. Hence those $\alphanull{n}-1$ linear levels of $T_S$ are also linear levels of $T_1$. Due to the fact that $T_1$ has $\alphanull{n}$ linear levels, $\alphanull{n}-1$ of which are not levels of $T'_M$, and $T'_M$ contains at least one linear level (located at its trunk), then we can conclude that, among the $\alphanull{n}$ linear levels of $T_1$, there is exactly one that is a linear level of $T'_M$ and the remaining $\alphanull{n}-1$ linear levels of $T_1$ are located outside $T'_M$. This implies that among the levels $x_{\abs{\im e}+1},\ldots,x_{n}$ of $T_S$ there are exactly $\alphanull{n}-1$ linear levels. Since among the levels $x_{\abs{\im e}+1},\ldots,x_{i-1}$ of $T_S$ there are at least $\alphanull{n}-1$ linear levels, then we can conclude that among the levels $x_{\abs{\im e}+1},\ldots,x_{i-1}$ of $T_S$ there are exactly $\alphanull{n}-1$ linear levels, and the levels $x_i,\ldots,x_n$ are all branching levels. Moreover, we know that $x_i$ is the leftmost branching level among the levels $x_{\abs{\im e}+1},\ldots,x_n$ of $T_S$, which implies that $i=\abs{\im e}+\alphanull{n}$. Therefore the $\alphanull{n}-1$ linear levels of $T_S$ are the levels $x_{\abs{\im e}+1},\ldots,x_{\abs{\im e}+\alphanull{n}-1}$, and the levels $x_{\abs{\im e}+\alphanull{n}},\ldots,x_n$ of $T_S$ are branching levels.


 	
 	
 	Now our goal is to show that $T_G$ has only one level (that is, that $\im e=\set{x_1}$). In order to do this, we separate the proof into two cases.
 	
 	\smallskip
 	
 	\textit{Case 1:} Assume that $n=\abs{X}=4$. Since $e\neq\id{X}$, then we have $\abs{\im e}<n=4$. Furthermore, due to the fact that $\abs{S}=\maxnull{n}=\maxnull{4}=4$, and by Lemma~\ref{T(X): |im e|=1}, we have that $\abs{\im e}\notin\set{2,3}$. Thus $\abs{\im e}=1$ and, consequently, $T_G$ has only one level --- level $x_1$.
 	
 	\smallskip
 	
 	\textit{Case 2:} Assume that $n=\abs{X}\geqslant 5$. We have that
 	\begin{gather*}
 		1^{n-1}=1<3^2\cdot 3^{n-5}=3^{n-3}\leqslant\maxnull{n}\\
 		\shortintertext{and}
 		2^{n-2}=2^3\cdot 2^{n-5}<3^2\cdot 3^{n-5}=3^{n-3}\leqslant\maxnull{n},
 	\end{gather*}
 	which implies that $\alphanull{n}\geqslant 3$. This means that every branching of $T_1$ has $\alphanull{n}\geqslant 3$ arcs. In addition, we have that the trunk of $T'_M$ has length $1$ (because the levels of a trunk are all linear and $T'_M$ contains exactly one linear level) and, consequently, the trunk of $T_M$ has length $2$ (recall that $T'_M$ is obtained from $T_M$ by removing one of the arcs of its trunk). Moreover, we saw in the proof of Theorem~\ref{|S|=|N|, idemp distinct id_X} that, if $T_M$ contains branchings, then the length of the trunk of $T_M$ (which is equal to $2$) is an upper bound for the number of arcs of any branching of $T_M$. Hence, if $T_M$ contains branchings, then they will all have exactly $2$ arcs. We know that any branching of $T_M$ is also going to be a branching of $T'_M$ and, consequently, a branching of $T_1$ (we note that $T'_M$ is a subgraph of $T_1$). This allows us to conclude that, if $T_M$ contains branchings, then $T_1$ contains branchings with $2$ arcs, which is not possible. Hence $T_M$ contains no branchings and, consequently, neither does $T'_M$. This implies that $T'_M$ has just one level (which is linear). Since $T_G$ and $T'_M$ have the same number of levels, then we conclude that $T_G$ has only one level --- level $x_1$.
 	
 	\smallskip

 	In both cases we established that $T_G$ contains only one level. This implies that $\abs{\im e}=1$ and, consequently, that $\abs{G}=1$. Hence $T_G$ has only one leaf (by part 2 of Lemma~\ref{T(X): properties tree of S}), which implies that $T_G$ is just a path of length $1$. As a consequence of the fact that $T'_M$ has the same number of levels and the same number of leaves as $T_G$, we have that $T'_M$ is also a path of length $1$. Thus $T_G$ and $T'_M$ have the same structure and, consequently, $T_S$ and $T_1$ also have the same structure (recall that $T_1$ is obtained from $T_S$ by replacing its subgraph $T_G$ by $T'_M$).
 	
 	We have that level $x_1$ of $T_S$ (the unique level of $T_G$) is linear. Furthermore, we also know that levels $x_{\abs{\im e}+1},\ldots,x_{\abs{\im e}+\alphanull{n}-1}$ of $T_S$ are linear, that is, levels $x_2,\ldots,x_{\alphanull{n}}$ of $T_S$ are linear (recall that $\abs{\im e}=1$), and we know that levels $x_{\abs{\im e}+\alphanull{n}},\ldots,x_n$ of $T_S$ are branching levels. Hence $T_S$ and $T_1$ have a trunk of length $\alphanull{n}$ and all their linear levels are associated with the trunk. This means that we do not make any modifications in the tree $T_1$ in order to obtain $T_2$ (because there are no linear levels outside the trunk of $T_1$). Thus $T_S$, $T_1$, $T_2$ and $T_N$ all have the same structure. Therefore $T_S$ has a trunk of length $\alphanull{n}$, all the linear levels of $T_S$ are associated with its trunk, and a branching with $(n)\alpha$ arcs occurs in all vertices that are words of length between $(n)\alpha$ and $n-1$.
 	
 	Now we are going to see what the labels of the arcs of $T_S$ look like. We begin by considering a branching of $T_S$ (which is located at one of the levels $x_{\alphanull{n}+1},\ldots,x_n$). This branching has $(n)\alpha$ arcs. Let $x_{i_1},\ldots,x_{i_{\alphanull{n}}}$ (where $i_1<i_2<\cdots<i_{\alphanull{n}}$) be their labels. According to Lemma~\ref{i_1,...,i_s<i, i_2,...,i_s linear levels}, $x_{i_2},\ldots,x_{i_{\alphanull{n}}}$ are linear levels of $T_S$. Since level $x_{i_1}$ precedes those $x_{i_2},\ldots,x_{i_{\alphanull{n}}}$ linear levels and $T_S$ has exactly $\alphanull{n}$ linear levels (levels $x_1,\ldots,x_{\alphanull{n}}$) then we must have $x_{i_m}=x_m$ for all $m\in\X{\alphanull{n}}$. Since we took an arbitrary branching of $T_S$, then we can conclude that all the branchings of $T_S$ are labelled with $x_1,\ldots,x_{\alphanull{n}}$. Now we will see how the arcs of the trunk of $T_S$ are labelled. For each $m\in\X{\alphanull{n}}$ let $y_m$ be the label of the arc of the trunk located at level $x_m$. It follows from part 2 of Lemma~\ref{T(X): labels of T_S}, and the fact that $T_S$ contains branchings at level $x_{\alphanull{n}+1}$ whose arcs have labels $x_1,\ldots,x_{\alphanull{n}}$, that there exist $\beta_1,\ldots,\beta_{\alphanull{n}}\in S$ which are equal in $\set{x_1,\ldots,x_{\alphanull{n}}}$ and such that $x_m=x_{\alphanull{n}+1}\beta_m$ for all $m\in\X{\alphanull{n}}$. Let $l\in\X{k}$ be such that $x_{\alphanull{n}+1}\in A_l$ (we observe that $x_{\alphanull{n}+1}\notin\set{x_1}=\im e=A_0$). Since in the sequence $x_1,\ldots,x_n$ the elements of $A_j$ precede the elements of $A_{j+1}$ for all $j\in\set{0,\ldots,k-1}$, then we have that the elements of $\bigcup_{j=0}^{l-1}A_j$ precede $x_{\alphanull{n}+1}$. Hence $\bigcup_{j=0}^{l-1}A_j\subseteq\set{x_1,\ldots,x_{\alphanull{n}}}$ and, consequently, $\beta_1,\ldots,\beta_s\in S$ are equal in $\bigcup_{j=0}^{l-1}A_j$. Therefore Lemma~\ref{T(X): main lemma} implies that for all $\gamma\in S$ we have
 	\begin{displaymath}
 		\parens{x_{\alphanull{n}+1}\beta_1}\gamma=\parens{x_{\alphanull{n}+1}\beta_2}\gamma=\cdots=\parens{x_{\alphanull{n}+1}\beta_{\alphanull{n}}}\gamma,
 	\end{displaymath}
 	that is, $x_1\gamma=x_2\gamma=\cdots=x_{\alphanull{n}}\gamma$. Additionally, Lemma~\ref{T(X): tr commutes with idemp} guarantees that $x_1\gamma\in \im e=\set{x_1}$ for all $\gamma\in S$ and, consequently, we have $x_1=x_1\gamma=x_2\gamma=\cdots=x_{\alphanull{n}}\gamma$ for all $\gamma\in S$. By part 1 of Lemma~\ref{T(X): labels of T_S} we have that, for each $m\in\X{\alphanull{n}}$, there exists $\gamma_m\in S$ such that $y_m=x_m\gamma_m$. Hence $x_1=y_1=y_2=\cdots=y_{\alphanull{n}}$ and, thus, all the arcs of the trunk of $T_S$ have label $x_1$.
 	
 	
 	Finally, we are going to establish that $S=N=\nulltr{X}{x_1}{x_{\alphanull{n}}}$. Let $\beta\in S$. For all $i\in\Xn$ we have that $x_i\beta$ is equal to the $i$-th letter of the word $w_\beta$ determined by $\beta$, which is equal to the label of an arc of level $x_i$. Moreover, we have that levels $x_1,\ldots,x_{\alphanull{n}}$ have one arc each, whose label is $x_1$. Hence $\set{x_1,\ldots,x_{\alphanull{n}}}\beta=\set{x_1}$. In addition, we have that the labels of the arcs of levels $x_{\alphanull{n}+1},\ldots,x_n$ belong to $\set{x_1,\ldots,x_{\alphanull{n}}}$. This implies that $\set{x_{\alphanull{n}+1},\ldots,x_n}\beta\subseteq\set{x_1,\ldots,x_{\alphanull{n}}}$ and, consequently, we have that $\im\beta=\set{x_1,\ldots,x_{\alphanull{n}}}\beta\cup\set{x_{\alphanull{n}+1},\ldots,x_n}\beta\subseteq\set{x_1}\cup\set{x_1,\ldots,x_{\alphanull{n}}}=\set{x_1,\ldots,x_{\alphanull{n}}}$. Therefore $\beta\in \nulltr{X}{x_1}{x_{\alphanull{n}}}$. Since $\beta$ is an arbitrary transformation of $S$, then we conclude that $S\subseteq\nulltr{X}{x_1}{x_{\alphanull{n}}}$ and, since $\abs{S}=\abs{N}=\abs{\nulltr{X}{x_1}{x_{\alphanull{n}}}}$, we obtain that $S=\nulltr{X}{x_1}{x_{\alphanull{n}}}$, which concludes the proof.
 \end{proof}
 
 Now we will use the characterization of the maximum-order commutative transformation semigroups with a unique idempotent to prove, in an alternative way, that the maximum-order commutative nilpotent subsemigroups of $\tr{X}$ are all null semigroups. This result appeared in \cite[Theorem 3.12]{Commutative_nilpotent_transformation_semigroups_paper} but was proved directly, not deduced as a consequence of the more general results above.
 
 \begin{corollary}\label{T(X): maximum size nilpotent smg}
 	The maximum size of a commutative nilpotent subsemigroup of $\tr{X}$ is $\maxnull{\abs{X}}$. Moreover, $S$ is a commutative nilpotent subsemigroup of $\tr{X}$ of size $\maxnull{\abs{X}}$ if and only if at least one of the following conditions is satisfied:
 	\begin{enumerate}
 		\item There exist pairwise distinct $x_1,\ldots,x_t\in S$ such that $S=\nulltr{X}{x_1}{x_t}$, where $t=\alphanull{\abs{X}}$.
 		
 		\item $\abs{X}=2$ and $S=\set{\id{X}}$.
 	\end{enumerate}
 \end{corollary}
 
 \begin{proof}
 	First we notice that commutative nilpotent semigroups have exactly one idempotent.
 	
 	Let $S$ be commutative nilpotent subsemigroup of $\tr{X}$ of maximum size. Since null semigroups are commutative nilpotent semigroups, then we have $\abs{S}\geqslant\maxnull{\abs{X}}$ (by Theorem~\ref{maximum size null semigroup}). We analyse three cases.
 	
 	\smallskip
 	
 	\textit{Case 1:} Assume that $\abs{X}=1$. It is straightforward to see that $\abs{S}=1=\maxnull{1}=\maxnull{\abs{X}}$ and that $S=\set{\id{X}}=N^X_x$, where $X=\set{x}$. We note that $\alphanull{\abs{X}}=\alphanull{1}=1$.
 	
 	\smallskip
 	
 	\textit{Case 2:} Assume that $\abs{X}\in\set{2,3}$. It follows from Theorem~\ref{T(X): largest comm smg 1 idemp} that the largest commutative subsemigroup of $\tr{X}$ with a unique idempotent is a group of size $\abs{X}$. Since $S$ is a nilpotent semigroup, then this implies that $\abs{S}\leqslant\abs{X}-1\leqslant 2$ and, consequently, we have that $S$ is a null semigroup. Then Theorem~\ref{maximum size null semigroup} implies that $\abs{S}\leqslant\maxnull{\abs{X}}$ and, consequently, we have $\abs{S}=\maxnull{\abs{X}}$.
 	
 	If $\abs{X}=2$, then Theorem~\ref{null semigroups of maximum size} implies that $S=\set{\id{X}}$ or $S=\nulltr{X}{x_1}{x_t}$, where $t=\alphanull{\abs{X}}$ and $x_1,\ldots,x_t\in S$ are pairwise distinct.
 	
 	If $\abs{X}=3$, then Theorem~\ref{null semigroups of maximum size} implies that $S=\nulltr{X}{x_1}{x_t}$, where $t=\alphanull{\abs{X}}$ and $x_1,\ldots,x_t\in S$ are pairwise distinct.
 	
 	\smallskip
 	
 	\textit{Case 3:} Assume that $\abs{X}\geqslant 4$. It follows from Theorem~\ref{T(X): largest comm smg 1 idemp} that the maximum-order commutative subsemigroup of $\tr{X}$ with a unique idempotent have size $\maxnull{\abs{X}}$ and they can either be groups or one of the null semigroups $\nulltr{X}{x_1}{x_t}$, where $t=\alphanull{\abs{X}}$ and $x_1,\ldots,x_t\in S$ are pairwise distinct. The former implies that $\abs{S}=\maxnull{\abs{X}}$. Furthermore, since $S$ is a nilpotent semigroup and $\abs{S}>1$, then $S$ is not a group. Hence there exist pairwise distinct $x_1,\ldots,x_t\in S$ ($t=\alphanull{\abs{X}}$) such that $S=\nulltr{X}{x_1}{x_t}$.
 \end{proof}
 
 

 Our next objective is to characterize the largest commutative subsemigroups of $\ptr{X}$ with a unique idempotent. With this in mind, we define the following subsets of $\ptr{X}$, which will turn out to be null semigroups.
 
 For each $B\subseteq X$ such that $\abs{B}=\parens{\abs{X}+1}\alpha-1$ we define:
 \begin{displaymath}
 	\nullptr{X}{B}=\gset{\beta\in\ptr{X}}{\dom{\beta}\subseteq X\setminus B \text{ and } \im{\beta}\subseteq B}.
 \end{displaymath}

 \begin{proposition}\label{P(X): Omega^X_B}
 	For each $B\subseteq X$ such that $\abs{B}=\alphanull{\abs{X}+1}-1$, we have that $\nullptr{X}{B}$ is a null subsemigroup of $\ptr{X}$ of size $\maxnull{\abs{X}+1}$.
 \end{proposition}
 
 \begin{proof}
 	Let $B\subseteq X$ be such that $\abs{B}=\alphanull{\abs{X}+1}-1$. We have that $\emptyset\subseteq X\setminus B$ and $\emptyset\subseteq B$. Hence $\emptyset\in\nullptr{X}{B}$. Moreover, for all $\beta,\gamma\in\nullptr{X}{B}$ we have
 	\begin{align*}
 		\dom\beta\gamma&=\parens{\im\beta\cap\dom\gamma}\beta^{-1}\\
 		&\subseteq \parens{B\cap\parens{X\setminus B}}\beta^{-1} &\bracks{\text{because } \im\beta\subseteq B \text{ and } \dom\beta\subseteq X\setminus B}\\
 		&=\emptyset,
 	\end{align*}
 	which implies that for all $\beta,\gamma\in\nullptr{X}{B}$ we have $\beta\gamma=\emptyset$. Therefore $\nullptr{X}{B}$ is a null subsemigroup of $\ptr{X}$.
 	
 	Now we will see that $\abs{\nullptr{X}{B}}=\maxnull{\abs{X}+1}$. Let $\beta\in\nullptr{X}{B}$. For each $x\in B$ we have that $x\in X\setminus\dom\beta$, and for each $x\in X\setminus B$ we have either $x\in X\setminus\dom\beta$ or $x\in\dom\beta$ and $x\beta\in B$. This implies that, in $\beta$, we have $1$ possibility for each $x\in B$, and we have $\abs{B}+1$ possibilities for each $x\in X\setminus B$. Hence there are
 	\begin{displaymath}
 		1^{\abs{B}}\cdot\parens{\abs{B}+1}^{\abs{X\setminus B}}=\parens{\abs{B}+1}^{\abs{X}-\abs{B}}=\parens{\alphanull{\abs{X}+1}}^{\parens{\abs{X}+1}-\alphanull{\abs{X}+1}}=\maxnull{\abs{X}+1}
 	\end{displaymath}
 	possibilities for $\beta$. This is enough to conclude that $\abs{\nullptr{X}{B}}=\maxnull{\abs{X}+1}$.
 \end{proof}
 
 Now we will demonstrate that the null semigroups $\nullptr{X}{B}$, where $B\subseteq X$ is such that $\abs{B}=\alphanull{\abs{X}+1}-1$, are maximum-order commutative subsemigroups of $\ptr{X}$ with a unique idempotent. Moreover, when $\abs{X}\geqslant 3$, there are no other commutative subsemigroups of $\ptr{X}$ with a unique idempotent that have maximum size. When $\abs{X}\leqslant 2$, one of the maximum-order commutative subsemigroups of $\ptr{X}$ with a unique idempotent is a cyclic group.
 
 \begin{corollary}\label{P(X): comm smg 1 idemp}
 	The maximum size of a commutative subsemigroup of $\ptr{X}$ with a unique idempotent is $\maxnull{\abs{X}+1}$. Furthermore, if $S$ is a maximum-order commutative subsemigroup of $\ptr{X}$ with a unique idempotent, then:
 	\begin{enumerate}
 		\item If $\abs{X}\leqslant 2$, then either $C_{\abs{X}}\simeq S\subseteq\sym{X}$ or $S=\nullptr{X}{B}$ for some $B\subseteq X$ such that $\abs{B}=\alphanull{\abs{X}+1}-1$.
 		\item If $\abs{X}\geqslant 3$, then $S=\nullptr{X}{B}$ for some $B\subseteq X$ such that $\abs{B}=\alphanull{\abs{X}+1}-1$.
 	\end{enumerate}
 \end{corollary}
 
 \begin{proof}
 	
 	Let $S$ be a largest commutative subsemigroup of $\ptr{X}$ with a unique idempotent. Then, by Proposition~\ref{P(X): proposition P(X) --> T(Y)}, $\newtr{S}$ is a subsemigroup of $\tr{\newtr{X}}$ and $\newtr{S}\simeq S$. Therefore $\abs{S}=\abs{\newtr{S}}$ and $\newtr{S}$ is commutative (because $S$ is commutative) and $\newtr{S}$ contains only one idempotent (because $S$ contains only one idempotent).
 	
 	\smallskip
 	
 	\textit{Case 1:} Assume that $\abs{X}=1$. We have that $\ptr{X}=\set{\emptyset,\id{X}}$. Since $S$ contains only one idempotent, then we must have $S=\set{\id{X}}=\sym{X}\simeq C_{\abs{X}}$ or $S=\set{\emptyset}=\nullptr{X}{X}$. (We observe that $\abs{X}=1=2-1=\alphanull{2}-1=\alphanull{\abs{X}+1}-1$.) Consequently, $\abs{S}=1=\maxnull{2}=\maxnull{\abs{X}+1}$.
 	
 	\smallskip
 	
 	\textit{Case 2:} Assume that $\abs{X}=2$ and $X=\set{x_1,x_2}$. We have that
 	\begin{displaymath}
 		\ptr{X}=\set*{\emptyset,
 			\begin{pmatrix}
 				x_1\\x_1
 			\end{pmatrix},
 			\begin{pmatrix}
 				x_2\\x_2
 			\end{pmatrix},
 			\begin{pmatrix}
 				x_1\\x_2
 			\end{pmatrix},
 			\begin{pmatrix}
 				x_2\\x_1
 			\end{pmatrix},
 			\begin{pmatrix}
 				x_1&x_2\\x_1&x_1
 			\end{pmatrix},
 			\begin{pmatrix}
 				x_1&x_2\\x_2&x_2
 			\end{pmatrix},
 			\begin{pmatrix}
 				x_1&x_2\\x_2&x_1
 			\end{pmatrix},
 			\id{X}}.
 	\end{displaymath}
 	Furthermore, it is easy to verify that there are only three partial transformations in $\ptr{X}$ that are not idempotents, namely,
 	\begin{displaymath}
 		\beta_1=\begin{pmatrix}
 			x_2\\x_1
 		\end{pmatrix} \quad \text{and} \quad
 		\beta_2=\begin{pmatrix}
 			x_1\\x_2
 		\end{pmatrix} \quad \text{and} \quad
 		\beta_3=\begin{pmatrix}
 			x_1&x_2\\x_2&x_1
 		\end{pmatrix},
 	\end{displaymath}
 	and it is also easy to verify that these three partial transformations do not commute with each other. Hence $S$ contains at most one of them. We notice that, since $S$ is a maximum-order commutative subsemigroup of $\ptr{X}$ with a unique idempotent, then $S$ must contain exactly one idempotent and exactly one element of $\set{\beta_1,\beta_2,\beta_3}$. Therefore $\abs{S}=2=\maxnull{3}=\maxnull{\abs{X}+1}$. In the next two sub-cases we characterize $S$.
 	
 	\smallskip
 	
 	\textsc{Sub-case 1:} Assume that $\beta_i\in S$ for some $i\in\set{1,2}$. Then $\emptyset=\beta_i^2\in S$, which implies that $S=\set{\emptyset,\beta_i}=\nullptr{X}{\set{x_i}}$. (Note that $\abs{\set{x_i}}=1=2-1=\alphanull{3}-1=\alphanull{\abs{X}+1}-1$.)
 	
 	\smallskip
 	
 	\textsc{Sub-case 2:} Assume that $\beta_3\in S$. Then $\id{X}=\beta_3^2\in S$ and, consequently, we have that $C_{\abs{X}}\simeq S=\set{\id{X},\beta_3}\subseteq\sym{X}$. 
 	
 	
 	\smallskip
 	
 	\textit{Case 3:} Assume that $\abs{X}\geqslant 3$. Then $\abs{\newtr{X}}\geqslant 4$ and, consequently, Theorem~\ref{T(X): largest comm smg 1 idemp} implies that $\abs{\newtr{S}}\leqslant \maxnull{\abs{\newtr{X}}}$. (We note that, if $\abs{\newtr{X}}=4$, then $\maxnull{\abs{\newtr{X}}}=\maxnull{4}=4=\abs{\newtr{X}}$.) Furthermore, in Proposition~\ref{P(X): Omega^X_B} we saw that the semigroups $\nullptr{X}{B}$, where $B\subseteq X$ is such that $\abs{B}=\alphanull{\abs{X}+1}-1$, are null subsemigroups of $\ptr{X}$ (which are commutative semigroups with exactly one idempotent) of size $\maxnull{\abs{X}+1}$. As a consequence of the fact that $S$ is a largest commutative subsemigroup of $\ptr{X}$ with a unique idempotent, we have that $\abs{\newtr{S}}=\abs{S}\geqslant \maxnull{\abs{X}+1}=\maxnull{\abs{\newtr{X}}}$. Thus $\abs{\newtr{S}}=\maxnull{\abs{\newtr{X}}}$. Then, by Theorem~\ref{T(X): largest comm smg 1 idemp} we have that:
 	\begin{enumerate}
 		\item If $\abs{\newtr{X}}=4$ (that is, $\abs{X}=3$), then either $\newtr{S}\subseteq \sym{\newtr{X}}$ or $\newtr{S}=N^{\newtr{X}}_{x_1,x_2}$ for some distinct $x_1,x_2\in \newtr{X}$.
 		
 		\item If $\abs{\newtr{X}}\geqslant 5$ (that is, $\abs{X}\geqslant 4$), then $\newtr{S}=\nulltr{\newtr{X}}{x_1}{x_{\alphanull{\abs{\newtr{X}}}}}$ for some pairwise distinct $x_1,\ldots,x_{\alphanull{\abs{\newtr{X}}}}\in \newtr{X}$.
 	\end{enumerate}
 	
 	Before we proceed with the characterization of $S$ we need to establish that, regardless of the size of $\newtr{X}$, we must have $\newtr{S}=\nulltr{\newtr{X}}{x_1}{x_{\alphanull{\abs{\newtr{X}}}}}$ for some pairwise distinct $x_1,\ldots,x_{\alphanull{\abs{\newtr{X}}}}\in \newtr{X}$. It is clear that we just need to show that, when $\abs{\newtr{X}}=4$ (that is, $\abs{X}=3$), we must have $\newtr{S}\nsubseteq\sym{\newtr{X}}$.
 	
 	Suppose that $\abs{\newtr{X}}=4$ (and $\abs{X}=3$). Assume, with the aim of obtaining a contradiction that $\newtr{S}\subseteq\sym{\newtr{X}}$. This implies that $\set{\new}\newtr{\beta}^{-1}=\set{\new}$ for all $\beta\in S$. Hence we must have $\dom\beta=X$ for all $\beta\in S$; that is, $S\subseteq\tr{X}$. Consequently, $S$ is a commutative subsemigroup of $\tr{X}$ with a unique idempotent and, by Theorem~\ref{T(X): largest comm smg 1 idemp}, we have that $4=\maxnull{4} =\maxnull{\abs{\newtr{X}}} =\abs{\newtr{S}} =\abs{S}\leqslant\abs{X}=3$, which is a contradiction. Therefore $\newtr{S}\nsubseteq\sym{\newtr{X}}$ and, thus, $\newtr{S}=N^{\newtr{X}}_{x_1,x_2}$ for some distinct $x_1,x_2\in \newtr{X}$. (We note that $\alphanull{\abs{\newtr{X}}}=\alphanull{4}=2$.)
 	
 	The previous paragraph allows us to conclude that $\newtr{S}=\nulltr{\newtr{X}}{x_1}{x_{\alphanull{\abs{\newtr{X}}}}}$ for some pairwise distinct $x_1,\ldots,x_{\alphanull{\abs{\newtr{X}}}}\in \newtr{X}$. It follows from the definition of $\nulltr{\newtr{X}}{x_1}{x_{\alphanull{\abs{\newtr{X}}}}}$ that $x_1$ is the unique element of $\newtr{X}$ such that $x_1\newtr{\beta}=x_1$ for all $\beta\in S$; and it follows from the definition of $\newtr{S}$ that $\new\newtr{\beta}=\new$ for all $\beta\in S$. Hence $x_1=\new$. Let $B=\set{x_1,\ldots,x_{\alphanull{\abs{\newtr{X}}}}}\setminus\set{\new}=\set{x_2,\ldots,x_{\alphanull{\abs{\newtr{X}}}}}$ and $\beta\in S$. Since $\newtr{S}=\nulltr{\newtr{X}}{x_1}{x_{\alphanull{\abs{\newtr{X}}}}}$, then we have that $B\newtr{\beta}=\set{x_2,\ldots,x_{\alphanull{\abs{\newtr{X}}}}}\newtr{\beta}=\set{x_1}=\set{\new}$, which implies that $\dom\beta\subseteq X\setminus B$; and we have that $\im\newtr{\beta}\subseteq\set{x_1,\ldots,x_{\alphanull{\abs{\newtr{X}}}}}=B\cup\set{\new}$, which implies that $\im\beta\subseteq B$. Thus $\beta\in\nullptr{X}{B}$ and, consequently, $S\subseteq \nullptr{X}{B}$. (We notice that $\abs{B}=\alphanull{\abs{\newtr{X}}}-1=\alphanull{\abs{X}+1}-1$.) Due to the fact that $\abs{S}=\abs{\newtr{S}}=\maxnull{\abs{\newtr{X}}}=\maxnull{\abs{X}+1}=\abs{\nullptr{X}{B}}$ (where the last equality follows from Proposition~\ref{P(X): Omega^X_B}), then we can conclude that $S=\nullptr{X}{B}$.
 \end{proof}
 
 In the last corollary of the section we establish that the commutative nilpotent subsemigroups of $\ptr{X}$ of maximum size are all null semigroups of size $\maxnull{\abs{X}+1}$.
 
 \begin{corollary}\label{P(X): largest comm nilpotent smg}
 	The maximum size of a commutative nilpotent subsemigroup of $\ptr{X}$ is $\maxnull{\abs{X}+1}$. Moreover, $S$ is a commutative nilpotent subsemigroup of $\ptr{X}$ of size $\maxnull{\abs{X}+1}$ if and only if at least one of the following conditions is satisfied:
 	\begin{enumerate}
 		
 		\item $S=\nullptr{X}{B}$ for some $B\subseteq X$ such that $\abs{B}=\parens{\abs{X}+1}\alpha-1$.
 		
 		\item $\abs{X}=1$ and $S=\set{\id{X}}$.
 	\end{enumerate}
 \end{corollary}
 
 \begin{proof}
 	Let $S$ be commutative nilpotent subsemigroup of $\ptr{X}$ of maximum size. Then $S$ contains only one idempotent --- its zero. Consequently, it follows from Corollary~\ref{P(X): comm smg 1 idemp} that $\abs{S}\leqslant \maxnull{\abs{X}+1}$. Furthermore, Proposition~\ref{P(X): Omega^X_B} implies that for all $B \subseteq X$ such that $\abs{B}=\alphanull{\abs{X}+1}-1$ we have that $\nullptr{X}{B}$ is a null subsemigroup of $\ptr{X}$ (and, consequently, a commutative nilpotent subsemigroup of $\ptr{X}$) whose size is $\maxnull{\abs{X}+1}$. Then, as a consequence of the fact that $S$ is a maximum-order commutative nilpotent subsemigroup of $\ptr{X}$, we have that $\abs{S}\geqslant\maxnull{\abs{X}+1}$. Consequently, $\abs{S}=\maxnull{\abs{X}+1}$. Hence, by Corollary~\ref{P(X): comm smg 1 idemp}, at least one of the following conditions hold:
 	\begin{enumerate}
 		\item There exists $B\subseteq X$ such that $\abs{B}=\parens{\abs{X}+1}\alpha-1$ and $S=\nullptr{X}{B}$.
 		
 		\item $\abs{X}=1$ and $S\subseteq\sym{X}=\set{\id{X}}$.
 		
 		\item $\abs{X}=2$ and $S\simeq C_{\abs{X}}=C_2$. 
 	\end{enumerate}
 	In order to finish this proof we just need to observe that condition 3 never holds: in fact, when $\abs{X}=2$ we cannot have $S\simeq C_2$ because $C_2$ is not a nilpotent semigroup.
 \end{proof}

 \section{The largest commutative (full and partial) transformation semigroups}\label{sec: largest comm smg T(X)}
 
 Recall that $X$ denotes a finite set. This section concerns the maximum-order commutative subsemigroups of $\tr{X}$ and of $\ptr{X}$. We will prove that the maximum size of a commutative subsemigroup of $\tr{X}$ is $2^{\abs{X}-1}$, when $\abs{X}\leqslant 6$, and at least $\maxnull{\abs{X}}+1$, when $\abs{X}\geqslant 7$. We recall that in Section~\ref{sec: largest comm smg idemp T(X)} we described some commutative subsemigroups (of idempotents) of $\tr{X}$ of size $2^{\abs{X}-1}$ --- the semigroups $\commidemp{X}{x}$, where $x\in X$ \eqref{T(X): Gamma^X_i definition}. In this section we will see that, when $\abs{X}\leqslant 6$ but $\abs{X}\neq 2$, those semigroups are precisely the maximum-order commutative subsemigroups of $\tr{X}$ and, when $\abs{X}=2$, the only commutative subsemigroup of $\tr{X}$ (other than the semigroups $\commidemp{X}{x}$, where $x\in X$) is the subgroup of $\sym{X}$ isomorphic to the cyclic group $C_{\abs{X}}$. Furthermore, in \cite{Null_semigroups} were described null semigroups (which are commutative) of size $\maxnull{\abs{X}}$. We will see that, when $\abs{X}\geqslant 7$, these semigroups have size greater than $2^{\abs{X}-1}$, and so the semigroups $\commidemp{X}{x}$, where $x\in X$, are no longer the largest commutative subsemigroups of $\tr{X}$. Finally, we will demonstrate that, when $\abs{X}\leqslant 5$, the unique commutative subsemigroup of $\ptr{X}$ of maximum size is $\idemp{\psym{X}}$ --- the unique commutative subsemigroup of idempotents of $\ptr{X}$ of maximum size --- which has size $2^{\abs{X}}$. When $\abs{X}\geqslant 6$, we will demonstrate that $\maxnull{\abs{X}+1}+1$ is a lower bound for the maximum size of a commutative subsemigroup of $\ptr{X}$, and that the maximum size of a commutative subsemigroup of $\tr{Y}$, where $Y$ is a set such that $\abs{Y}=\abs{X}+1$, is an upper bound.

 Like in the previous two sections, we begin by proving the results concerning $\tr{X}$. With this goal in mind, we introduce the first lemma of this section, which shows that, when $\abs{X}\geqslant 3$, the largest commutative subsemigroups of $\tr{X}$ are not contained in $\sym{X}$. (Recall that, by Proposition~\ref{T(X): Gamma^X_i comm smg idemp}, there exist commutative subsemigroups of $\tr{X}$ of size $2^{\abs{X}-1}$.)
 
 \begin{lemma}\label{T(X): upper bound comm subsmg S(X)}
 	Suppose that $\abs{X}\geqslant 3$. Let $S$ be a commutative subsemigroup of $\tr{X}$. If $S\subseteq\sym{X}$, then $\abs{S}< 2^{\abs{X}-1}$.
 \end{lemma}
 
 \begin{proof}
 	Suppose that $S\subseteq\sym{X}$. Let $\alpha\in S$. Since the unique idempotent of $\sym{X}$ is $\id{X}$, then $\id{X}$ is the unique idempotent of $S$. Hence Proposition~\ref{T(X): unique idemp is id_X => group in S(X)} implies that $S$ is a subgroup of $\sym{X}$.
 	
 	In order to prove the result, we consider the following three cases:
 	
 	\smallskip
 	
 	\textit{Case 1:} Suppose that $\abs{X}=3k$, for some $k\in\mathbb{Z}$. Since $\abs{X}\geqslant 3$, then $k\geqslant 1$. Furthermore, it follows from Theorem~\ref{T(X): maximum size abelian subgroup of S(X)} that $\abs{S}\leqslant 3^k$. Hence
 	\begin{displaymath}
 		\abs{S}\leqslant 3^k<4^k=2^{2k}\leqslant 2^{2k}\cdot 2^{k-1}=2^{3k-1}=2^{\abs{X}-1}.
 	\end{displaymath}
 	
 	\smallskip
 	
 	\textit{Case 2:} Suppose that $\abs{X}=3k+1$, for some $k\in\mathbb{Z}$. Due to the fact that $\abs{X}\geqslant 3$, we must have $k\geqslant 1$. Additionally, as a result of Theorem~\ref{T(X): maximum size abelian subgroup of S(X)}, we have $\abs{S}\leqslant 4\cdot 3^{k-1}$. Thus
 	\begin{displaymath}
 		\abs{S}\leqslant 4\cdot 3^{k-1}\leqslant 4\cdot 4^{k-1}=4^k=2^{2k}<2^{2k}\cdot 2^k=2^{\parens{3k+1}-1}=2^{\abs{X}-1}.
 	\end{displaymath}
 	
 	\smallskip
 	
 	\textit{Case 3:} Suppose that $\abs{X}=3k+2$, for some $k\in\mathbb{Z}$. Then $k\geqslant 1$ (because $\abs{X}\geqslant 3$). Furthermore, Theorem~\ref{T(X): maximum size abelian subgroup of S(X)} guarantees that $\abs{S}\leqslant 2\cdot 3^k$. Thus
 	\begin{displaymath}
 		\abs{S}\leqslant 2\cdot 3^k<2\cdot 8^k=2\cdot 2^{3k}=2^{\parens{3k+2}-1}=2^{\abs{X}-1}.\qedhere
 	\end{displaymath}
 \end{proof}
 
 In the next lemma we mention the class $\setclass{S}{X}$ (where $S$ is a commutative subsemigroup of $\tr{X}$), which was originally introduced in Section~\ref{sec: largest comm smg idemp T(X)} in \eqref{T(X): C(S,X) definition}. We also recall that, if $S\nsubseteq\sym{X}$, then Lemma~\ref{T(X): existence of I, b|_(X/I)} implies that the class $\setclass{S}{X}$ is non-empty.
 
 It follows from Lemma~\ref{T(X): upper bound comm subsmg S(X)} that, when $\abs{X}\geqslant 3$, the largest commutative subsemigroups of $\tr{X}$ are not contained in $\sym{X}$. Hence, if $\abs{X}\geqslant 3$ and $S$ is a maximum-order commutative subsemigroup of $\tr{X}$, then Lemma~\ref{T(X): existence of I, b|_(X/I)} implies that $\setclass{S}{X}\neq\emptyset$. The aim of Lemma~\ref{T(X): upper bound S, |I|=2} is to provide the tools to prove (in Theorem~\ref{T(X): maximum size comm smg}) that, when $\abs{X}\in\set{3,4,5,6}$ and $S$ is a maximum-order commutative subsemigroup of $\tr{X}$, the smallest sets of $\setclass{S}{X}$ do not have size $2$. 
 
 \begin{lemma}\label{T(X): upper bound S, |I|=2}
 	Suppose that $\abs{X}\leqslant 6$. Let $S$ be a commutative subsemigroup of $\tr{X}$ such that $\setclass{S}{X}\neq\emptyset$ and let $I\in\setclass{S}{X}$ be of minimum size. Let $S'=\gset{\beta|_{X\setminus I}}{\beta\in S}$. If $\abs{I}=2$, then
 	\begin{enumerate}
 		\item $\abs{S}\leqslant \max\set{\abs{X},\parens{\abs{X}-1}\parens{\abs{X}-2},\abs{S'}\cdot\parens{\abs{X}-2}}$.
 		
 		\item If either $\abs{X}\leqslant 5$ and $\abs{S'}\leqslant 2^{\abs{X}-3}$ or $\abs{X}=6$ and $\abs{S'}<2^{\abs{X}-3}$, then $\abs{S}<2^{\abs{X}-1}$.
 		
 		\item If $\abs{X}=6$, then $S'\neq\commidemp{X\setminus I}{x}$ for all $x\in X\setminus I$.
 	\end{enumerate}
 \end{lemma}
 
 We note that it follows from the definition of $\setclass{S}{X}$ and Lemma~\ref{T(X): lemma induction} that $S'$ is a commutative subsemigroup of $\tr{X\setminus I}$.
 
 \begin{proof}
 	Suppose that $\abs{I}=2$. Let $n=\abs{X}$. Assume that $I=\set{i_1,i_2}$ and $X\setminus I=\set{x_1,\ldots,x_{n-2}}$.
 	
 	It follows from Lemma~\ref{T(X): lemma alpha|_I cycle} that there exists $\alpha\in S$ such that $\alpha|_I$ is a product of (disjoint) cycles of the same length, which is at least $2$. Since $\abs{I}=2$, then $\alpha|_I$ must be a cycle of length $2$. Hence
 	\begin{displaymath}
 		\alpha|_I=\begin{pmatrix}
 			i_1 & i_2 \\
 			i_2 & i_1
 		\end{pmatrix}.
 	\end{displaymath}

 	For each $j\in\set{1,2}$ we define $A_j=\gset{\beta\in S}{i_1\beta=i_j}$ and for each $j\in\X{n-2}$ we define $B_j=\gset{\beta\in S}{i_1\beta=x_j}$. It is clear that these $n$ sets form a partition of $S$. We begin by proving the lemma below.
 	
 	\begin{lemma}\label{T(X): lemma upper bound S, |I|=2}
 		Let $j\in\X{n-2}$ and suppose that $B_j\neq\emptyset$. Let $\beta_j\in B_j$ and $\beta\in S$.
 		\begin{enumerate}
 			\item If $i_1\beta=i_1$, then $i_2\beta=i_2$ and $x_j\beta=x_j$.
 			
 			\item If $i_1\beta=i_2$, then $i_2\beta=i_1$ and $x_j\beta=x_j\alpha$.
 			
 			\item If $i_1\beta=x_k$ for some $k\in\X{n-2}$, then $i_2\beta=x_k\alpha$ and $x_j\beta=x_k\beta_j$.
 		\end{enumerate}
 	\end{lemma}
 	
 	\begin{proof}
 		It follows from the fact that $\beta_j\in B_j$ that $i_1\beta_j=x_j$. We consider the following cases.
 		
 		\smallskip
 		
 		\textit{Case 1:} Suppose that $i_1\beta=i_1$. We have
 		\begin{align*}
 			i_2\beta &=\parens{i_1\alpha}\beta &\bracks{\text{since } i_1\alpha=i_2}\\
 			&=\parens{i_1\beta}\alpha &\bracks{\text{since } \alpha,\beta\in S, \text{ which is commutative}}\\
 			&=i_1\alpha &\bracks{\text{since } i_1\beta=i_1}\\
 			&=i_2
 		\end{align*}
 		and
 		\begin{align*}
 			x_j\beta &=\parens{i_1\beta_j}\beta &\bracks{\text{since } i_1\beta_j=x_j}\\
 			&=\parens{i_1\beta}\beta_j &\bracks{\text{since } \beta_j,\beta\in S, \text{ which is commutative}}\\
 			&=i_1\beta_j &\bracks{\text{since } i_1\beta=i_1}\\
 			&=x_j.
 		\end{align*}
 		
 		\smallskip
 		
 		\textit{Case 2:} Suppose that $i_1\beta=i_2$. We have
 		\begin{align*}
 			i_2\beta &=\parens{i_1\alpha}\beta &\bracks{\text{since } i_1\alpha=i_2}\\
 			&=\parens{i_1\beta}\alpha &\bracks{\text{since } \alpha,\beta\in S, \text{ which is commutative}}\\
 			&=i_2\alpha &\bracks{\text{since } i_1\beta=i_2}\\
 			&=i_1
 		\end{align*}
 		and
 		\begin{align*}
 			x_j\beta &=\parens{i_1\beta_j}\beta &\bracks{\text{since } i_1\beta_j=x_j}\\
 			&=\parens{i_1\beta}\beta_j &\bracks{\text{since } \beta_j,\beta\in S, \text{ which is commutative}}\\
 			&=i_2\beta_j &\bracks{\text{since } i_1\beta=i_2}\\
 			&=\parens{i_1\alpha}\beta_j &\bracks{\text{since } i_1\alpha=i_2}\\
 			&=\parens{i_1\beta_j}\alpha &\bracks{\text{since } \alpha,\beta_j\in S, \text{ which is commutative}}\\
 			&=x_j\alpha.
 		\end{align*}
 		
 		\smallskip
 		
 		\textit{Case 3:} Suppose that $i_1\beta=x_k$ for some $k\in\X{n-2}$. We have
 		\begin{align*}
 			i_2\beta &=\parens{i_1\alpha}\beta &\bracks{\text{since } i_1\alpha=i_2}\\
 			&=\parens{i_1\beta}\alpha &\bracks{\text{since } \alpha,\beta\in S, \text{ which is commutative}}\\
 			&=x_k\alpha
 		\end{align*}
 		and
 		\begin{align*}
 			x_j\beta &=\parens{i_1\beta_j}\beta &\bracks{\text{since } i_1\beta_j=x_j}\\
 			&=\parens{i_1\beta}\beta_j &\bracks{\text{since } \beta_j,\beta\in S, \text{ which is commutative}}\\
 			&=x_k\beta_j. & &\qedhere
 		\end{align*}
 	\end{proof}
 	
 	We now continue with the proof of Lemma~\ref{T(X): upper bound S, |I|=2}.
 	
 	\medskip
 	
 	\textbf{Part 1.} The aim of this part is to show that $\abs{S}\leqslant \max\set{n,\parens{n-1}\parens{n-2},\abs{S'}\cdot\parens{n-2}}$. We have that $\set{A_1,A_2,B_1,\ldots,B_{n-2}}$ is a partition of $S$. So, in order to determine an upper bound for $\abs{S}$, we just need to determine upper bounds for $\abs{A_1},\abs{A_2},\abs{B_1},\ldots,\abs{B_{n-2}}$. We consider the three cases below.
 	
 	\smallskip
 	
 	\textit{Case 1:} Suppose that the sets $B_1,\ldots,B_{n-2}$ are all non-empty. For each $j\in\X{n-2}$ we select $\beta_j\in B_j$. Then $i_1\beta_j=x_j$ for all $j\in\X{n-2}$.

 	Let $\beta\in A_1$. Then $i_1\beta=i_1$. It follows from part 1 of Lemma~\ref{T(X): lemma upper bound S, |I|=2} that
 	\begin{displaymath}
 		\beta=\begin{pmatrix}
 			i_1 & i_2 & x_1 & x_2 & \cdots & x_{n-3} & x_{n-2}\\
 			i_1 & i_2 & x_1 & x_2 & \cdots & x_{n-3} & x_{n-2}
 		\end{pmatrix}=\id{X}.
 	\end{displaymath}
 	Hence $A_1=\set{\id{X}}$ and, consequently, $\abs{A_1}=1$.
 	
 	Let $\beta\in A_2$. Then $i_1\beta=i_2$ and, by part 2 of Lemma~\ref{T(X): lemma upper bound S, |I|=2} we have that
 	\begin{displaymath}
 		\beta=\begin{pmatrix}
 			i_1 & i_2 & x_1 & x_2 & \cdots & x_{n-3} & x_{n-2}\\
 			i_2 & i_1 & x_1\alpha & x_2\alpha & \cdots & x_{n-3}\alpha & x_{n-2}\alpha
 		\end{pmatrix}=\alpha,
 	\end{displaymath}
 	which implies that $A_2=\set{\alpha}$ and, consequently, we have $\abs{A_2}=1$.
 	
 	Let $k\in\X{n-2}$ and $\beta\in B_k$. We have that $i_1\beta=x_k$ and, by part 3 of Lemma~\ref{T(X): lemma upper bound S, |I|=2}, we have that
 	\begin{displaymath}
 		\beta=\begin{pmatrix}
 			i_1 & i_2 & x_1 & x_2 & \cdots & x_{n-3} & x_{n-2}\\
 			x_k & x_k\alpha & x_k\beta_1 & x_k\beta_2 & \cdots & x_k\beta_{n-3} & x_k\beta_{n-2}
 		\end{pmatrix}.
 	\end{displaymath}
 	This allows us to conclude that $\abs{B_k}=1$ for all $k\in\X{n-2}$.
 	
 	Therefore
 	\begin{displaymath}
 		\abs{S}=\abs[\bigg]{A_1\cup A_2\cup\parens[\bigg]{\,\bigcup_{j=1}^{n-2} B_j}}=\abs{A_1}+\abs{A_2}+\sum_{j=1}^{n-2}\abs{B_j}=1+1+\sum_{j=1}^{n-2} 1 =n.
 	\end{displaymath}
 	
 	\smallskip
 	
 	\textit{Case 2:} Suppose that among the sets $B_1,\ldots,B_{n-2}$ there are exactly $n-3$ that are non-empty. Assume, without loss of generality, that $B_1,\ldots,B_{n-3}$ are those sets. Hence $B_{n-2}=\emptyset$. For each $j\in\X{n-3}$ we select $\beta_j\in B_j$. Then $i_1\beta_j=x_j$ for all $j\in\X{n-3}$.

 	Let $\beta\in A_1$. Then $i_1\beta=i_1$. It follows from part 1 of Lemma~\ref{T(X): lemma upper bound S, |I|=2} that
 	\begin{displaymath}
 		\beta=\begin{pmatrix}
 			i_1 & i_2 & x_1 & x_2 & \cdots & x_{n-3} & x_{n-2}\\
 			i_1 & i_2 & x_1 & x_2 & \cdots & x_{n-3} & y
 		\end{pmatrix},
 	\end{displaymath}
 	for some $y\in \set{x_1,\ldots,x_{n-2}}$ (we recall that the fact that $I\in\setclass{S}{X}$ implies that $\set{x_1,\ldots,x_{n-2}}\beta=\parens{X\setminus I}\beta\subseteq X\setminus I=\set{x_1,\ldots,x_{n-2}}$). This implies that there exist at most $n-2$ possibilities for $\beta$. Since $\beta$ is an arbitrary element of $A_1$, we can conclude that $\abs{A_1}\leqslant n-2$.
 	
 	Let $\beta\in A_2$. Then $i_1\beta=i_2$ and, consequently, part 2 of Lemma~\ref{T(X): lemma upper bound S, |I|=2} implies that
 	\begin{displaymath}
 		\beta=\begin{pmatrix}
 			i_1 & i_2 & x_1 & x_2 & \cdots & x_{n-3} & x_{n-2}\\
 			i_2 & i_1 & x_1\alpha & x_2\alpha & \cdots & x_{n-3}\alpha & y
 		\end{pmatrix},
 	\end{displaymath}
 	for some $y\in \set{x_1,\ldots,x_{n-2}}$ (we recall that the fact that $I\in\setclass{S}{X}$ implies that $\set{x_1,\ldots,x_{n-2}}\beta=\parens{X\setminus I}\beta\subseteq X\setminus I=\set{x_1,\ldots,x_{n-2}}$). Hence we have $n-2$ possibilities for $\beta$, which implies that $\abs{A_2}\leqslant n-2$.
 	
 	Let $k\in\X{n-3}$ and $\beta\in B_k$. Then $i_1\beta=x_k$ and, by part 3 of Lemma~\ref{T(X): lemma upper bound S, |I|=2}, we have that
 	\begin{displaymath}
 		\beta=\begin{pmatrix}
 			i_1 & i_2 & x_1 & x_2 & \cdots & x_{n-3} & x_{n-2}\\
 			x_k & x_k\alpha & x_k\beta_1 & x_k\beta_2 & \cdots & x_k\beta_{n-3} & y
 		\end{pmatrix},
 	\end{displaymath}
 	for some $y\in \set{x_1,\ldots,x_{n-2}}$ (we note that we have $\set{x_1,\ldots,x_{n-2}}\beta=\parens{X\setminus I}\beta\subseteq X\setminus I=\set{x_1,\ldots,x_{n-2}}$ because $I\in\setclass{S}{X}$). Hence there are at most $n-2$ possibilities for $\beta$. It follows from the fact that $k$ is an arbitrary element of $\X{n-3}$ that $\abs{B_k}\leqslant n-2$ for all $k\in\X{n-3}$.
 	
 	Therefore
 	\begin{align*}
 		\abs{S} &=\abs[\bigg]{A_1\cup A_2\cup\parens[\bigg]{\,\bigcup_{j=1}^{n-2} B_j}}\\
 		&=\abs{A_1}+\abs{A_2}+\sum_{j=1}^{n-2}\abs{B_j}\\
 		&\leqslant \parens{n-2} + \parens{n-2} + \sum_{j=1}^{n-3}\parens{n-2}\\
 		&=\parens{n-1}\parens{n-2}.
 	\end{align*}
 	
 	\smallskip
 	
 	\textit{Case 3:} Suppose that among the sets $B_1,\ldots,B_{n-2}$ there are at most $n-4$ that are non-empty. Let $t$ be the number of non-empty sets. We have that $0\leqslant t \leqslant n-4$. Assume, without loss of generality, that those $t$ non-empty sets are $B_1,\ldots,B_t$.
 	
 	Let $S'=\gset{\beta|_{X\setminus I}}{\beta\in S}$ and $\overline{S}=\gset{\beta|_I}{\beta\in S}$. We have that $\abs{S}\leqslant \abs{S'}\cdot\abs{\overline{S}}$. In what follows we find upper bounds for $\abs{\overline{S}}$.
 	
 	
 	It is clear that $\overline{S}$ is given by the (disjoint) union of the $t+2$ sets $\gset{\beta|_I}{\beta\in A_1}$ and $\gset{\beta|_I}{\beta\in A_2}$ and $\gset{\beta|_I}{\beta\in B_j}$, where $j\in\X{t}$. (We recall that $B_{t+1}=\cdots=B_{n-2}=\emptyset$ and, consequently, $S=A_1\cup A_2\cup\bigcup_{j=1}^t B_j$.) We are going to see that these sets are singletons.
 	
 	Let $\beta\in A_1$. Then $i_1\beta=i_1$. It follows from part 1 of Lemma~\ref{T(X): lemma upper bound S, |I|=2} that
 	\begin{displaymath}
 		\beta|_I=\begin{pmatrix}
 			i_1 & i_2\\
 			i_1 & i_2
 		\end{pmatrix}.
 	\end{displaymath}
 	Consequently, we have that $\abs{\gset{\beta|_I}{\beta\in A_1}}=1$.
 	
 	Let $\beta\in A_2$. Then $i_1\beta=i_2$ and, consequently, part 2 of Lemma~\ref{T(X): lemma upper bound S, |I|=2} implies that
 	\begin{displaymath}
 		\beta|_I=\begin{pmatrix}
 			i_1 & i_2\\
 			i_2 & i_1
 		\end{pmatrix}.
 	\end{displaymath}
 	Hence $\abs{\gset{\beta|_I}{\beta\in A_2}}=1$.
 	
 	Let $j\in\X{t}$ and $\beta\in B_j$. We have that $i_1\beta=x_j$, which implies, by part 3 of Lemma~\ref{T(X): lemma upper bound S, |I|=2}, that
 	\begin{displaymath}
 		\beta|_I=\begin{pmatrix}
 			i_1 & i_2\\
 			x_j & x_j\alpha
 		\end{pmatrix}.
 	\end{displaymath}
 	This allows us to conclude that $\abs{\gset{\beta|_I}{\beta\in B_j}}=1$ for all $j\in\X{t}$.
 	
 	It follows from the last three paragraphs that
 	\begin{align*}
 		\abs{\overline{S}}&=\abs{\gset{\beta|_I}{\beta\in A_1}}+\abs{\gset{\beta|_I}{\beta\in A_2}}+\sum_{j=1}^t\abs{\gset{\beta|_I}{\beta\in B_j}}\\
 		&= 1+1+\sum_{j=1}^t 1\\
 		&= t+2\\
 		&\leqslant n-2. &\kern -10mm\bracks{\text{since } t\leqslant n-4}
 	\end{align*}
 	Therefore $\abs{S}\leqslant\abs{S'}\cdot\abs{\overline{S}}\leqslant \abs{S'}\cdot\parens{n-2}$.
 	
 	\smallskip
 	
 	The previous three cases allow us to conclude that $\abs{S}\leqslant\max\set{n,\parens{n-1}\parens{n-2},\abs{S'}\cdot\parens{n-2}}$.
 	
 	\medskip
 	
 	\textbf{Part 2.} Now we will prove that, if either $n\leqslant 5$ and $\abs{S'}\leqslant 2^{n-3}$ or $n=6$ and $\abs{S'}<2^{n-3}$, then $\abs{S}<2^{n-1}$.
 	
 	\smallskip
 	
 	\textit{Case 1:} Suppose that $n\leqslant 5$ and $\abs{S'}\leqslant 2^{n-3}$. Then, by part 1,
 	\begin{displaymath}
 		\abs{S}\leqslant\max\set{n,\parens{n-1}\parens{n-2},\abs{S'}\cdot\parens{n-2}}\leqslant\max\set{n,\parens{n-1}\parens{n-2},2^{n-3}\parens{n-2}}.
 	\end{displaymath}
 	So, in order to prove that $\abs{S}<2^{n-1}$, we just need to verify that $n<2^{n-1}$ and $\parens{n-1}\parens{n-2}<2^{n-1}$ and $2^{n-3}\parens{n-2}<2^{n-1}$. We observe that we have $2=\abs{I}<\abs{X}=n\leqslant 5$, which implies that $n\in\set{3,4,5}$. Therefore we just need to verify that these three inequalities hold when $n\in\set{3,4,5}$. From Table~\ref{T(X): table |I|=2} we can easily verify that $n<2^{n-1}$ and $\parens{n-1}\parens{n-2}<2^{n-1}$ and $2^{n-3}\parens{n-2}<2^{n-1}$ when $n\in\set{3,4,5}$. Thus $\abs{S}<2^{n-1}$.
 	
 	\smallskip
 	
 	\textit{Case 2:} Suppose that $n=6$ and $\abs{S'}<2^{n-3}$. From Table~\ref{T(X): table |I|=2} we can conclude that $n<2^{n-1}$ and $\parens{n-1}\parens{n-2}<2^{n-1}$ and $\abs{S'}\cdot\parens{n-2}<2^{n-3}\parens{n-2}=2^{n-1}$. Hence, by part 1,
 	\begin{displaymath}
 		\abs{S}\leqslant\max\set{n,\parens{n-1}\parens{n-2},\abs{S'}\cdot\parens{n-2}}<2^{n-1}.
 	\end{displaymath}
 	
 	
 	\begin{table}[hbt]
 		\centering
 		\begin{tabular}{cccc}
 			\toprule
 			$n$ & $(n-1)(n-2)$ & $2^{n-3}(n-2)$ & $2^{n-1}$ \\
 			\midrule
 			$3$ & $2$          & $1$            & $4$       \\
 			$4$ & $6$          & $4$            & $8$       \\
 			$5$ & $12$         & $12$           & $16$      \\
 			$6$ & $20$         & $32$           & $32$      \\
 			\bottomrule
 		\end{tabular}
 		\caption{Comparison between $n$ and $\parens{n-1}\parens{n-2}$ and $2^{n-3}\parens{n-2}$ and $2^{n-1}$ when $n\in\set{3,4,5,6}$.}
 		\label{T(X): table |I|=2}
 	\end{table}
 	
 	\medskip
 	
 	\textbf{Part 3.} Finally, we will demonstrate that, if $n=6$, then $S'\neq\commidemp{X\setminus\set{I}}{x}$ for all $x\in X\setminus I$. Suppose that $n=6$ and assume, with the aim of obtaining a contradiction, that there exists $x\in X\setminus I=\set{x_1,x_2,x_3,x_4}$ such that $S'=\commidemp{X\setminus I}{x}$. We can assume, without loss of generality, that $x=x_1$. Due to the fact that $S'=\commidemp{X\setminus I}{x_1}$, we have $x_1\beta=x_1$ and $y\beta\in\set{x_1,y}$ for all $\beta\in S$ and $y\in \parens{X\setminus I}\setminus\set{x_1}=\set{x_2,x_3,x_4}$.
 	
 	We have
 	\begin{align*}
 		8 &=2^{\abs{X\setminus I}-1} &\bracks{\text{since } \abs{X}=n=6 \text{ and } \abs{I}=2}\\
 		&=\abs{\commidemp{X\setminus I}{x_1}} &\bracks{\text{by Proposition~\ref{T(X): Gamma^X_i comm smg idemp}}}\\
 		&=\abs{S'}\\
 		&\leqslant\abs{S} &\bracks{\text{since } S'=\gset{\beta|_{X\setminus I}}{\beta\in S}}\\
 		&=\sum_{j=1}^2\abs{A_j}+\sum_{j=1}^4\abs{B_j}. &\bracks{\text{since } \set{A_1,A_2,B_1,B_2,B_3,B_4} \text{ is a partition of } S}    
 	\end{align*}
 	So our goal is to determine the size of $A_1,A_2,B_1,B_2,B_3,B_4$ and then show that their sum is smaller than $8$, which will be a contradiction.
 	
 	Let $j\in\set{2,3,4}$. It follows from the minimality of the size of $I$ that $\set{x_j}\notin\setclass{S}{X}$, which implies the existence of $z_j\in X\setminus\set{x_j}$ and $\beta_j\in S$ such that $z_j\beta_j=x_j$. Additionally, we have $x_1\beta_j=x_1$ and $y\beta_j\in\set{x_1,y}$ for all $y\in \parens{X\setminus I}\setminus\set{x_1}=\set{x_2,x_3,x_4}$. Consequently, we must have $z_j\in\set{i_1,i_2,x_j}\setminus\set{x_j}$; that is, $i_1\beta_j=x_j$ or $i_2\beta_j=x_j$. In what follows we will show that $i_1\beta_j=x_j$. If $i_1\beta_j=x_j$, then there is nothing to prove. If $i_2\beta_j=x_j$, then we have
 	\begin{align*}
 		\parens{i_1\beta_j}\alpha &=\parens{i_1\alpha}\beta_j &\bracks{\text{since } \beta_j,\alpha\in S, \text{ which is commutative}}\\
 		&=i_2\beta_j &\bracks{\text{since } i_1\alpha=i_2}\\
 		&=x_j
 	\end{align*}
 	and, since $\set{i_1,i_2}\alpha=\set{i_1,i_2}$ and $y\alpha\in\set{x_1,y}$ for all $y\in\set{x_2,x_3,x_4}$, by a process of elimination, the only remaining possibility for $i_1\beta_j$ is $x_j$; that is, we must have $i_1\beta_j=x_j$.
 	
 	Since $j$ is an arbitrary element of $\set{2,3,4}$, then we can conclude that there exist $\beta_2,\beta_3,\beta_4\in S$ such that $i_1\beta_2=x_2$ and $i_1\beta_3=x_3$ and $i_1\beta_4=x_4$. We notice that $\beta_2\in B_2$, $\beta_3\in B_3$ and $\beta_4\in B_4$.
 	
 	Now we determine the size of $A_1,A_2,B_1,B_2,B_3,B_4$. Before we do that, we recall that, since $S'=\commidemp{X\setminus I}{x_1}$, then $x_1\beta=x_1$ for all $\beta\in S$.
 	
 	Let $\beta\in A_1$. Then $i_1\beta=i_1$ and, consequently, part 1 of Lemma~\ref{T(X): lemma upper bound S, |I|=2} implies that
 	\begin{displaymath}
 		\beta=\begin{pmatrix}
 			i_1&i_2&x_1&x_2&x_3&x_4\\
 			i_1&i_2&x_1&x_2&x_3&x_4
 		\end{pmatrix}=\id{X}.
 	\end{displaymath}
 	Thus $\abs{A_1}=1$.
 	
 	Let $\beta\in A_2$. Then $i_1\beta=i_2$ and, consequently, part 2 of Lemma~\ref{T(X): lemma upper bound S, |I|=2} implies that
 	\begin{displaymath}
 		\beta=\begin{pmatrix}
 			i_1&i_2&x_1&x_2&x_3&x_4\\
 			i_2&i_1&x_1&x_2\alpha&x_3\alpha&x_4\alpha
 		\end{pmatrix}.
 	\end{displaymath}
 	Thus $\abs{A_2}=1$.
 	
 	Let $j\in\set{1,2,3,4}$ and $\beta\in B_j$. Then $i_1\beta=x_j$ and, consequently, part 3 of Lemma~\ref{T(X): lemma upper bound S, |I|=2} implies that
 	\begin{displaymath}
 		\beta=\begin{pmatrix}
 			i_1&i_2&x_1&x_2&x_3&x_4\\
 			x_j&x_j\alpha&x_1&x_j\beta_2&x_j\beta_3&x_j\beta_4
 		\end{pmatrix}.
 	\end{displaymath}
 	Thus $\abs{B_1}\leqslant 1$ and $\abs{B_j}=1$ for all $j\in\set{2,3,4}$ (notice that we only know that $B_j\neq\emptyset$ for all $j\in\set{2,3,4}$).
 	
 	Therefore
 	\begin{displaymath}
 		8\leqslant\sum_{j=1}^2 \abs{A_j}+\sum_{j=1}^4\abs{B_j}\leqslant 6
 	\end{displaymath}
 	which is a contradiction. Thus $S'\neq\commidemp{X\setminus I}{x}$ for all $x\in X\setminus I$.
 \end{proof}
 
 Assume that $\abs{X}\in\set{3,4,5,6}$ and let $S$ be a maximum-order commutative subsemigroup of $\tr{X}$. Just as the previous lemma is a tool to show (in Theorem~\ref{T(X): maximum size comm smg}) that the minimum size of a set of $\setclass{S}{X}$ is distinct from $2$, Lemma~\ref{T(X): upper bound S, |I|>2} is a tool to show that the minimum size of a set of $\setclass{S}{X}$ cannot be larger than $2$.

 \begin{lemma}\label{T(X): upper bound S, |I|>2}
 	Suppose that $\abs{X}\leqslant 6$. Let $S$ be a commutative subsemigroup of $\tr{X}$ be such that $\setclass{S}{X}\neq\emptyset$ and let $I\in\setclass{S}{X}$ be of minimum size. If $\abs{I}\geqslant 3$, then $\abs{S}\leqslant \abs{S'}\cdot\abs{X}$, where $S'=\gset{\beta|_{X\setminus I}}{\beta\in S}$. Moreover, if we also have $\abs{S'}\leqslant 2^{\abs{X}-\abs{I}-1}$, then $\abs{S}<2^{\abs{X}-1}$.
 \end{lemma}
 
 We observe that it follows from the definition of $\setclass{S}{X}$ and Lemma~\ref{T(X): lemma induction} that $S'$ is a commutative subsemigroup of $\tr{X}$.
 
 \begin{proof}
 	Suppose that $\abs{I}\geqslant 3$. Let $m=\abs{I}$ and assume that $I=\set{i_1,\ldots,i_m}$.
 	
 	Let $S'=\gset{\beta|_{X\setminus I}}{\beta\in S}$ and $\overline{S}=\gset{\beta|_I}{\beta\in S}$. It is clear that $\abs{S}\leqslant \abs{S'}\cdot\abs{\overline{S}}$. In what follows we determine an upper bound for $\abs{\overline{S}}$.
 	
 	Due to the fact that $I$ is of minimum size and $\abs{I}\geqslant 2$, we have, by Lemma~\ref{T(X): lemma alpha|_I cycle}, that there exists $\alpha\in S$ such that $\alpha|_I\in\sym{I}$ is a product of (disjoint) cycles of the same length, which must be at least $2$. This motivates the division of the proof into two cases.
 	
 	\smallskip
 	
 	\textit{Case 1:} Suppose that $\alpha|_I$ is a cycle (of length $m=\abs{I}$). Assume that
 	\begin{displaymath}
 		\alpha|_I=\begin{pmatrix}
 			i_1&i_2&\cdots&i_{m-1}&i_m\\
 			i_2&i_3&\cdots&i_m&i_1
 		\end{pmatrix}.
 	\end{displaymath}
 	
 	Let $\beta\in S$ and let $y\in X$ be such that $i_1\beta=y$. For all $k\in\set{2,\ldots,m}$ we have
 	\begin{align*}
 		i_k\beta&=\parens{i_1\alpha^{k-1}}\beta &\bracks{\text{since } i_1\alpha^{k-1}=i_k}\\
 		&=\parens{i_1\beta}\alpha^{k-1} &\bracks{\text{since } \alpha,\beta\in S, \text{ which is commutative}}\\
 		&=y\alpha^{k-1}. &\bracks{\text{since } i_1\beta=y}
 	\end{align*}
 	Since $\beta$ is an arbitrary element of $S$, then $\beta|_I$ is an arbitrary element of $\overline{S}$ and, consequently, we can conclude that
 	\begin{displaymath}
 		\overline{S}\subseteq\gset*{\begin{pmatrix}
 				i_1&i_2&i_3&\cdots&i_m\\
 				y&y\alpha&y\alpha^2&\cdots&y\alpha^{m-1}
 		\end{pmatrix}}{y\in X}.
 	\end{displaymath}
 	Therefore $\abs{\overline{S}}\leqslant \abs{X}$ and, thus, $\abs{S}\leqslant \abs{S'}\cdot\abs{\overline{S}}\leqslant \abs{S'}\cdot\abs{X}$.
 	
 	\smallskip
 	
 	\textit{Case 2:} Suppose that $\alpha|_I$ is not a cycle. Then $\alpha|_I$ is a product of at least two (disjoint) cycles, all of which have the same length, which is greater than $1$. Hence the number of cycles in $\alpha|_I$ divides $\abs{I}$ and, consequently, $\abs{I}$ cannot be a prime number. Furthermore, we have that $2\leqslant \abs{I}<\abs{X}\leqslant 6$ and $2$, $3$ and $5$ are prime numbers. Thus $\abs{I}=4$. Assume that
 	\begin{displaymath}
 		\alpha|_I=\begin{pmatrix}
 			i_1&i_2&i_3&i_4\\
 			i_2&i_1&i_4&i_3
 		\end{pmatrix}.
 	\end{displaymath}
 	Moreover, it follows from the fact that $I$ is of minimum size that $\set{i_3,i_4}\notin\setclass{S}{X}$, which implies that there exists $\gamma\in S$ such that $\parens{X\setminus\set{i_3,i_4}}\gamma\nsubseteq X\setminus\set{i_3,i_4}$. Since we also have $\parens{X\setminus\set{i_1,i_2,i_3,i_4}}\gamma=\parens{X\setminus I}\gamma\subseteq X\setminus I\subseteq X\setminus\set{i_3,i_4}$, then we can conclude that $\set{i_1,i_2}\gamma\nsubseteq X\setminus\set{i_3,i_4}$; that is, $\set{i_1,i_2}\gamma\cap\set{i_3,i_4}\neq\emptyset$. Assume, without loss of generality, that $i_1\gamma=i_3$.
 	
 	Let $\beta\in S$. Let $y\in X$ be such that $i_1\beta=y$. We have
 	\begin{align*}
 		i_2\beta &=\parens{i_1\alpha}\beta &\bracks{\text{since } i_1\alpha=i_2}\\
 		&=\parens{i_1\beta}\alpha &\bracks{\text{since } \alpha,\beta\in S, \text{ which is commutative}}\\
 		&=y\alpha &\bracks{\text{since } i_1\beta=y}
 	\end{align*}
 	and
 	\begin{align*}
 		i_3\beta &=\parens{i_1\gamma}\beta &\bracks{\text{since } i_1\gamma=i_3}\\
 		&=\parens{i_1\beta}\gamma &\bracks{\text{since } \gamma,\beta\in S, \text{ which is commutative}}\\
 		&=y\gamma &\bracks{\text{since } i_1\beta=y}
 	\end{align*}
 	and
 	\begin{align*}
 		i_4\beta &=\parens{i_3\alpha}\beta &\bracks{\text{since }  i_3\alpha=i_4}\\
 		&=\parens{i_3\beta}\alpha &\bracks{\text{since } \alpha,\beta\in S, \text{ which is commutative}}\\
 		&=\parens{y\gamma}\alpha. &\bracks{\text{since } i_3\beta=y\gamma}
 	\end{align*}
 	It follows from the fact that $\beta$ is an arbitrary element of $S$ that $\beta|_I$ is an arbitrary element of $\overline{S}$. This means we just proved that
 	\begin{displaymath}
 		\overline{S}\subseteq\gset*{\begin{pmatrix}
 				i_1&i_2&i_3&i_4\\
 				y&y\alpha&y\gamma&y\gamma\alpha
 		\end{pmatrix}}{y\in X}.
 	\end{displaymath}
 	Therefore $\abs{S}\leqslant \abs{S'}\cdot\abs{\overline{S}}\leqslant \abs{S'}\cdot\abs{X}$.
 	
 	\smallskip
 	
 	In both cases we concluded that $\abs{S}\leqslant \abs{S'}\cdot\abs{X}$. Now we will see that, if $\abs{S'}\leqslant 2^{\abs{X}-\abs{I}-1}$, then $\abs{S}<2^{\abs{X}-1}$. Suppose that $\abs{S'}\leqslant 2^{\abs{X}-\abs{I}-1}$. Since $\abs{I}\geqslant 3$, then we have
 	\begin{displaymath}
 		\abs{S}\leqslant \abs{S'}\cdot\abs{X}\leqslant 2^{\abs{X}-\abs{I}-1}\abs{X} \leqslant 2^{\abs{X}-3-1}\abs{X} =2^{\abs{X}-4}\abs{X}.
 	\end{displaymath}
 	Consequently, in order to prove that $\abs{S}<2^{\abs{X}-1}$, we just need to establish that $2^{\abs{X}-4}\abs{X}<2^{\abs{X}-1}$.
 	
 	We observe that we have $3\leqslant\abs{I}<\abs{X}\leqslant 6$, which implies that $\abs{X}\in\set{4,5,6}$. In Table~\ref{T(X): table |I|>2} we can verify that, when $\abs{X}$ is one of these three integers, then $2^{\abs{X}-4}\abs{X}<2^{\abs{X}-1}$.
 	\begin{table}[hbt]
 		\centering
 		\begin{tabular}{ccc}
 			\toprule
 			$|X|$ & $2^{|X|-4}\abs{X}$ & $2^{|X|-1}$ \\
 			\midrule
 			$4$   & $4$         & $8$         \\
 			$5$   & $10$        & $16$        \\
 			$6$   & $24$        & $32$        \\
 			\bottomrule
 		\end{tabular}
 		\caption{Comparison between $2^{\abs{X}-4}\abs{X}$ and $2^{\abs{X}-1}$ when $\abs{X}\in\set{4,5,6}$.}
 		\label{T(X): table |I|>2}
 	\end{table}
 	
 	This concludes the proof that $\abs{S}<2^{\abs{X}-1}$.
 \end{proof}
 
 Our next goal is to demonstrate that, when $\abs{X}\in\set{4,5,6}$ and $S$ is a commutative subsemigroup of $\tr{X}$ such that $\set{i}\in\setclass{S}{X}$ (where $i\in X$), then the existence of at least three `copies' of an element of $\gset{\beta|_{X\setminus\set{i}}}{\beta\in S}$ in $S$ --- that is, the existence of at least three distinct elements of $S$ that are equal in $X\setminus\set{i}$ --- imply that $S$ is not one of the largest commutative subsemigroup of $\tr{X}$. This will be proved in Lemmata~\ref{T(X): |X|=4, upper bound S when 3 copies of alpha}, \ref{T(X): |X|=5, upper bound S when 3 copies of alpha} and \ref{T(X): |X|=6, upper bound S when 3 copies of alpha} (for $\abs{X}=4$, $\abs{X}=5$ and $\abs{X}=6$, respectively). Before we do this we need another two results --- Lemmata~\ref{T(X): lemma copies of alphas's} and \ref{T(X): lemma upper bound S when 3 copies of alpha} --- which show how the existence of `copies' of an element of $\gset{\beta|_{X\setminus\set{i}}}{\beta\in S}$ in $S$ restricts the transformations of $S$.
 
 \begin{lemma}\label{T(X): lemma copies of alphas's}
 	Let $S$ be a commutative subsemigroup of $\tr{X}$ such that $\setclass{S}{X}\neq\emptyset$. Suppose that there exists $i\in X$ such that $\set{i}\in\setclass{S}{X}$. If there exist distinct $\alpha_1,\alpha_2\in S$ such that $\alpha_1|_{X\setminus\set{i}}=\alpha_2|_{X\setminus\set{i}}$ and $i\alpha_1=i$, then for all distinct $\beta_1,\beta_2\in S$ such that $\beta_1|_{X\setminus\set{i}}=\beta_2|_{X\setminus\set{i}}$ we have $i\beta_1=i$ or $i\beta_2=i$.
 \end{lemma}
 
 \begin{proof}
 	Suppose that there exist distinct $\alpha_1,\alpha_2\in S$ such that $\alpha_1|_{X\setminus\set{i}}=\alpha_2|_{X\setminus\set{i}}$ and $i\alpha_1=i$. Then $i\alpha_2\neq i\alpha_1=i$. Let $\beta_1,\beta_2\in S$ be such that $\beta_1\neq\beta_2$ and $\beta_1|_{X\setminus\set{i}}=\beta_2|_{X\setminus\set{i}}$. Hence $i\beta_1\neq i\beta_2$ and, consequently, we have $i\beta_1\neq i$ or $i\beta_2 \neq i$. Assume, without loss of generality, that $i\beta_1\neq i$. We have that
 	\begin{align*}
 		\parens{i\beta_2}\alpha_1 &=\parens{i\alpha_1}\beta_2 &\bracks{\text{since } \beta_2,\alpha_1\in S, \text{ which is commutative}}\\
 		&=i\beta_2 &\bracks{\text{since } i\alpha_1=i}\\
 		&\neq i\beta_1\\
 		&=\parens{i\alpha_1}\beta_1 &\bracks{\text{since } i\alpha_1=i}\\
 		&=\parens{i\beta_1}\alpha_1 &\bracks{\text{since } \alpha_1,\beta_1\in S, \text{ which is commutative}}\\
 		&=\parens{i\beta_1}\alpha_1|_{X\setminus\set{i}} &\bracks{\text{since } i\beta_1\neq i}\\
 		&=\parens{i\beta_1}\alpha_2|_{X\setminus\set{i}} &\bracks{\text{since } \alpha_1|_{X\setminus\set{i}}=\alpha_2|_{X\setminus\set{i}}}\\
 		&=\parens{i\beta_1}\alpha_2\\
 		&=\parens{i\alpha_2}\beta_1 &\bracks{\text{since } \beta_1,\alpha_2\in S \text{ which is commutative}}\\
 		&=\parens{i\alpha_2}\beta_1|_{X\setminus\set{i}} &\bracks{\text{since } i\alpha_2\neq i}\\
 		&=\parens{i\alpha_2}\beta_2|_{X\setminus\set{i}} &\bracks{\text{since } \beta_1|_{X\setminus\set{i}}=\beta_2|_{X\setminus\set{i}}}\\
 		&=\parens{i\alpha_2}\beta_2\\
 		&=\parens{i\beta_2}\alpha_2. &\bracks{\text{since } \alpha_2,\beta_2\in S \text{ which is commutative}}
 	\end{align*}
 	As a consequence of the fact that $\alpha_1|_{X\setminus\set{i}}=\alpha_2|_{X\setminus\set{i}}$, we must have $i\beta_2=i$, which concludes the proof.
 \end{proof}
 
 \begin{lemma}\label{T(X): lemma upper bound S when 3 copies of alpha}
 	Let $S$ be a commutative subsemigroup of $\tr{X}$ such that $\setclass{S}{X}\neq\emptyset$. Suppose that there exists $i\in X$ such that $\set{i}\in\setclass{S}{X}$ and that there exist pairwise distinct $\alpha_1,\alpha_2,\alpha_3\in S$ such that $\alpha_1|_{X\setminus\set{i}}=\alpha_2|_{X\setminus\set{i}}=\alpha_3|_{X\setminus\set{i}}$. For each $j\in\set{1,2,3}$ let $x_j=i\alpha_j$ and let $x=x_1\alpha_1$. Then
 	\begin{enumerate}
 		\item We have that $x_1,x_2,x_3\in X\setminus\set{i}$ and are pairwise distinct.
 		
 		\item For all $\beta\in S$, if $i\beta\in\set{x_1,x_2,x_3}$, then $x_1\beta=x_2\beta=x_3\beta=x$.
 		
 		\item For all $\beta\in S$, if $i\beta=i$, then $x_1\beta=x_1$, $x_2\beta=x_2$ and $x_3\beta=x_3$.
 		
 		\item For all $\beta\in S$, if $i\beta\in X\setminus\set{x_1,x_2,x_3,i}$, then $x_1\beta=x_2\beta=x_3\beta=\parens{i\beta}\alpha_1$.
 	\end{enumerate}
 \end{lemma}
 
 \begin{proof}
 	\textbf{Part 1.} Due to the fact that $\alpha_1,\alpha_2,\alpha_3$ are pairwise distinct and $\alpha_1|_{X\setminus\set{i}}=\alpha_2|_{X\setminus\set{i}}=\alpha_3|_{X\setminus\set{i}}$, we have that $i\alpha_1,i\alpha_2,i\alpha_3$ are pairwise distinct (that is, $x_1,x_2,x_3$ are pairwise distinct). This implies that at least two of them are not equal to $i$ and, consequently, (the contrapositive of) Lemma~\ref{T(X): lemma copies of alphas's} guarantees that for all distinct $\beta_1,\beta_2\in S$ such that $\beta_1|_{X\setminus\set{i}}=\beta_2|_{X\setminus\set{i}}$ we have $i\beta_1,i\beta_2\in X\setminus\set{i}$. In particular, we have $i\alpha_1,i\alpha_2,i\alpha_3\in X\setminus\set{i}$; that is, $x_1,x_2,x_3\in X\setminus\set{i}$.
 	
 	\medskip
 	
 	\textbf{Part 2.} Let $\beta\in S$ and assume that $i\beta\in\set{x_1,x_2,x_3}$. Let $k\in\set{1,2,3}$ be such that $i\beta=x_k$. For all $j\in\set{1,2,3}$ we have
 	\begin{align*}
 		x_j\beta &=\parens{i\alpha_j}\beta &\bracks{\text{since } x_j=i\alpha_j}\\
 		&=\parens{i\beta}\alpha_j &\bracks{\text{since } \alpha_j,\beta\in S, \text{ which is commutative}}\\
 		&=\parens{i\beta}\alpha_j|_{X\setminus\set{i}} &\bracks{\text{since, by part 1, } i\beta=x_k\in X\setminus\set{i}}\\
 		&=\parens{i\beta}\alpha_1|_{X\setminus\set{i}} &\bracks{\text{since } \alpha_1|_{X\setminus\set{i}}=\alpha_j|_{X\setminus\set{i}}}\\
 		&=\parens{i\beta}\alpha_1\\
 		&=\parens{i\alpha_k}\alpha_1 &\bracks{\text{since } i\beta=x_k=i\alpha_k}\\
 		&=\parens{i\alpha_1}\alpha_k &\bracks{\text{since } \alpha_k,\alpha_1\in S, \text{ which is commutative}}\\
 		&=x_1\alpha_k &\bracks{\text{since } x_1=i\alpha_1}\\
 		&=x_1\alpha_k|_{X\setminus\set{i}} &\bracks{\text{since, by part 1, } x_1\in X\setminus\set{i}}\\
 		&=x_1\alpha_1|_{X\setminus\set{i}} &\bracks{\text{since } \alpha_1|_{X\setminus\set{i}}=\alpha_k|_{X\setminus\set{i}}}\\
 		&=x_1\alpha_1\\
 		&=x.
 	\end{align*}
 	
 	\medskip

 	\textbf{Part 3.} Let $\beta\in S$ and assume that $i\beta=i$. Then, for all $j\in\set{1,2,3}$, we have
 	\begin{align*}
 		x_j\beta &=\parens{i\alpha_j}\beta &\bracks{\text{since } x_j=i\alpha_j}\\
 		&=\parens{i\beta}\alpha_j &\bracks{\text{since } \alpha_j,\beta\in S, \text{ which is commutative}}\\
 		&=i\alpha_j &\bracks{\text{since } i\beta=i}\\
 		&=x_j.
 	\end{align*}
 	
 	\medskip
 	
 	\textbf{Part 4.} Let $\beta\in S$ and assume that $i\beta\in X\setminus\set{x_1,x_2,x_3,i}$. Then, for all $j\in\set{1,2,3}$, we have
 	\begin{align*}
 		x_j\beta &=\parens{i\alpha_j}\beta &\bracks{\text{since } x_j=i\alpha_j}\\
 		&=\parens{i\beta}\alpha_j &\bracks{\text{since } \alpha_j,\beta\in S, \text{ which is commutative}}\\
 		&=\parens{i\beta}\alpha_j|_{X\setminus\set{i}} &\bracks{\text{since } i\beta\in X\setminus\set{x_1,x_2,x_3,i}\subseteq X\setminus\set{i}}\\
 		&=\parens{i\beta}\alpha_1|_{X\setminus\set{i}} &\bracks{\text{since } \alpha_1|_{X\setminus\set{i}}=\alpha_j|_{X\setminus\set{i}}}\\
 		&=\parens{i\beta}\alpha_1. & & \qedhere
 	\end{align*}
 \end{proof}
 
 Now we are ready to prove that, when $\abs{X}=4$, the existence of three `copies' of an element of $\gset{\beta|_{X\setminus\set{i}}}{\beta\in S}$ in $S$ ensures that $S$ cannot be a maximum-order commutative subsemigroup of $\tr{X}$. (Recall that in Proposition~\ref{T(X): Gamma^X_i comm smg idemp} we saw that there exist commutative subsemigroups of $\tr{X}$ of size $2^{\abs{X}-1}$.)
 
 \begin{lemma}\label{T(X): |X|=4, upper bound S when 3 copies of alpha}
 	Suppose that $\abs{X}=4$. Let $S$ be a commutative subsemigroup of $\tr{X}$ such that $\setclass{S}{X}\neq\emptyset$. Suppose that there exists $i\in X$ such that $\set{i}\in\setclass{S}{X}$. If there exist pairwise distinct $\alpha_1,\alpha_2,\alpha_3\in S$ such that $\alpha_1|_{X\setminus\set{i}}=\alpha_2|_{X\setminus\set{i}}=\alpha_3|_{X\setminus\set{i}}$, then $\abs{S}<2^{\abs{X}-1}$.
 \end{lemma}
 
 \begin{proof}
 	Suppose that there exist pairwise distinct $\alpha_1,\alpha_2,\alpha_3\in S$ such that $\alpha_1|_{X\setminus\set{i}}=\alpha_2|_{X\setminus\set{i}}=\alpha_3|_{X\setminus\set{i}}$. For each $j\in\set{1,2,3}$ let $x_j=i\alpha_j$ and let $x=x_1\alpha_1$.
 	
 	It follows from part 1 of Lemma~\ref{T(X): lemma upper bound S when 3 copies of alpha} that $x_1,x_2,x_3\in X\setminus\set{i}$ and are pairwise distinct. Then, since $\abs{X}=4$, we have $X=\set{x_1,x_2,x_3,i}$.
 	
 	Let $\beta\in S$. We analyse two cases.
 	
 	\smallskip
 	
 	\textit{Case 1:} Assume that $i\beta\in X\setminus\set{i}=\set{x_1,x_2,x_3}$. Then, by part 2 of Lemma~\ref{T(X): lemma upper bound S when 3 copies of alpha}, we have that $x_1\beta=x_2\beta=x_3\beta=x$. Furthermore, since for all $j\in\set{1,2,3}$ we have $i\alpha_j=x_j\in X\setminus\set{i}$, then part 2 of Lemma~\ref{T(X): lemma upper bound S when 3 copies of alpha} also implies that for all $j\in\set{1,2,3}$ we have $x_1\alpha_j=x_2\alpha_j=x_3\alpha_j=x$. This implies that $\beta|_{X\setminus\set{i}}=\alpha_1|_{X\setminus\set{i}}=\alpha_2|_{X\setminus\set{i}}=\alpha_3|_{X\setminus\set{i}}$. Moreover, $i\beta\in\set{x_1,x_2,x_3}=\set{i\alpha_1,i\alpha_2,i\alpha_3}$ and, consequently, $\beta\in\set{\alpha_1,\alpha_2,\alpha_3}$.
 	
 	\smallskip
 	
 	\textit{Case 2:} Assume that $i\beta=i$. Then part 3 of Lemma~\ref{T(X): lemma upper bound S when 3 copies of alpha} ensures that $x_1\beta=x_1$, $x_2\beta=x_2$ and $x_3\beta=x_3$. Thus $\beta=\id{X}$.
 	
 	\smallskip
 	
 	Since $\beta$ is an arbitrary element of $S$, then we can conclude that $S\subseteq\set{\alpha_1,\alpha_2,\alpha_3,\id{X}}$ and, consequently,
 	\begin{displaymath}
 		\abs{S}\leqslant 4<2^{4-1}=2^{\abs{X}-1}. \qedhere
 	\end{displaymath}
 \end{proof}
 
 In the next lemma we establish that, when $\abs{X}=5$, the existence of three `copies' of an element of $\gset{\beta|_{X\setminus\set{i}}}{\beta\in S}$ in $S$ implies that $S$ is not a maximum-order commutative subsemigroup of $\tr{X}$.
 
 \begin{lemma}\label{T(X): |X|=5, upper bound S when 3 copies of alpha}
 	Suppose that $\abs{X}=5$. Let $S$ be a commutative subsemigroup of $\tr{X}$ such that $\setclass{S}{X}\neq\emptyset$. Suppose that there exists $i\in X$ such that $\set{i}\in\setclass{S}{X}$. If there exist pairwise distinct $\alpha_1,\alpha_2,\alpha_3\in S$ such that $\alpha_1|_{X\setminus\set{i}}=\alpha_2|_{X\setminus\set{i}}=\alpha_3|_{X\setminus\set{i}}$, then $\abs{S}<2^{\abs{X}-1}$.
 \end{lemma}
 
 \begin{proof}
 	Suppose that there exist pairwise distinct $\alpha_1,\alpha_2,\alpha_3\in S$ such that $\alpha_1|_{X\setminus\set{i}}=\alpha_2|_{X\setminus\set{i}}=\alpha_3|_{X\setminus\set{i}}$. For each $j\in\set{1,2,3}$ let $x_j=i\alpha_j$ and let $x=x_1\alpha_1$.
 	
 	Part 1 of Lemma~\ref{T(X): lemma upper bound S when 3 copies of alpha} guarantees that $x_1,x_2,x_3\in X\setminus\set{i}$ and are pairwise distinct. Moreover, $\abs{X}=5$. Hence there exists $x_4\in X$ such that $X=\set{x_1,x_2,x_3,x_4,i}$. Let
 	\begin{align*}
 		A_1&=\gset[\Bigg]{\begin{pmatrix}
 				x_1 & x_2 & x_3 & x_4 & i \\
 				x & x & x & x_k & x_j
 		\end{pmatrix}}{k\in \set{1,2,3,4} \text{ and } j\in\set{1,2,3}};\\[2mm]
 		A_2&=\gset[\Bigg]{\begin{pmatrix}
 				x_1 & x_2 & x_3 & x_4 & i \\
 				x_4\alpha_1 & x_4\alpha_1 & x_4\alpha_1 & x_k & x_4
 		\end{pmatrix}}{k\in \set{1,2,3,4}};\\[2mm]
 		A_3&=\gset[\Bigg]{\begin{pmatrix}
 				x_1 & x_2 & x_3 & x_4 & i \\
 				x_1 & x_2 & x_3 & x_k & i
 		\end{pmatrix}}{k\in \set{1,2,3,4}}.
 	\end{align*}
 	
 	First we are going to check that $S\subseteq A_1\cup A_2 \cup A_3$. Let $\beta\in S$. We have that $\set{i}\in\setclass{S}{X}$, which implies that $\beta|_{X\setminus\set{i}}\in\tr{X\setminus\set{i}}$. Consequently, $x_4\beta\in X\setminus\set{i}=\set{x_1,x_2,x_3,x_4}$. We consider the three cases below.
 	
 	
 	
 	
 	\smallskip
 	
 	\textit{Case 1:} Assume that $i\beta\in\set{x_1,x_2,x_3}$. It follows from part 2 of Lemma~\ref{T(X): lemma upper bound S when 3 copies of alpha} that $x_1\beta=x_2\beta=x_3\beta=x$. Thus $\beta\in A_1$.
 	
 	\smallskip
 	
 	\textit{Case 2:} Assume that $i\beta=i$. It follows from part 3 of Lemma~\ref{T(X): lemma upper bound S when 3 copies of alpha} that $x_1\beta=x_1$, $x_2\beta=x_2$ and $x_3\beta=x_3$. Thus $\beta\in A_3$.
 	
 	\smallskip
 	
 	\textit{Case 3:} Assume that $i\beta\in X\setminus\set{x_1,x_2,x_3,i}=\set{x_4}$. It follows from part 4 of Lemma~\ref{T(X): lemma upper bound S when 3 copies of alpha} that $x_1\beta=x_2\beta=x_3\beta=\parens{i\beta}\alpha_1=x_4\alpha_1$. Thus $\beta\in A_2$.
 	
 	\smallskip
 	
 	The previous three cases allow us to conclude that $S\subseteq A_1\cup A_2\cup A_3$.
 	
 	It follows from the fact that $\abs{X}=5$ that $X\setminus\set{i,x,x_4,x_4\alpha_1}\neq \emptyset$. Let $y\in X\setminus\set{i,x,x_4,x_4\alpha_1}$. For each $j\in\set{1,2,3,4}$ we define
 	\begin{gather*}
 		\beta_j=\begin{pmatrix}
 			x_1 & x_2 & x_3 & x_4 & i \\
 			x_j\alpha_1 & x_j\alpha_1 & x_j\alpha_1 & x_4 & x_j
 		\end{pmatrix}\\
 		\shortintertext{and}
 		\gamma_j=\begin{pmatrix}
 			x_1 & x_2 & x_3 & x_4 & i \\
 			x_j\alpha_1 & x_j\alpha_1 & x_j\alpha_1 & y & x_j
 		\end{pmatrix}
 		\shortintertext{and we define}
 		\delta=\begin{pmatrix}
 			x_1 & x_2 & x_3 & x_4 & i \\
 			x_1 & x_2 & x_3 & y & i
 		\end{pmatrix}.
 	\end{gather*}
 	Let $B_1=\set{\beta_1,\beta_2,\beta_3,\beta_4}$ and $B_2=\set{\gamma_1,\gamma_2,\gamma_3,\gamma_4}$ and $B_3=\set{\delta}$. It is straightforward to verify that these transformations are pairwise distinct. Furthermore, the fact that $x_1\alpha_1=x_2\alpha_1=x_3\alpha_1=x$ and $y\in X\setminus\set{i}$ implies that $\beta_1,\beta_2,\beta_3,\gamma_1,\gamma_2,\gamma_3\in A_1$, $\beta_4,\gamma_4\in A_2$ and $\delta\in A_3$.
 	
 	Since $x=x_1\alpha_1=x_2\alpha_1=x_3\alpha_1$ and $y\in X\setminus\set{x,x_4\alpha_1}$, then we have that $y\neq x_j\alpha_1$ for all $j\in\set{1,2,3,4}$. In addition, we have $y\in X\setminus\set{i,x_4}$, which implies that $y\in\set{x_1,x_2,x_3}$. Consequently, for all $j,k\in\set{1,2,3,4}$, we have
 	\begin{gather*}
 		x_4\beta_j\gamma_k =x_4\gamma_k =y\neq x_j\alpha_1 =y\beta_j =x_4\gamma_k\beta_j\\
 		\shortintertext{and}
 		x_4\beta_j\delta =x_4\delta =y \neq x_j\alpha_1 =y\beta_j =x_4\delta\beta_j\\
 		\shortintertext{and}
 		x_4\gamma_k\delta =y\delta =y \neq x_j\alpha_1 =y\gamma_k =x_4\delta\gamma_k,
 	\end{gather*}
 	which implies that $\beta_j\gamma_k\neq\gamma_k\beta_j$, $\beta_j\delta\neq\delta\beta_j$ and $\gamma_k\delta\neq\delta\gamma_k$ for all $j,k\in\set{1,2,3,4}$. Since $S$ is commutative, then this means that among the sets $S\cap B_1$, $S\cap B_2$ and $S\cap B_3$ there is at most one that is non-empty. Hence there exist distinct $j,k\in\set{1,2,3}$ such that $S\cap\parens{B_j\cup B_k}=\emptyset$ and, consequently, we have $S\subseteq \parens{A_1\cup A_2\cup A_3}\setminus\parens{B_j\cup B_k}$.
 	
 	Therefore, noting that $B_j\cup B_k\subseteq A_1\cup A_2\cup A_3$,
 	\begin{align*}
 		\abs{S} &\leqslant \abs{A_1}+\abs{A_2}+\abs{A_3}-\parens{\abs{B_j}+\abs{B_k}}\\
 		&\leqslant \abs{A_1}+\abs{A_2}+\abs{A_3}-\min\set{\abs{B_1}+\abs{B_2},\abs{B_1}+\abs{B_3},\abs{B_2}+\abs{B_3}}\\
 		&= 12+4+4-\min\set{4+4,4+1,4+1}\\
 		&= 15\\
 		&< 2^{5-1}\\
 		&= 2^{\abs{X}-1}. \qedhere
 	\end{align*}
 \end{proof}
 
 Finally, we will demonstrate that, if $\abs{X}=6$ and $S$ contains three `copies' of an element of $\gset{\beta|_{X\setminus\set{i}}}{\beta\in S}$, then $S$ is not a commutative subsemigroup of $\tr{X}$ of maximum size.
 
 \begin{lemma}\label{T(X): |X|=6, upper bound S when 3 copies of alpha}
 	Suppose that $\abs{X}=6$. Let $S$ be a commutative subsemigroup of $\tr{X}$ such that $\setclass{S}{X}\neq\emptyset$. Suppose that there exists $i\in X$ such that $\set{i}\in\setclass{S}{X}$. If there exist pairwise distinct $\alpha_1,\alpha_2,\alpha_3\in S$ such that $\alpha_1|_{X\setminus\set{i}}=\alpha_2|_{X\setminus\set{i}}=\alpha_3|_{X\setminus\set{i}}$, then $\abs{S}<2^{\abs{X}-1}$.
 \end{lemma}
 
 \begin{proof}
 	Suppose that there exist pairwise distinct $\alpha_1,\alpha_2,\alpha_3\in S$ such that $\alpha_1|_{X\setminus\set{i}}=\alpha_2|_{X\setminus\set{i}}=\alpha_3|_{X\setminus\set{i}}$. For each $j\in\set{1,2,3}$ let $x_j=i\alpha_j$ and let $x=x_1\alpha_1$.
 	
 	
 	
 	We have that $x_1,x_2,x_3\in X\setminus\set{i}$ and are pairwise distinct (by part 1 of Lemma~\ref{T(X): lemma upper bound S when 3 copies of alpha}). Then, since $\abs{X}=6$, we have that $X=\set{x_1,x_2,x_3,x_4,x_5,i}$ for some $x_4,x_5\in X$. Let
 	\begin{align*}
 		A_1&=\gset[\Bigg]{\begin{pmatrix}
 				x_1 & x_2 & x_3 & x_4 & x_5 & i \\
 				x & x & x & x_k & x_m & x_j
 		\end{pmatrix}}{k,m,j\in\set{1,2,3}};\\[2mm]
 		A_2&=\gsetsplit[\Bigg]{\begin{pmatrix}
 				x_1 & x_2 & x_3 & x_4 & x_5 & i \\
 				x & x & x & x_k & x_m & x_j
 		\end{pmatrix}}{k,m\in\set{1,2,3,4,5}\\[-10pt]
 			&\qquad\qquad\qquad\qquad\qquad\qquad\qquad\qquad\qquad\quad \text{ and } \set{k,m}\cap\set{4,5}\neq\emptyset\\[-10pt]
 			&\qquad\qquad\qquad\qquad\qquad\qquad\qquad\qquad\qquad\qquad\qquad\qquad \text{ and } j\in\set{1,2,3}};\\[2mm]
 		A_3&=\gset[\Bigg]{\begin{pmatrix}
 				x_1 & x_2 & x_3 & x_4 & x_5 & i \\
 				x_j\alpha_1 & x_j\alpha_1 & x_j\alpha_1 & x_k & x_m & x_j
 		\end{pmatrix}}{k,m\in\set{1,2,3,4,5} \text{ and } j\in\set{4,5}};\\[2mm]
 		A_4&=\gset[\Bigg]{\begin{pmatrix}
 				x_1 & x_2 & x_3 & x_4 & x_5 & i \\
 				x_1 & x_2 & x_3 & x_k & x_m & i
 		\end{pmatrix}}{k,m\in\set{1,2,3}};\\[2mm]
 		A_5&=\gset[\Bigg]{\begin{pmatrix}
 				x_1 & x_2 & x_3 & x_4 & x_5 & i \\
 				x_1 & x_2 & x_3 & x_k & x_m & i
 		\end{pmatrix}}{k\in\set{1,2,3} \text{ and } m\in\set{4,5}};\\[2mm]
 		A_6&=\gset[\Bigg]{\begin{pmatrix}
 				x_1 & x_2 & x_3 & x_4 & x_5 & i \\
 				x_1 & x_2 & x_3 & x_k & x_m & i
 		\end{pmatrix}}{k\in\set{4,5} \text{ and } m\in\set{1,2,3}};\\[2mm]
 		A_7&=\gset[\Bigg]{\begin{pmatrix}
 				x_1 & x_2 & x_3 & x_4 & x_5 & i \\
 				x_1 & x_2 & x_3 & x_k & x_m & i
 		\end{pmatrix}}{k,m\in\set{4,5}}.
 	\end{align*}
 	
 	We divide the remainder of the proof into several parts: in the first part we will see that determining $\abs{S}$ can be accomplished by determining $\abs{S\cap A_j}$ for all $j\in\set{1,2,3,4,5,6,7}$; parts 2--10 concern the size of the sets $\abs{S\cap A_j}$ for all $j\in\set{1,2,3,4,5,6,7}$; and in the last part we prove the desired result, that is, we prove that $\abs{S}<2^{\abs{X}-1}$. In summary, the eleven parts of the proof establish the following:
 	\begin{enumerate}
 		\item $\abs{S}=\sum_{j=1}^7 \abs{S\cap A_j}$.
 		\item If $S\cap A_2\neq\emptyset$, then $\abs{S\cap A_1}\leqslant 9$.
 		\item If $S\cap A_3\neq\emptyset$, then $\abs{S\cap A_1}\leqslant 9$.
 		\item If $S\cap A_4\neq\emptyset$, then $\abs{S\cap A_1}\leqslant 3$.
 		\item $\abs{S\cap A_2}\leqslant 12$.
 		\item $\abs{S\cap A_3}\leqslant 5$.
 		\item $\abs{S\cap A_4}\leqslant 1$.
 		\item $\abs{S\cap A_5}\leqslant 1$.
 		\item $\abs{S\cap A_6}\leqslant 1$.
 		\item $\abs{S\cap A_7}\leqslant 2$.
 		\item $\abs{S}<2^{\abs{X}-1}$.
 	\end{enumerate}
 	
 	\medskip
 	
 	\textbf{Part 1.} In what follows we establish that $\abs{S}=\sum_{j=1}^7 \abs{S\cap A_j}$. In order to do this we first demonstrate that $S\subseteq \bigcup_{j=1}^7 A_j$. Let $\beta\in S$. As a consequence of the fact that $\set{i}\in\setclass{S}{X}$, we have that $\beta|_{X\setminus\set{i}}\in\tr{X\setminus\set{i}}$, which implies that $x_4\beta,x_5\beta\in X\setminus\set{i}=\set{x_1,x_2,x_3,x_4,x_5}$.
 	
 	\smallskip
 	
 	\textit{Case 1:} Assume that $i\beta\in\set{x_1,x_2,x_3}$. Then, by part 2 of Lemma~\ref{T(X): lemma upper bound S when 3 copies of alpha}, we have that $x_1\beta=x_2\beta=x_3\beta=x$.  Hence $\beta\in A_1\cup A_2$.
 	
 	\smallskip
 	
 	\textit{Case 2:} Assume that $i\beta=i$. Then part 3 of Lemma~\ref{T(X): lemma upper bound S when 3 copies of alpha} implies that $x_1\beta=x_1$, $x_2\beta=x_2$ and $x_3\beta=x_3$. Consequently, $\beta\in A_4\cup A_5 \cup A_6 \cup A_7$.
 	
 	\smallskip
 	
 	\textit{Case 3:} Assume that $i\beta\in X\setminus\set{x_1,x_2,x_3,i}=\set{x_4,x_5}$. Let $j\in\set{x_4,x_5}$ be such that $i\beta=x_j$. It follows from part 4 of Lemma~\ref{T(X): lemma upper bound S when 3 copies of alpha} that $x_1\beta=x_2\beta=x_3\beta=\parens{i\beta}\alpha_1=x_j\alpha_1$. Hence $\beta\in A_3$.
 	
 	\smallskip
 	
 	It follows from the three cases above that $S\subseteq\bigcup_{j=1}^7 A_j$. Therefore
 	\begin{displaymath}
 		\abs{S}=\abs[\bigg]{S\cap\parens[\bigg]{\,\bigcup_{j=1}^7 A_j}}=\abs[\bigg]{\,\bigcup_{j=1}^7 \parens{S\cap A_j}}=\sum_{j=1}^7 \abs{S\cap A_j}.
 	\end{displaymath}
 	
 	\medskip
 	
 	\textbf{Part 2.} The aim of this part is to prove that if $S\cap A_2\neq\emptyset$, then $\abs{S\cap A_1}\leqslant 9$. Suppose that $S\cap A_2\neq\emptyset$. Let $\beta\in S\cap A_2$. We have that $x_4\beta\in\set{x_4,x_5}$ or $x_5\beta=\set{x_4,x_5}$. Interchanging $x_4$ and $x_5$ if necessary, assume that $x_4\beta\in\set{x_4,x_5}$. Let $j\in\set{4,5}$ be such that $x_4\beta=x_j$ and let $k\in\set{4,5}\setminus\set{j}$. For all $\gamma\in S\cap A_1$ we have that
 	\begin{align*}
 		x_j\gamma &=\parens{x_4\beta}\gamma &\bracks{\text{since } x_4\beta=x_j}\\
 		&=\parens{x_4\gamma}\beta &\bracks{\text{since }\beta,\gamma\in S, \text{ which is commutative}}\\
 		&=x. &\bracks{\text{since } x_4\gamma\in\set{x_1,x_2,x_3} \text{ and } x_1\beta=x_2\beta=x_3\beta=x}
 	\end{align*}
 	Since we also have $x_1\gamma=x_2\gamma=x_3\gamma=x$ and $x_k\gamma,i\gamma\in\set{x_1,x_2,x_3}$ for all $\gamma\in S\cap A_1$, then we can conclude that $\abs{S\cap A_1}\leqslant 3\cdot 3=9$.
 	
 	\medskip
 	
 	\textbf{Part 3.} We are going to see that if $S\cap A_3\neq\emptyset$, then $\abs{S\cap A_1}\leqslant 9$. Suppose that $S\cap A_3\neq\emptyset$. Let $\beta\in S\cap A_3$. Let $j\in\set{4,5}$ be such that $i\beta=x_j$ and let $k\in\set{4,5}\setminus\set{j}$. For all $\gamma\in S\cap  A_1$ we have that
 	\begin{align*}
 		x_j\gamma &=\parens{i\beta}\gamma &\bracks{\text{since } i\beta=x_j}\\
 		&=\parens{i\gamma}\beta &\bracks{\text{since }\beta,\gamma\in S, \text{ which is commutative}}\\
 		&=x_j\alpha_1. &\bracks{\text{since } i\gamma\in\set{x_1,x_2,x_3} \text{ and } x_1\beta=x_2\beta=x_3\beta=x_j\alpha_1}
 	\end{align*}
 	It follows from the fact that $x_1\gamma=x_2\gamma=x_3\gamma=x$ and $x_k\gamma,i\gamma\in\set{x_1,x_2,x_3}$ for all $\gamma\in S\cap A_1$ that $\abs{S\cap A_1}\leqslant 3\cdot 3=9$.
 	
 	\medskip
 	
 	\textbf{Part 4.} The objective of this part is to see that, if $S\cap A_4\neq\emptyset$, then $\abs{S\cap A_1}\leqslant 3$. Suppose that $S\cap A_4\neq\emptyset$. Let $\beta\in S\cap A_4$ and $\gamma\in S\cap A_1$. We have $x_1\gamma=x_2\gamma=x_3\gamma=x$. In addition, for all $j\in\set{4,5}$ we must have
 	\begin{align*}
 		x_j\gamma &=\parens{x_j\gamma}\beta &\bracks{\text{since } x_j\gamma\in\set{x_1,x_2,x_3} \text{ and } x_k\beta=x_k, \text{ } k\in\set{1,2,3}}\\
 		&=\parens{x_j\beta}\gamma &\bracks{\text{since } \beta,\gamma\in S, \text{ which is commutative}}\\
 		&= x. &\bracks{\text{since } x_j\beta\in\set{x_1,x_2,x_3} \text{ and } x_1\gamma=x_2\gamma=x_3\gamma=x}
 	\end{align*}
 	Due to the fact that $i\gamma\in\set{x_1,x_2,x_3}$, then we can conclude that $\abs{S\cap A_1}\leqslant 3$.
 	
 	\medskip
 	
 	\textbf{Part 5.} Now we are going to demonstrate that $\abs{S\cap A_2}\leqslant 12$. We begin by partitioning $A_2$ into four sets. Let
 	\begin{align*}
 		B_1&=\gset{\beta\in A_2}{x_4\beta\in\set{x_4,x_5} \text{ and } x_5\beta\in\set{x_1,x_2,x_3}};\\
 		B_2&=\gset{\beta\in A_2}{x_4\beta\in\set{x_1,x_2,x_3} \text{ and } x_5\beta\in\set{x_4,x_5}};\\
 		B_3&=\gset{\beta\in A_2}{x_4\beta,x_5\beta\in\set{x_4,x_5} \text{ and either } x_4\beta\neq x_4 \text{ or } x_5\beta\neq x_5};\\
 		B_4&=\gset{\beta\in A_2}{x_4\beta=x_4 \text{ and } x_5\beta=x_5}.
 	\end{align*}
 	Then it is clear that
 	\begin{displaymath}
 		\abs{S\cap A_2}=\abs[\bigg]{S\cap\parens[\bigg]{\,\bigcup_{j=1}^4 B_j}}=\abs[\bigg]{\,\bigcup_{j=1}^4 \parens{S\cap B_j}}=\sum_{j=1}^4 \abs{S\cap B_j}.
 	\end{displaymath}
 	
 	In order to determine an upper bound for $\abs{S\cap A_2}$ we consider several cases. We have that among the sets $S\cap B_1$, $S\cap B_2$ and $S\cap B_3$ there are at least two that are empty or there is at most one that is empty (that is, there are at least two that are not empty). The former situation is analysed in case 1, 2, 3 and the latter is analysed in cases 4, 5, 6.
 	
 	Before we start the case analysis, we are going to determine upper bounds for $\abs{S\cap B_j}$, for all $j\in\set{1,2,3,4}$. More specifically, we will see that $\abs{S\cap B_1}\leqslant 9$, $\abs{S\cap B_2}\leqslant 9$, $\abs{S\cap B_3}\leqslant 3$ and $\abs{S\cap B_4}\leqslant 3$. These four bounds are essential in cases 1--6 below.
 	
 	First, we are going to establish that $\abs{S\cap B_1}\leqslant 9$ (we can verify in a similar way that $\abs{S\cap B_2}\leqslant 9$). Let $\beta,\gamma\in B_1$ be such that $x_4\beta=x_4$ and $x_4\gamma=x_5$. We have that $x_4\beta\gamma=x_4\gamma=x_5$ and $x_4\gamma\beta=x_5\beta\in\set{x_1,x_2,x_3}$, which implies that $x_4\beta\gamma\neq x_4\gamma\beta$. Hence $\beta\gamma\neq\gamma\beta$ and, consequently, at most one of $\beta$ and $\gamma$ are in $S\cap B_1$ (since $S$ is commutative). This proves that we either have $x_4\beta=x_4$ for all $\beta\in S\cap B_1$, or $x_4\beta=x_5$ for all $\beta\in S\cap B_1$. Furthermore, for all $\beta\in S\cap B_1$ we have that $x_1\beta=x_2\beta=x_3\beta=x$ and $x_5\beta,i\beta\in\set{x_1,x_2,x_3}$. Therefore $\abs{S\cap B_1}\leqslant 3\cdot 3=9$.
 	
 	Now we are going to establish that $\abs{S\cap B_3}\leqslant 3$. For all $\beta\in B_3$ we have that $x_4\beta=x_5\beta=x_4$, or $x_4\beta=x_5\beta=x_5$, or $x_4\beta=x_5$ and $x_5\beta=x_4$. Let $\beta_1,\beta_2,\beta_3\in B_3$ be such that $x_4\beta_1=x_5\beta_1=x_4$, $x_4\beta_2=x_5\beta_2=x_5$, $x_4\beta_3=x_5$ and $x_5\beta_3=x_4$. We have that
 	\begin{gather*}
 		x_4\beta_1\beta_2=x_4\beta_2=x_5\neq x_4=x_5\beta_1=x_4\beta_2\beta_1\\
 		\shortintertext{and}
 		x_4\beta_1\beta_3=x_4\beta_3=x_5\neq x_4=x_5\beta_1=x_4\beta_3\beta_1\\
 		\shortintertext{and}
 		x_4\beta_2\beta_3=x_5\beta_3=x_4\neq x_5=x_5\beta_2=x_4\beta_3\beta_2,
 	\end{gather*}
 	which implies that among the transformations $\beta_1,\beta_2,\beta_3$ there is at most one that lies in $S\cap B_3$. This proves that all transformations of $S\cap B_3$ must be equal in $\set{x_4,x_5}$. Furthermore we have that $x_1\beta=x_2\beta=x_3\beta=x$ and $i\beta\in\set{x_1,x_2,x_3}$ for all $\beta\in S\cap B_3$, which allow us to conclude that $\abs{S\cap B_3}\leqslant 3$.
 	
 	Finally, we are going to establish that $\abs{S\cap B_4}\leqslant 3$. For all $\beta\in S\cap B_4$ we have that $x_1\beta=x_2\beta=x_3\beta=x$, $x_4\beta=x_4$, $x_5\beta=x_5$ and $i\beta\in\set{x_1,x_2,x_3}$. Thus it is clear that $\abs{S\cap B_4}\leqslant 3$.
 	
 	At last, we can start the case analysis.

 	
 	\smallskip
 	
 	\textit{Case 1:} Assume that $S\cap B_2=S\cap B_3=\emptyset$. We established earlier that $\abs{S\cap B_1}\leqslant 9$ and $\abs{S\cap B_4}\leqslant 3$. Then we have
 	\begin{displaymath}
 		\abs{S\cap A_2}=\abs{S\cap B_1}+\abs{S\cap B_4}\leqslant 9+3=12. 
 	\end{displaymath}
 	
 	\smallskip
 	
 	\textit{Case 2:} Assume that $S\cap B_1=S\cap B_3=\emptyset$. We established earlier that $\abs{S\cap B_2}\leqslant 9$ and $\abs{S\cap B_4}\leqslant 3$. Then we have
 	\begin{displaymath}
 		\abs{S\cap A_2}=\abs{S\cap B_2}+\abs{S\cap B_4}\leqslant 9+3=12. 
 	\end{displaymath}
 	
 	\smallskip
 	
 	\textit{Case 3:} Assume that $S\cap B_1=S\cap B_2=\emptyset$. We showed earlier that $\abs{S\cap B_3}\leqslant 3$ and $\abs{S\cap B_4}\leqslant 3$. Hence we have
 	\begin{displaymath}
 		\abs{S\cap A_2}=\abs{S\cap B_3}+\abs{S\cap B_4}\leqslant 3+3=6\leqslant 12. 
 	\end{displaymath}
 	
 	
 	\smallskip
 	
 	\textit{Case 4:} Assume that $S\cap B_1\neq\emptyset$ and $S\cap B_2\neq\emptyset$. Let $\beta\in S\cap B_1$ and $\gamma\in S\cap B_2$. We have that
 	\begin{align*}
 		\parens{x_4\beta}\gamma &=\parens{x_4\gamma}\beta &\bracks{\text{since } \beta,\gamma\in S, \text{ which is commutative}}\\
 		&=x &\bracks{\text{since } x_4\gamma\in\set{x_1,x_2,x_3} \text{ and } x_1\beta=x_2\beta=x_3\beta=x}\\
 		&=\parens{x_5\beta}\gamma &\bracks{\text{since } x_5\beta\in\set{x_1,x_2,x_3} \text{ and } x_1\gamma=x_2\gamma=x_3\gamma=x}\\
 		&=\parens{x_5\gamma}\beta. &\bracks{\text{since } \beta,\gamma\in S, \text{ which is commutative}}
 	\end{align*}
 	
 	Assume, with the aim of obtaining a contradiction, that $x\in\set{x_4,x_5}$. This implies that $\parens{x_4\beta}\gamma,\parens{x_5\gamma}\beta\in\set{x_4,x_5}$. Moreover, we have that $x_4\gamma, x_5\beta\in\set{x_1,x_2,x_3}$ and, consequently, we have that $x_4\beta\neq x_4$ and $x_5\gamma\neq x_5$.  Since $x_4\beta,x_5\gamma\in\set{x_4,x_5}$, then we can conclude that $x_4\beta=x_5$ and $x_5\gamma=x_4$. Hence $x_4=x_5\gamma=\parens{x_4\beta}\gamma=\parens{x_5\gamma}\beta=x_4\beta=x_5$, which is a contradiction. 
 	
 	Therefore $x\in X\setminus\set{x_4,x_5}$. In addition, we have that $x=x_1\alpha_1\in X\setminus\set{i}$ (because $\set{i}\in\setclass{S}{X}$), which implies that $x\in\set{x_1,x_2,x_3}$. Consequently, $\parens{x_4\beta}\gamma, \parens{x_5\gamma}\beta\in\set{x_1,x_2,x_3}$ and, since $x_5\gamma,x_4\beta\in\set{x_4,x_5}$, we must have $x_4\beta\neq x_5$ and $x_5\gamma\neq x_4$. Thus $x_4\beta=x_4$ and $x_5\gamma=x_5$ (because $x_4\beta,x_5\gamma\in\set{x_4,x_5}$). Consequently, we have $x_4\gamma=\parens{x_4\beta}\gamma=x=\parens{x_5\gamma}\beta=x_5\beta$ and, thus,
 	\begin{displaymath}
 		\beta=\begin{pmatrix}
 			x_1&x_2&x_3&x_4&x_5&i\\
 			x&x&x&x_4&x&x_k
 		\end{pmatrix} \quad \text{and} \quad
 		\gamma=\begin{pmatrix}
 			x_1&x_2&x_3&x_4&x_5&i\\
 			x&x&x&x&x_5&x_m
 		\end{pmatrix}
 	\end{displaymath}
 	for some $k,m\in\set{1,2,3}$, which implies that there are $3$ possibilities for $\beta$ and $3$ possibilities for $\gamma$. Therefore $\abs{S\cap B_1}\leqslant 3$ and $\abs{S\cap B_2}\leqslant 3$. Moreover, we established before that $\abs{S\cap B_3}\leqslant 3$ and $\abs{S\cap B_4}\leqslant 3$. Thus
 	\begin{displaymath}
 		\abs{S\cap A_2}=\abs{S\cap B_1}+\abs{S\cap B_2}+\abs{S\cap B_3}+\abs{S\cap B_4}\leqslant 3+3+3+3=12.
 	\end{displaymath}
 	
 	\smallskip
 	
 	\textit{Case 5:} Assume that $S\cap B_1\neq\emptyset$ and $S\cap B_3\neq\emptyset$. Let $\beta\in S\cap B_1$ and $\gamma\in S\cap B_3$. We have that
 	\begin{align*}
 		\parens{x_4\gamma}\beta &=\parens{x_4\beta}\gamma &\bracks{\text{since } \beta,\gamma\in S, \text{ which is commutative}}\\
 		&\in\set{x_4,x_5}, &\bracks{\text{since } x_4\beta\in\set{x_4,x_5} \text{ and } x_4\gamma,x_5\gamma\in\set{x_4,x_5}}
 	\end{align*}
 	which implies that $x_4\gamma\neq x_5$ (because $x_5\beta\in\set{x_1,x_2,x_3}$). Hence $x_4\gamma=x_4$ (because $x_4\gamma\in\set{x_4,x_5}$) and, consequently, $x_5\gamma\neq x_5$ (because we must have $x_4\gamma\neq x_4$ or $x_5\gamma\neq x_5$), which implies that $x_5\gamma=x_4$ (because $x_5\gamma\in\set{x_4,x_5}$). Moreover, we have
 	\begin{align*}
 		x &=\parens{x_5\beta}\gamma &\bracks{\text{since } x_5\beta\in\set{x_1,x_2,x_3} \text{ and } x_1\gamma=x_2\gamma=x_3\gamma=x}\\
 		&=\parens{x_5\gamma}\beta &\bracks{\text{since } \beta,\gamma\in S, \text{ which is commutative}}\\
 		&=x_4\beta &\bracks{\text{since } x_5\gamma=x_4}\\
 		&\in\set{x_4,x_5}.
 	\end{align*}
 	Thus for all $\alpha\in B_4$ we have
 	\begin{align*}
 		x_5\alpha\beta &=x_5\beta &\bracks{\text{since } x_5\alpha=x_5}\\
 		&\neq x &\bracks{\text{since } x_5\beta\in\set{x_1,x_2,x_3} \text{ and } x\in\set{x_4,x_5}}\\
 		&=x_5\beta\alpha &\bracks{\text{since } x_5\beta\in\set{x_1,x_2,x_3} \text{ and } x_1\alpha=x_2\alpha=x_3\alpha=x}
 	\end{align*}
 	and, consequently, we can conclude that there is no $\alpha$ in $B_4$ that commutes with $\beta$. Since $\beta\in S$ and $S$ is commutative, then this implies that $S\cap B_4=\emptyset$.
 	
 	If $S\cap B_2\neq\emptyset$, then, by case 4, we have that $\abs{S\cap A_2}\leqslant 12$. If $S\cap B_2=\emptyset$, then
 	\begin{displaymath}
 		\abs{S\cap A_2}=\abs{S\cap B_1}+\abs{S\cap B_3}\leqslant 9+3=12.
 	\end{displaymath}
 	(We recall that we established earlier that $\abs{S\cap B_1}\leqslant 9$ and $\abs{S\cap B_3}\leqslant 3$.)
 	
 	\smallskip
 	
 	\textit{Case 6:} Assume that $S\cap B_2\neq\emptyset$ and $S\cap B_3\neq\emptyset$. We can prove in a similar way to case 5 that $\abs{S\cap A_2}\leqslant 12$.
 	
 	\medskip
 	
 	\textbf{Part 6.} The aim of this part is to show that $\abs{S\cap A_3}\leqslant 5$. We begin by partitioning $A_3$. Let
 	\begin{align*}
 		B_1&=\gset{\beta\in A_3}{i\beta=x_4};\\
 		B_2&=\gset{\beta\in A_3}{i\beta=x_5}.
 	\end{align*}
 	We divide the proof into three cases.
 	
 	\smallskip
 	
 	\textit{Case 1:} Assume that $S\cap B_1=\emptyset$. Then $S\cap A_3=S\cap B_2$. For all $\beta,\gamma\in B_2$ such that $x_5\beta\neq x_5\gamma$ we have that $i\beta\gamma=x_5\gamma\neq x_5\beta=i\gamma\beta$. Since $S$ is commutative, then this implies that all transformations of $S\cap B_2$ must be equal in $\set{x_5}$. Furthermore, all transformations of $S\cap B_2$ are equal in $\set{x_1,x_2,x_3,i}$ and we have $x_4\beta\in\set{x_1,x_2,x_3,x_4,x_5}$ for all $\beta\in S\cap B_2$ and, consequently, $S\cap B_2$ contains at most $5$ transformations; that is, $\abs{S\cap A_3}=\abs{S\cap B_2}\leqslant 5$.
 	
 	\smallskip
 	
 	\textit{Case 2:} Assume that $S\cap B_2=\emptyset$. We can prove, as in case 1, that $\abs{S\cap A_3}\leqslant 5$.

 	\smallskip
 	
 	\textit{Case 3:} Assume that $S\cap B_1\neq\emptyset$ and $S\cap B_2\neq\emptyset$. Let $\beta\in S\cap B_1$ and $\gamma\in S\cap B_2$. Let $\alpha\in S\cap A_3$. If $\alpha\in B_1$, then $i\alpha=x_4=i\beta$, $x_1\alpha=x_2\alpha=x_3\alpha=x_4\alpha_1=x_1\beta=x_2\beta=x_3\beta$ and
 	\begin{align*}
 		x_4\alpha &=\parens{i\beta}\alpha &\bracks{\text{since } i\beta=x_4}\\
 		&=\parens{i\alpha}\beta &\bracks{\text{since } \beta,\alpha\in S, \text{ which is commutative}}\\
 		&=x_4\beta &\bracks{\text{since } i\alpha=x_4}
 	\end{align*}
 	and
 	\begin{align*}
 		x_5\alpha &=\parens{i\gamma}\alpha &\bracks{\text{since } i\gamma=x_5}\\
 		&=\parens{i\alpha}\gamma &\bracks{\text{since } \gamma,\alpha\in S, \text{ which is commutative}}\\
 		&=\parens{i\beta}\gamma &\bracks{\text{since } i\alpha=x_4=i\beta}\\
 		&=\parens{i\gamma}\beta &\bracks{\text{since } \beta,\gamma\in S, \text{ which is commutative}}\\
 		&=x_5\beta, &\bracks{\text{since } i\gamma=x_5}
 	\end{align*}
 	which implies that $\alpha=\beta$. If $\alpha\in B_2$, then we can prove in a similar way that $\alpha=\gamma$. Thus $\abs{S\cap A_3}\leqslant 2\leqslant 5$.

 	\medskip
 	
 	\textbf{Part 7.} We are going to establish that $\abs{S\cap A_4}\leqslant 1$. Let $\beta,\gamma\in S\cap A_4$. We want to prove that $\beta=\gamma$. We have $y\beta=y=y\gamma$ for all $y\in\set{x_1,x_2,x_3,i}=X\setminus\set{x_4,x_5}$. Furthermore, for all $j\in\set{4,5}$ we have
 	\begin{align*}
 		x_j\beta &=\parens{x_j\beta}\gamma &\bracks{\text{since } x_j\beta\in\set{x_1,x_2,x_3}}\\
 		&=\parens{x_j\gamma}\beta &\bracks{\text{since } \beta,\gamma\in S, \text{ which is commutative}}\\
 		&=x_j\gamma. &\bracks{\text{since } x_j\gamma\in\set{x_1,x_2,x_3}}
 	\end{align*}
 	Thus $\beta=\gamma$ and, consequently, we can conclude that $\abs{S\cap A_4}\leqslant 1$.
 	
 	\medskip
 	
 	\textbf{Part 8.} Now we prove that $\abs{S\cap A_5}\leqslant 1$. Let $\beta,\gamma\in S\cap A_5$. We want to prove that $\beta=\gamma$. We have that $y\beta=y=y\gamma$ for all $y\in\set{x_1,x_2,x_3,i}=X\setminus\set{x_4,x_5}$. Moreover,
 	\begin{align*}
 		x_4\beta &=\parens{x_4\beta}\gamma &\bracks{\text{since } x_4\beta\in\set{x_1,x_2,x_3}}\\
 		&=\parens{x_4\gamma}\beta &\bracks{\text{since } \beta,\gamma\in S, \text{ which is commutative}}\\
 		&=x_4\gamma. &\bracks{\text{since } x_4\gamma\in\set{x_1,x_2,x_3}}
 	\end{align*}
 	Finally, we are going to verify that $x_5\beta=x_5\gamma$. We have that $x_5\beta,x_5\gamma\in\set{x_4,x_5}$. If $x_5\gamma=x_4$, then we have $\parens{x_5\beta}\gamma=\parens{x_5\gamma}\beta=x_4\beta\in\set{x_1,x_2,x_3}$, which implies that $x_5\beta\in\set{x_1,x_2,x_3,x_4}\cap\set{x_4,x_5}$ and, consequently, $x_5\beta=x_4=x_5\gamma$. If $x_5\gamma=x_5$, then we have $\parens{x_5\beta}\gamma=\parens{x_5\gamma}\beta=x_5\beta\in\set{x_4,x_5}$, which implies that $x_5\beta=x_5=x_5\gamma$. Therefore $\beta=\gamma$ and, consequently, we must have $\abs{S\cap A_5}\leqslant 1$.
 	
 	\medskip
 	
 	\textbf{Part 9.} Proving that $\abs{S\cap A_6}\leqslant 1$ is analogous to proving that $\abs{S\cap A_5}\leqslant 1$ (which was established in part 8).
 	
 	\medskip
 	
 	\textbf{Part 10.} We are going to see that $\abs{S\cap A_7}\leqslant 2$. Let $\beta_1,\beta_2,\beta_3\in A_7$ be such that $x_4\beta_1=x_5\beta_1=x_4$, $x_4\beta_2=x_5\beta_2=x_5$, $x_4\beta_3=x_5$ and $x_5\beta_3=x_4$. Then $A_7=\set{\id{X},\beta_1,\beta_2,\beta_3}$. We have that $\id{X}$ commutes with $\beta_1,\beta_2,\beta_3$ and we have that $\beta_1,\beta_2,\beta_3$ do not commute with each other (because
 	\begin{gather*}
 		x_4\beta_1\beta_2=x_4\beta_2=x_5\neq x_4=x_5\beta_1=x_4\beta_2\beta_1\\
 		\shortintertext{and}
 		x_4\beta_1\beta_3=x_4\beta_3=x_5\neq x_4=x_5\beta_1=x_4\beta_3\beta_1\\
 		\shortintertext{and}
 		x_4\beta_2\beta_3=x_5\beta_3=x_4\neq x_5=x_5\beta_2=x_4\beta_3\beta_2). 
 	\end{gather*}
 	Therefore we must have $\abs{S\cap A_7}\leqslant 2$.
 	
 	\medskip
 	
 	\textbf{Part 11.} At last, we can proceed with demonstrating that $\abs{S}<2^{\abs{X}-1}$. We divide this proof into four cases, which we present below.
 	
 	\smallskip
 	
 	\textit{Case 1:} Suppose that $S\cap A_2\neq\emptyset$. Then
 	\begin{align*}
 		\abs{S} &=\sum_{j=1}^7 \abs{S\cap A_j} &\bracks{\text{by part 1}}\\
 		&\leqslant 9+12+5+1+1+1+2 &\bracks{\text{by parts 2, 5, 6, 7, 8, 9, 10}}\\
 		&=31\\
 		&<2^{\abs{X}-1}. &\bracks{\text{since } \abs{X}=6}
 	\end{align*}
 	
 	\smallskip
 	
 	\textit{Case 2:} Suppose that $S\cap A_2=\emptyset$ and $S\cap A_3\neq\emptyset$. Then
 	\begin{align*}
 		\abs{S} &=\sum_{j=1}^7 \abs{S\cap A_j} &\bracks{\text{by part 1}}\\
 		&=\abs{S\cap A_1}+\sum_{j=3}^7 \abs{S\cap A_j} &\bracks{\text{since } S\cap A_2=\emptyset}\\
 		&\leqslant 9+5+1+1+1+2 &\bracks{\text{by parts 3, 6, 7, 8, 9, 10}}\\
 		&=19\\
 		&<2^{\abs{X}-1}. &\bracks{\text{since } \abs{X}=6}
 	\end{align*}
 	
 	\smallskip
 	
 	\textit{Case 3:} Suppose that $S\cap A_2=S\cap A_3=\emptyset$ and $S\cap A_4\neq\emptyset$. Then
 	\begin{align*}
 		\abs{S} &=\sum_{j=1}^7 \abs{S\cap A_j} &\bracks{\text{by part 1}}\\
 		&=\abs{S\cap A_1}+\sum_{j=4}^7 \abs{S\cap A_j} &\bracks{\text{since } S\cap A_2=S\cap A_3=\emptyset}\\
 		&\leqslant 3+1+1+1+2 &\bracks{\text{by parts 4, 7, 8, 9, 10}}\\
 		&=8\\
 		&<2^{\abs{X}-1}. &\bracks{\text{since } \abs{X}=6}
 	\end{align*}
 	
 	\smallskip
 	
 	\textit{Case 4:} Suppose that $S\cap A_2=S\cap A_3=S\cap A_4=\emptyset$. We observe that $\abs{S\cap A_1}\leqslant \abs{A_1}\leqslant 3\cdot 3\cdot 3=27$. Then
 	\begin{align*}
 		\abs{S} &=\sum_{j=1}^7 \abs{S\cap A_j} &\bracks{\text{by part 1}}\\
 		&=\abs{S\cap A_1}+\sum_{j=5}^7 \abs{S\cap A_j} &\bracks{\text{since } S\cap A_2=S\cap A_3=S\cap A_4=\emptyset}\\
 		&\leqslant 27+1+1+2 &\bracks{\text{by parts 8, 9, 10}}\\
 		&=31\\
 		&<2^{\abs{X}-1}. &\bracks{\text{since } \abs{X}=6} &\qedhere
 	\end{align*}
 \end{proof}
 
 At last we characterize the maximum-order commutative subsemigroups of $\tr{X}$ when $\abs{X}\leqslant 6$. Moreover, we give a lower bound for the maximum size of a commutative subsemigroups of $\tr{X}$ when $\abs{X}\geqslant 7$.
 
 \begin{theorem}\label{T(X): maximum size comm smg}
 	\begin{enumerate}
 		\item Suppose that $\abs{X}\leqslant 6$. Then the maximum size of a commutative subsemigroup of $\tr{X}$ is $2^{\abs{X}-1}$. Moreover,
 		\begin{enumerate}
 			\item If $\abs{X}\neq 2$, then the maximum-order commutative subsemigroups of $\tr{X}$ are precisely the semigroups of idempotents $\commidemp{X}{x}$, where $x\in X$.
 			
 			\item If $\abs{X}=2$, then the maximum-order commutative subsemigroups of $\tr{X}$ are the semigroups of idempotents $\commidemp{X}{x}$, where $x\in X$, and the subgroup of $\sym{X}$ isomorphic to $C_2$.
 		\end{enumerate}
 		
 		\item Suppose that $\abs{X}\geqslant 7$. Then the maximum size of a commutative subsemigroup of $\tr{X}$ is at least $\maxnull{\abs{X}}+1$.
 	\end{enumerate}
 \end{theorem}
 
 For the definition of the semigroups $\commidemp{X}{x}$, where $x\in X$, see \eqref{T(X): Gamma^X_i definition}.
 
 \begin{proof}
 	\textbf{Part 1.} Suppose that $\abs{X}\leqslant 6$. Let $\mathcal{C}^X$ be the class formed by the commutative subsemigroups of $\tr{X}$ that are not contained in $\sym{X}$. 
 	
 	Suppose that $\abs{X}=1$, then $S=\set{\id{X}}=\commidemp{X}{\id{X}}=\tr{X}$ and $\abs{S}=1=2^{\abs{X}-1}$.
 	
 	Suppose that $\abs{X}=2$ and $X=\set{x_1,x_2}$. Then we have
 	\begin{displaymath}
 		\tr{X}=\set[\bigg]{\begin{pmatrix}
 				x_1 & x_2\\
 				x_1 & x_1
 			\end{pmatrix}, \begin{pmatrix}
 				x_1 & x_2\\
 				x_2 & x_2
 			\end{pmatrix}, \begin{pmatrix}
 				x_1 & x_2\\
 				x_2 & x_1
 			\end{pmatrix}, \id{X}}
 	\end{displaymath}
 	and it is easy to see that there are no distinct transformations in $\tr{X}\setminus\set{\id{X}}$ that commute. Consequently, $\abs{S}\leqslant 2=2^{\abs{X}-1}$ and the largest commutative subsemigroups of $\tr{X}$ are
 	\begin{displaymath}
 		\commidemp{X}{x_1}=\set[\bigg]{\begin{pmatrix}
 				x_1 & x_2\\
 				x_1 & x_1
 			\end{pmatrix}, \id{X}} \; \text{ and } \;
 		\commidemp{X}{x_2}=\set[\bigg]{\begin{pmatrix}
 				x_1 & x_2\\
 				x_2 & x_2
 			\end{pmatrix}, \id{X}} \; \text{ and } \;
 		C_2\simeq\set[\bigg]{\begin{pmatrix}
 				x_1 & x_2\\
 				x_2 & x_1
 			\end{pmatrix}, \id{X}}.
 	\end{displaymath}
 	
 	Now suppose that $3\leqslant \abs{X}\leqslant 6$. It follows from
 	Proposition~\ref{T(X): Gamma^X_i comm smg idemp} that there are commutative subsemigroups of $\tr{X}$ of size $2^{\abs{X}-1}$. Moreover, in Lemma~\ref{T(X): upper bound comm subsmg S(X)} we saw that any commutative subsemigroup of $\tr{X}$ contained in $\sym{X}$ has size at most $2^{\abs{X}-1}-1$. Hence the maximum-order commutative subsemigroups of $\tr{X}$ are not contained in $\sym{X}$ and, consequently, finding the maximum-order commutative subsemigroups of $\tr{X}$ is equivalent to finding the maximum-order semigroups in $\mathcal{C}^X$, which is what we will do below.
 	
 	
 	
 	We are going to prove that, when $2\leqslant \abs{X}\leqslant 6$, the maximum-order semigroups in $\mathcal{C}^X$ are precisely the semigroups $\commidemp{X}{x}$, where $x\in X$, which have size $2^{\abs{X}-1}$ (see Proposition~\ref{T(X): Gamma^X_i comm smg idemp}). We will prove this result by induction on the size of $\abs{X}$.
 	
 	Assume that $\abs{X}=2$. We saw earlier that the largest commutative subsemigroups of $\tr{X}$ that are not contained in $\sym{X}$ are precisely the semigroups of idempotents $\commidemp{X}{x}$, where $x\in X$; that is, the largest semigroups in $\mathcal{C}^X$ are precisely the semigroups $\commidemp{X}{x}$, where $x\in X$.
 	
 	Now assume that $3\leqslant\abs{X}\leqslant 6$ and assume that, for all set $Y$ such that $2\leqslant\abs{Y}<\abs{X}$, the maximum-order semigroups in $\mathcal{C}^Y$ are the semigroups $\commidemp{Y}{x}$ of size $2^{\abs{Y}-1}$, where $x\in Y$.
 	
 	Let $S$ be a maximum-order semigroup in $\mathcal{C}^X$. We have that $S\nsubseteq\sym{X}$ and, consequently, Lemma~\ref{T(X): existence of I, b|_(X/I)} ensures that $\setclass{S}{X}\neq\emptyset$. Let $I\in\setclass{S}{X}$ be of minimum size and let $S'=\gset{\beta|_{X\setminus I}}{\beta\in S}$. We have that $I$ is a non-empty proper subset of $X$ and that $\beta|_{X\setminus I}\in\tr{X\setminus I}$ for all $\beta\in S$. Furthermore, $S$ is a commutative subsemigroup of $\tr{X}$. Hence Lemma~\ref{T(X): lemma induction} implies that $S'$ is a commutative subsemigroup of $\tr{X\setminus I}$.
 	
 	Before continuing with the proof of Theorem~\ref{T(X): maximum size comm smg}, we establish the following lemma, which states that $2^{\abs{X}-\abs{I}-1}$ is an upper bound for the size of $\abs{S'}$, and that $2^{\abs{X}-1}$ is a lower bound for the size of $\abs{S}$.
 	
 	\begin{lemma}\label{T(X): largest cmm smg, lemma |S'|<=}
 		We have that $\abs{S'}\leqslant 2^{\abs{X}-\abs{I}-1}$ and $\abs{S}\geqslant 2^{\abs{X}-1}$.
 	\end{lemma}
 	
 	\begin{proof}
 		
 		First, we will show that $\abs{S'}\leqslant 2^{\abs{X}-\abs{I}-1}$.
 		
 		\smallskip
 		
 		\textit{Case 1:} Assume that $S'\subseteq\sym{X\setminus I}$. If $\abs{X\setminus I}\in\set{1,2}$, then it follows from what we proved earlier that $\abs{S'}\leqslant 2^{\abs{X\setminus I}-1}=2^{\abs{X}-\abs{I}-1}$. If $\abs{X\setminus I}\geqslant 3$, then Lemma~\ref{T(X): upper bound comm subsmg S(X)} implies that $\abs{S'}\leqslant 2^{\abs{X\setminus I}-1}=2^{\abs{X}-\abs{I}-1}$.
 		
 		\smallskip
 		
 		\textit{Case 2:} Assume that $S'\nsubseteq\sym{X\setminus I}$. Then $S'\in \mathcal{C}^{X\setminus I}$ and $\abs{X\setminus I}\geqslant 2$ (we note that $\mathcal{C}^Y$ is empty if $\abs{Y}=1$). By the induction hypothesis we have that $\abs{S'}\leqslant 2^{\abs{X\setminus I}-1}=2^{\abs{X}-\abs{I}-1}$.
 		
 		\smallskip
 		
 		In both cases we established that $\abs{S'}\leqslant 2^{\abs{X}-\abs{I}-1}$.

 		Now we will show that $\abs{S}\geqslant 2^{\abs{X}-1}$. Let $x\in X$. It follows from Proposition~\ref{T(X): Gamma^X_i comm smg idemp} that $\commidemp{X}{x}$ is a commutative subsemigroup of $\tr{X}$ of size $2^{\abs{X}-1}$. Additionally, $\commidemp{X}{x}\nsubseteq\sym{X}$ because $\commidemp{X}{x}$ contains a transformation of rank $1$. Hence $\commidemp{X}{j}\in\mathcal{C}^X$ and, consequently, the maximality of the size of $S$ implies that $\abs{S}\geqslant\abs{\commidemp{X}{j}}= 2^{\abs{X}-1}$.
 	\end{proof}
 	
 	Our next goal is to ascertain that $\abs{I}=1$, which is proved in the next lemma.
 	
 	\begin{lemma}\label{T(X): largest cmm smg, lemma |I|=1}
 		We have that $\abs{I}=1$.
 	\end{lemma}

 	\begin{proof}
 		Assume, with the aim of obtaining a contradiction, that $\abs{I}>1$. We have that $\abs{X}\leqslant 6$ and $\abs{S'}\leqslant 2^{\abs{X}-\abs{I}-1}$ and $\abs{S}\geqslant 2^{\abs{X}-1}$ (by Lemma~\ref{T(X): largest cmm smg, lemma |S'|<=}). Hence, by Lemma~\ref{T(X): upper bound S, |I|>2}, we must have $\abs{I}<3$ and, consequently, $\abs{I}=2$. Moreover, the fact that $\abs{S'}\leqslant 2^{\abs{X}-\abs{I}-1}=2^{\abs{X}-3}$ and $\abs{S}\geqslant 2^{\abs{X}-1}$, together with part 2 of Lemma~\ref{T(X): upper bound S, |I|=2}, implies that $\abs{X}>5$. Therefore $\abs{X}=6$. We are going to see that $\abs{S'}<2^{\abs{X}-3}$.
 		
 		\smallskip
 		
 		\textit{Case 1:} Assume that $S'\subseteq\sym{X\setminus I}$. Since $\abs{X}=6$ and $\abs{I}=2$, then we have that $\abs{X\setminus I}=4\geqslant 3$. Hence, by Lemma~\ref{T(X): upper bound comm subsmg S(X)}, we have $\abs{S'}<2^{\abs{X\setminus I}-1}=2^{\abs{X}-3}$.
 		
 		\smallskip
 		
 		\textit{Case 2:} Assume that $S'\nsubseteq\sym{X\setminus I}$. Then $S'\in\mathcal{C}^{X\setminus I}$. Due to the fact that $\abs{X\setminus I}=4\geqslant 2$, we can  use the induction hypothesis to conclude that the semigroups $\commidemp{X\setminus I}{x}$, where $x\in X\setminus I$, are precisely the largest semigroups in $\mathcal{C}^{X\setminus I}$ (which have size $2^{\abs{X\setminus I}-1}=2^{\abs{X}-3}$). In addition, by part 3 of Lemma~\ref{T(X): upper bound S, |I|=2}, we have that $S'\neq\commidemp{X\setminus I}{x}$ for all $x\in X\setminus I$. Consequently, $S'$ is not any of the commutative semigroups in $\mathcal{C}^{X\setminus I}$ of size $2^{\abs{X}-3}$; that is, $\abs{S'}<2^{\abs{X}-3}$.
 		
 		\smallskip
 		
 		In both cases we established that $\abs{S'}<2^{\abs{X}-3}$. Hence part 2 of Lemma~\ref{T(X): upper bound S, |I|=2} implies that $\abs{S}<2^{\abs{X}-1}$, which is a contradiction (because $\abs{S}\geqslant 2^{\abs{X}-1}$). Thus $\abs{I}=1$.
 	\end{proof}
 	
 	
 	By the previous lemma we have that $I=\set{i}$ for some $i\in X$.
 	
 	For each $\gamma\in S'$ we define $S_\gamma=\gset{\beta\in S}{\beta|_{X\setminus \set{i}}=\gamma}$. It is clear that $\set{S_\gamma}_{\gamma\in S'}$ is a partition of $S$. Let $\gamma'\in S'$ be such that $\abs{S_{\gamma'}}=\max\gset{\abs{S_{\gamma}}}{\gamma\in S'}$. In the next lemma we will see that for all $\gamma\in S'$ we have $\abs{S_{\gamma}}\leqslant 2$.
 	
 	\begin{lemma}\label{T(X): largest cmm smg, lemma |S gamma|<=2}
 		For each $\gamma\in S'$ we have that $\abs{S_{\gamma}}\leqslant 2$.
 	\end{lemma}
 	
 	\begin{proof}
 		\textit{Case 1:} Assume that $\abs{X}=3$. Assume, with the aim of obtaining a contradiction, that $\abs{S_{\gamma'}}\geqslant 3$. Let $\beta_1,\beta_2,\beta_3\in S_{\gamma'}$ be pairwise distinct. We have that $\beta_1|_{X\setminus\set{i}}=\beta_2|_{X\setminus\set{i}}=\beta_3|_{X\setminus\set{i}}=\gamma'$, which implies, by part 1 of Lemma~\ref{T(X): lemma upper bound S when 3 copies of alpha}, that $i\beta_1,i\beta_2,i\beta_3\in X\setminus\set{i}$ and are pairwise distinct. Since $\abs{X\setminus\set{i}}=2$, then we have reached a contradiction. Therefore $\abs{S_{\gamma'}}\leqslant 2$.
 		
 		\smallskip
 		
 		\textit{Case 2:} Assume that $\abs{X}\in\set{4,5,6}$. It follows from the fact that $\abs{S}\geqslant 2^{\abs{X}-1}$ and Lemmata~\ref{T(X): |X|=4, upper bound S when 3 copies of alpha}, \ref{T(X): |X|=5, upper bound S when 3 copies of alpha} and \ref{T(X): |X|=6, upper bound S when 3 copies of alpha} that there are no pairwise distinct $\beta_1,\beta_2,\beta_3\in S$ such that $\beta_1|_{X\setminus\set{i}}=\beta_2|_{X\setminus\set{i}}=\beta_3|_{X\setminus\set{i}}$. Hence there are no pairwise distinct $\beta_1,\beta_2,\beta_3\in S$ such that $\beta_1|_{X\setminus\set{i}}=\beta_2|_{X\setminus\set{i}}=\beta_3|_{X\setminus\set{i}}=\gamma'$; that is, such that $\beta_1,\beta_2,\beta_3\in S_{\gamma'}$. Therefore $\abs{S'_{\gamma'}}\leqslant 2$.
 		
 		\smallskip
 		
 		In the previous two cases we showed that $\abs{S_{\gamma'}}\leqslant 2$. Then for all $\gamma\in S'$ we have $\abs{S_\gamma}\leqslant\abs{S_{\gamma'}}\leqslant 2$. 
 	\end{proof}
 	
 	We can now resume the proof of Lemma~\ref{T(X): maximum size comm smg}. By Lemmata~\ref{T(X): largest cmm smg, lemma |S'|<=} and \ref{T(X): largest cmm smg, lemma |S gamma|<=2} we have
 	\begin{displaymath}
 		2^{\abs{X}-1}\leqslant \abs{S}=\sum_{\gamma\in S'} \abs{S_\gamma}\leqslant \sum_{\gamma\in S'} 2=\abs{S'}\cdot 2\leqslant 2^{\abs{X}-\abs{\set{i}}-1}\cdot 2=2^{\abs{X}-1}.
 	\end{displaymath}
 	Hence $\abs{S}=2^{\abs{X}-1}$ and $\abs{S'}=2^{\abs{X}-2}=2^{\abs{X\setminus\set{i}}-1}$ and $\sum_{\gamma\in S'} \abs{S_\gamma}=2^{\abs{X}-1}$. Since we also have $\abs{S_\gamma}\leqslant 2$ for all $\gamma\in S$, then we must have $\abs{S_\gamma}=2$ for all $\gamma\in S$.

 	We have that $S'$ is a commutative subsemigroup of $\tr{X\setminus\set{i}}$ of size $2^{\abs{X\setminus\set{i}}-1}$. Assume, with the aim of obtaining a contradiction, that $S'\notin\mathcal{C}^{X\setminus\set{i}}$ (that is, $S'\subseteq\sym{X\setminus\set{i}}$). We have that $3\leqslant\abs{X}\leqslant 6$, which implies that $2\leqslant\abs{X\setminus\set{i}}\leqslant 5$. Furthermore, it follows from Lemma~\ref{T(X): upper bound comm subsmg S(X)}, and the fact that $\abs{S'}=2^{\abs{X\setminus\set{i}}-1}$, that $\abs{X\setminus\set{i}}\leqslant 2$. Hence $\abs{X\setminus\set{i}}=2$ and $\abs{X}=3$. Let $x_1,x_2\in X$ be such that $X=\set{x_1,x_2,i}$. Since $S'\subseteq\sym{\set{x_1,x_2}}$ and $\abs{S'}=2^{2-1}=2$, then we have $S'=\set{\gamma,\id{X\setminus\set{i}}}$, where $\gamma=\begin{pmatrix}
 		x_1&x_2\\
 		x_2&x_1
 	\end{pmatrix}$. Let $\alpha_1,\alpha_2\in S$ be such that $S_{\gamma}=\set{\alpha_1,\alpha_2}$ (we recall that we proved earlier that $\abs{S_\gamma}=2$). We have $\alpha_1|_{X\setminus\set{i}}=\gamma=\alpha_2|_{X\setminus\set{i}}$ and, consequently, $i\alpha_1\neq i\alpha_2$. Hence $i\alpha_1\in\set{x_1,x_2}$ or $i\alpha_2\in\set{x_1,x_2}$. We can assume, without loss of generality, that $i\alpha_1=x_1$. Then $i\alpha_2\in\set{x_2,i}$ and 
 	\begin{align*}
 		x_2 &=x_1\gamma\\
 		&=x_1\alpha_2 &\bracks{\text{since } x_1\in X\setminus\set{i} \text{ and } \alpha_2|_{X\setminus\set{i}}=\gamma}\\
 		&=\parens{i\alpha_1}\alpha_2 &\bracks{\text{since } i\alpha_1=x_1}\\
 		&= \parens{i\alpha_2}\alpha_1 &\bracks{\text{since } \alpha_1,\alpha_2\in S, \text{ which is commutative}}\\
 		&=\begin{cases}
 			x_2\alpha_1 &\text{if } i\alpha_2=x_2,\\
 			i\alpha_1 &\text{if } i\alpha_2=i
 		\end{cases} \\
 		&=\begin{cases}
 			x_2\gamma &\text{if } i\alpha_2=x_2,\\
 			i\alpha_1 &\text{if } i\alpha_2=i
 		\end{cases} &\bracks{\text{since } x_2\in X\setminus\set{i} \text{ and } \alpha_1|_{X\setminus\set{i}}=\gamma}\\
 		&=x_1,
 	\end{align*}
 	which is a contradiction.
 	
 	Thus $S'\in\mathcal{C}^{X\setminus\set{i}}$ and, by the induction hypothesis, we have $S'=\commidemp{X\setminus\set{i}}{x}$ for some $x\in X\setminus\set{i}$. We are going to see that $S=\commidemp{X}{x}$. We note that it is enough to establish that $S\subseteq\commidemp{X}{x}$ because $\abs{S}=2^{\abs{X}-1}=\abs{\commidemp{X}{x}}$.
 	
 	It follows from the fact that $S'=\commidemp{X\setminus\set{i}}{x}$ that $\id{X\setminus\set{i}}\in S'$ and that there exists $\gamma\in S'$ such that $\im\gamma=\set{x}$. Moreover, $\abs{S_{\id{X\setminus\set{i}}}}=\abs{S_\gamma}=2$ and, consequently, there exist pairwise distinct $\beta_1,\beta_2,\gamma_1,\gamma_2\in S$ such that $S_{\id{X\setminus\set{i}}}=\set{\beta_1,\beta_2}$ and $S_\gamma=\set{\gamma_1,\gamma_2}$. We have that $\beta_1|_{X\setminus\set{i}}=\id{X\setminus\set{i}}=\beta_2|_{X\setminus\set{i}}$ and $\gamma_1|_{X\setminus\set{i}}=\gamma=\gamma_2|_{X\setminus\set{i}}$. The former implies that $i\beta_1\neq i\beta_2$. Hence $i\beta_1\neq i$ or $i\beta_2\neq i$. Assume, without loss of generality, that $i\beta_1\neq i$. We have that
 	\begin{align*}
 		\parens{i\beta_2}\beta_1 &=\parens{i\beta_1}\beta_2 &\bracks{\text{since } \beta_2,\beta_1\in S, \text{ which is commutative}}\\
 		&=\parens{i\beta_1}\id{X\setminus\set{i}} &\bracks{\text{since } i\beta_1\in X\setminus\set{i} \text{ and } \beta_2|_{X\setminus\set{i}}=\id{X\setminus\set{i}}}\\
 		&=i\beta_1.
 	\end{align*}
 	Then, since $i\beta_2\neq i\beta_1$, we must have $i\beta_2\neq \parens{i\beta_2}\beta_1$ and, since $y\beta_1=y\mskip 1.5mu \id{X\setminus\set{i}}=y$ for all $y\in X\setminus\set{i}$, then we must have $i\beta_2=i$. As a consequence of the fact that $\beta_1\neq\beta_2$ and $\beta_1|_{X\setminus\set{i}}=\beta_2|_{X\setminus\set{i}}$, and by Lemma~\ref{T(X): lemma copies of alphas's}, we have that for all distinct $\alpha_1,\alpha_2\in S$ such that $\alpha_1|_{X\setminus\set{i}}=\alpha_2|_{X\setminus\set{i}}$ we have $i\alpha_1=i$ or $i\alpha_2=i$. Then, since $\gamma_1\neq\gamma_2$ and $\gamma_1|_{X\setminus\set{i}}=\gamma_2|_{X\setminus\set{i}}$, we can conclude that $i\gamma_1=i$ or $i\gamma_2=i$. Assume, without loss of generality, that $i\gamma_1=i$. This implies that $\im\gamma_1=\im\gamma\cup\set{i\gamma_1}=\set{x,i}$. It follows from Lemma~\ref{T(X): ab=ba => xb€Im a for all x€Im a} that for all $\beta\in S$ we have $i\beta\in\im\gamma_1=\set{x,i}$. In addition, the fact that $S'=\commidemp{X\setminus\set{i}}{x}$ implies that $y\beta=y\beta|_{X\setminus\set{i}}\in\set{x,y}$ for all $y\in X\setminus\set{x,i}$ and $\beta\in S$, and that $x\beta=x\beta|_{X\setminus\set{i}}=x$ for all $\beta\in S$. Therefore $S\subseteq \commidemp{X}{x}$ and, consequently, $S=\commidemp{X}{x}$.
 	
 	\medskip

 	\textbf{Part 2.} Suppose that $\abs{X}\geqslant 7$. It follows from Theorem~\ref{null semigroups of maximum size} that there exists a null subemigroup $N$ of $\tr{X}$ such that $\abs{N}=\maxnull{\abs{X}}$ and the zero of $N$ has rank $1$. We have that $N$ is a commutative semigroup and $\id{X}\notin N$ (because null semigroups have a unique idempotent). Thus $N\cup\set{\id{X}}$ is a commutative subsemigroup of $\tr{X}$ of size $\maxnull{\abs{X}}+1$ and, consequently, the maximum size of a commutative subsemigroup of $\tr{X}$ is at least $\maxnull{\abs{X}}+1$.
 \end{proof}
 
 We note that, if $n\geqslant 7$, then
 \begin{equation}\label{T(X): 2^(n-1)<xi(n)+1}
 	2^{n-1}=2^6\cdot 2^{n-7}\leqslant 2^6\cdot 3^{n-7}<3^4\cdot 3^{n-7}=3^{n-3}\leqslant \maxnull{n}<\maxnull{n}+1.
 \end{equation}
 This implies that, when $\abs{X}\geqslant 7$, the maximum size of a commutative subsemigroup of $\tr{X}$ is no longer given by $2^{\abs{X}-1}$ and the semigroups of idempotents $\commidemp{X}{x}$ (where $x\in X$) are no longer the maximum-order commutative subsemigroup of $\tr{X}$.

 The last result of this section concerns the largest commutative subsemigroups of $\ptr{X}$.
 
 

 \begin{corollary}\label{P(X): largest comm smg}
 	\begin{enumerate}
 		\item Suppose that $\abs{X}\leqslant 5$. Then the maximum size of a commutative subsemigroup of $\ptr{X}$ is $2^{\abs{X}}$. Moreover, the unique maximum-order commutative subsemigroup of $\ptr{X}$ is $\idemp{\psym{X}}$.
 		
 		\item Suppose that $\abs{X}\geqslant 6$. Then the maximum size of a commutative subsemigroup of $\ptr{X}$ is at least $\maxnull{\abs{X}+1}+1$ and at most the maximum size of a commutative subsemigroup of $\tr{\newtr{X}}$.
 	\end{enumerate}
 \end{corollary}
 
 \begin{proof}
 	\textbf{Part 1.} Suppose that $\abs{X}=1$. Then $\ptr{X}=\set{\emptyset,\id{X}}=\idemp{\psym{X}}$, which is a commutative semigroup. Then the maximum size of a commutative subsemigroup of $\ptr{X}$ is $2$ and the unique semigroup that achieves that size is $\idemp{\psym{X}}$.
 	
 	Now suppose that $2\leqslant\abs{X}\leqslant 5$. Let $S$ be a largest commutative subsemigroup of $\ptr{X}$. Due to the fact that $\idemp{\psym{X}}$ is a commutative subsemigroup of $\ptr{X}$ of size $2^{\abs{X}}$, we have that $\abs{S}\geqslant 2^{\abs{X}}$. Moreover, Proposition~\ref{P(X): proposition P(X) --> T(Y)} implies that $\newtr{S}$ is a subsemigroup of $\tr{\newtr{X}}$ isomorphic to $S$. Hence $\newtr{S}$ is also commutative. In addition, we have that $3\leqslant\abs{\newtr{X}}\leqslant 6$ (because $\abs{X}\leqslant 5$). Consequently, by Theorem~\ref{T(X): maximum size comm smg}, we have that
 	\begin{displaymath}
 		2^{\abs{\newtr{X}}-1}=2^{\abs{X}}\leqslant\abs{S}=\abs{\newtr{S}}\leqslant 2^{\abs{\newtr{X}}-1},
 	\end{displaymath}
 	which implies that $\abs{\newtr{S}}=2^{\abs{\newtr{X}}-1}$. Therefore, by Theorem~\ref{T(X): maximum size comm smg}, we have that $\newtr{S}$ is a commutative semigroup of idempotents. Since $S\simeq\newtr{S}$, then $S$ is also a commutative semigroup of idempotents and it has size $2^{\abs{X}}$. Thus Corollary~\ref{P(X): largest comm smg idemp} implies that $S=\idemp{\psym{X}}$.
 	
 	\medskip
 	
 	\textbf{Part 2.} Suppose that $\abs{X}\geqslant 6$. Let $m$ and $n$ be the maximum sizes of commutative subsemigroups of $\ptr{X}$ and $\tr{\newtr{X}}$, respectively.
 	
 	First we will show that $m\geqslant \maxnull{\abs{X}+1}+1$. Let $B\subseteq X$ be such that $\abs{B}=\alphanull{\abs{X}+1}-1$. By Proposition~\ref{P(X): Omega^X_B} we have that $\nullptr{X}{B}$ is a null subsemigroup of $\ptr{X}$, which implies that $\nullptr{X}{B}$ is a commutative subsemigroup of $\ptr{X}$. Therefore $\nullptr{X}{B}\cup\set{\id{X}}$ is also a commutative subsemigroup of $\ptr{X}$ and, consequently, we must have $m\geqslant \abs{\nullptr{X}{B}\cup\set{\id{X}}}$. Proposition~\ref{P(X): Omega^X_B} establishes that $\abs{\nullptr{X}{B}}=\maxnull{\abs{X}+1}$. Additionally, we can easily check that the zero of $\nullptr{X}{B}$ (which is the unique idempotent of $\nullptr{X}{B}$) is $\emptyset$, which implies that $\id{X}\notin\nullptr{X}{B}$. Thus $m\geqslant\maxnull{\abs{X}+1}+1$.
 	
 	Now we will see that $m\leqslant n$. Let $S$ be a commutative subsemigroup of $\ptr{X}$ such that $\abs{S}=m$. By Proposition~\ref{P(X): proposition P(X) --> T(Y)} we have that $\newtr{S}$ is a subsemigroup of $\tr{\newtr{X}}$ such that $\newtr{S}\simeq S$. Furthermore, $\newtr{S}$ is commutative (because $S$ is commutative). Consequently, $m=\abs{S}=\abs{\newtr{S}}\leqslant n$.
 \end{proof}
 
 It follows from \eqref{T(X): 2^(n-1)<xi(n)+1} that $2^{n}<\maxnull{n+1}+1$ for all $n\geqslant 6$. This implies that, when $\abs{X}\geqslant 6$, $2^{\abs{X}}$ is not the maximum size of a commutative subsemigroup of $\ptr{X}$ and, consequently, $\idemp{\psym{X}}$ is not the maximum-order commutative subsemigroup of $\ptr{X}$.

 \section{Commuting graphs of (full and partial) transformation semigroups}\label{sec: properties comm graph T(X)}
 
 Recall that $X$ denotes a finite set. This section is dedicated to the study of some properties of the commuting graphs of $\tr{X}$ and $\ptr{X}$. We recall that Araújo, Kinyon and Konieczny \cite{Commuting_graph_T_X} already investigated some properties concerning $\commgraph{\tr{X}}$: they determined the diameter of $\commgraph{\tr{X}}$, as well as the diameter of $\commgraph{I}$, where $I$ is an ideal of $\tr{X}$. In this section we determine other properties of $\commgraph{\tr{X}}$ and we initiate the study of $\commgraph{\ptr{X}}$. We will show that the clique number of $\commgraph{\tr{X}}$ is equal to $2^{\abs{X}-1}-1$, when $2\leqslant \abs{X}\leqslant 6$, and at least $\maxnull{\abs{X}}$, when $\abs{X}\geqslant 7$. We will also investigate the clique number of $\commgraph{\ptr{X}}$: we will show that, when $2\leqslant \abs{X}\leqslant 5$, the clique number of $\commgraph{\ptr{X}}$ is $2^{\abs{X}}-2$; and, when $\abs{X}\geqslant 6$, the clique number of $\commgraph{\ptr{X}}$ is at least $\maxnull{\abs{X}+1}-1$ and at most $\cliquenumber{\commgraph{\tr{Y}}}-1$, where $Y$ is a set such that $\abs{Y}=\abs{X}+1$. This upper bound shows that, finding the clique number of $\commgraph{\tr{X}}$ when $\abs{X}\geqslant 7$, plays an important role in the determination of the clique number of $\commgraph{\ptr{X}}$ when $\abs{X}\geqslant 6$. Additionally, we will see that $\commgraph{\tr{X}}$ and $\commgraph{\ptr{X}}$ contain cycles if and only if $\abs{X}\geqslant 3$, in which case $\girth{\commgraph{\tr{X}}}=\girth{\commgraph{\ptr{X}}}=3$; and that $\commgraph{\tr{X}}$ and $\commgraph{\ptr{X}}$ contain left paths if and only if $\abs{X}\geqslant 3$, in which case $\knitdegree{\tr{X}}=\knitdegree{\ptr{X}}=1$.

 We can easily verify that $\centre{\tr{X}}=\set{\id{X}}$ and $\centre{\ptr{X}}=\set{\emptyset,\id{X}}$. This implies that $\tr{X}$ and $\ptr{X}$ are non-commutative if and only if $\abs{X}\geqslant 2$. Thus both $\commgraph{\tr{X}}$ and $\commgraph{\ptr{X}}$ are only defined when $\abs{X}\geqslant 2$. Moreover, this also implies that $\id{X}$ is the unique transformation that is not a vertex of $\commgraph{\tr{X}}$, and that $\emptyset$ and $\id{X}$ are precisely the partial transformations that are not vertices of $\commgraph{\ptr{X}}$.

 \begin{corollary}\label{T(X): clique number}
 	Suppose that $\abs{X}\geqslant 2$. Then
 	\begin{enumerate}
 		\item If $2\leqslant \abs{X}\leqslant 6$, then $\cliquenumber{\commgraph{\tr{X}}}=2^{\abs{X}-1}-1$.
 		
 		\item If $\abs{X}\geqslant 7$, then $\cliquenumber{\commgraph{\tr{X}}}\geqslant \parens{\abs{X}}\xi$.
 	\end{enumerate}
 \end{corollary}
 
 \begin{proof}
 	Let $m$ be the maximum size of a commutative subsemigroup of $\tr{X}$. It follows from Lemma~\ref{preli: largest cliques, commutative subsemigroups} that $\cliquenumber{\commgraph{\tr{X}}}=m-\abs{\centre{\tr{X}}}$. Since the unique element of $\centre{\tr{X}}$ is $\id{X}$, then we have $\cliquenumber{\commgraph{\tr{X}}}=m-1$. Furthermore, Theorem~\ref{T(X): maximum size comm smg} states that, if $2\leqslant \abs{X}\leqslant 6$, then $m=2^{\abs{X}-1}$ and, if $\abs{X}\geqslant 7$, then $m\geqslant\maxnull{\abs{X}}+1$. Thus, if $2\leqslant \abs{X}\leqslant 6$, then $\cliquenumber{\commgraph{\tr{X}}}=2^{\abs{X}-1}-1$ and, if $\abs{X}\geqslant 7$, then $\cliquenumber{\commgraph{\tr{X}}}\geqslant\maxnull{\abs{X}}$.
 \end{proof}
 
 \begin{corollary}\label{P(X): clique number}
 	Suppose that $\abs{X}\geqslant 2$. Then
 	\begin{enumerate}
 		\item If $2\leqslant \abs{X}\leqslant 5$, then $\cliquenumber{\commgraph{\ptr{X}}}=2^{\abs{X}}-2$.
 		
 		\item If $\abs{X}\geqslant 6$, then $\maxnull{\abs{X}+1}-1\leqslant\cliquenumber{\commgraph{\ptr{X}}}\leqslant \cliquenumber{\commgraph{\tr{\newtr{X}}}}-1$.
 	\end{enumerate}
 \end{corollary}
 
 \begin{proof}
 	Let $m$ be the maximum size of a commutative subsemigroup of $\ptr{X}$. It follows from Lemma~\ref{preli: largest cliques, commutative subsemigroups} that $\cliquenumber{\commgraph{\ptr{X}}}=m-\abs{\centre{\ptr{X}}}$. Furthermore, we have that $\centre{\ptr{X}}=\set{\emptyset,\id{X}}$, which implies that $\cliquenumber{\commgraph{\ptr{X}}}=m-2$.
 	
 	\smallskip
 	
 	\textit{Case 1:} Assume that $\abs{X}\leqslant 5$. In Corollary~\ref{P(X): largest comm smg} we established that $m=2^{\abs{X}}$. Thus we have $\cliquenumber{\commgraph{\ptr{X}}}=2^{\abs{X}}-2$.
 	
 	\smallskip
 	
 	\textit{Case 2:} Assume that $\abs{X}\geqslant 6$. By Corollary~\ref{P(X): largest comm smg} we have that $\maxnull{\abs{X}+1}+1\leqslant m\leqslant n$, where $n$ is the maximum size of a commutative subsemigroup of $\tr{\newtr{X}}$. Additionally, Lemma~\ref{preli: largest cliques, commutative subsemigroups} ensures that $\cliquenumber{\commgraph{\tr{\newtr{X}}}}=n-\abs{\centre{\tr{\newtr{X}}}}=n-\abs{\set{\id{\newtr{X}}}}=n-1$. Hence $\maxnull{\abs{X}+1}+1\leqslant m\leqslant \cliquenumber{\commgraph{\tr{\newtr{X}}}}+1$ and, consequently, $\maxnull{\abs{X}+1}-1\leqslant \cliquenumber{\commgraph{\ptr{X}}}\leqslant \cliquenumber{\commgraph{\tr{\newtr{X}}}}-1$.
 \end{proof}
 
 \begin{corollary}\label{T(X): girth}
 	Suppose that $\abs{X}\geqslant 2$. We have that
 	\begin{enumerate}
 		\item $\commgraph{\tr{X}}$ contains cycles if and only if $\abs{X}\geqslant 3$, in which case $\girth{\commgraph{\tr{X}}}=3$.
 		
 		\item $\commgraph{\ptr{X}}$ contains cycles if and only if $\abs{X}\geqslant 3$, in which case $\girth{\commgraph{\ptr{X}}}=3$.
 	\end{enumerate} 
 \end{corollary}
 
 \begin{proof}
 	Let $n=\abs{X}$ and assume that $X=\set{x_1,\ldots,x_n}$.
 	
 	\smallskip
 	
 	\textit{Case 1:} Suppose that $n=\abs{X}=2$. In Figure~\ref{P(X), Figure: commuting graph of P_2 and T_2} we have the commuting graph of $\ptr{X}$ and, distinguished in blue, we have the commuting graph of $\tr{X}$ (which is a subgraph of the commuting graph of $\ptr{X}$). By observation, we can easily verify that both graphs have no cycles.
 	
 	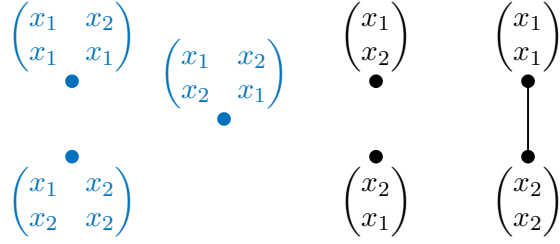
\begin{figure}[hbt]
 		\begin{center}
 			\begin{tikzpicture}
 				
 				\node[vertex, RoyalBlue] (11) at (0,0) {};
 				\node[vertex, RoyalBlue] (22) at (0,-1) {};
 				\node[vertex, RoyalBlue] (21) at (2,-0.5) {};
 				\node[vertex] (1) at (6,0) {};
 				\node[vertex] (2) at (6,-1) {};
 				\node[vertex] (a) at (4,0) {};
 				\node[vertex] (b) at (4,-1) {};
 				
 				\node[anchor=south, RoyalBlue] at (11) {$\begin{pmatrix}x_1&x_2\\ x_1&x_1\end{pmatrix}$};
 				\node[anchor=north, RoyalBlue] at (22) {$\begin{pmatrix}x_1&x_2 \\ x_2&x_2\end{pmatrix}$};
 				\node[anchor=south, RoyalBlue] at (21) {$\begin{pmatrix}x_1&x_2 \\ x_2&x_1\end{pmatrix}$};
 				\node[anchor=south] at (1) {$\begin{pmatrix}x_1 \\ x_1\end{pmatrix}$};
 				\node[anchor=north] at (2) {$\begin{pmatrix}x_2 \\ x_2\end{pmatrix}$};
 				\node[anchor=south] at (a) {$\begin{pmatrix}x_1 \\ x_2\end{pmatrix}$};
 				\node[anchor=north] at (b) {$\begin{pmatrix}x_2 \\ x_1\end{pmatrix}$};
 				
 				\draw[edge] (1) -- (2);
 				
 			\end{tikzpicture}
 		\end{center}
 		\caption{Commuting graph of $\ptr{\set{x_1,x_2}}$ and commuting graph of $\tr{\set{x_1,x_2}}$ (in blue).}
 		\label{P(X), Figure: commuting graph of P_2 and T_2}
 	\end{figure}
 	
%
%
%
 	
 	\smallskip
 	
 	\textit{Case 2:} Suppose that $n=\abs{X}\geqslant 3$. By Corollary~\ref{T(X): clique number} we have that
 	\begin{displaymath}
 		\cliquenumber{\commgraph{\tr{X}}}\geqslant \braces*{
 			\begin{array}{@{}l@{}l@{}ll@{}}
 				2^{\abs{X}-1}-1 &{}\geqslant 2^{3-1}-1 &{}=3 & \text{(if $3\leqslant \abs{X}\leqslant 6$)} \\[1mm]
 				\maxnull{\abs{X}} &{}\geqslant\maxnull{7} &{}=81 & \text{(if $\abs{X}\geqslant 7$)} \\
 			\end{array}
 		}\geqslant 3,
 	\end{displaymath}
 	which implies that $\commgraph{\tr{X}}$ contains three vertices that are adjacent to each other; that is, $\commgraph{\tr{X}}$ contains a cycle of length $3$. Since $\commgraph{\tr{X}}$ is a subgraph of $\commgraph{\ptr{X}}$, then $\commgraph{\ptr{X}}$ also contains a cycle of length $3$. Thus $\girth{\commgraph{\tr{X}}}=\girth{\commgraph{\ptr{X}}}=3$.
 \end{proof}

 \begin{proposition}\label{T(X): knit degree}
 	Suppose that $\abs{X}\geqslant 2$. We have that
 	\begin{enumerate}
 		\item $\commgraph{\tr{X}}$ contains left paths if and only if $\abs{X}\geqslant 3$, in which case $\knitdegree{\tr{X}}=1$.
 		
 		\item $\commgraph{\ptr{X}}$ contains left paths if and only if $\abs{X}\geqslant 3$, in which case $\knitdegree{\ptr{X}}=1$.
 	\end{enumerate} 
 \end{proposition}
 
 \begin{proof}
 	Let $n=\abs{X}$ and assume that $X=\set{x_1,\ldots,x_n}$.
 	
 	\smallskip
 	
 	\textit{Case 1:} Suppose that $n=\abs{X}=2$. We can easily verify in Figure~\ref{P(X), Figure: commuting graph of P_2 and T_2} that $\commgraph{\tr{X}}$ is a null graph. Therefore all paths of $\commgraph{\tr{X}}$ have length $0$, which implies that $\commgraph{\tr{X}}$ contains no left paths. Moreover, by observation of Figure~\ref{P(X), Figure: commuting graph of P_2 and T_2} we immediately conclude that the unique non-trivial path in $\commgraph{\ptr{X}}$ is
 	\begin{displaymath}
 		\begin{pmatrix}
 			x_1\\x_1
 		\end{pmatrix}-\begin{pmatrix}
 			x_2\\x_2
 		\end{pmatrix}.
 	\end{displaymath}
 	However
 	\begin{displaymath}
 		\begin{pmatrix}
 			x_1\\x_1
 		\end{pmatrix}\begin{pmatrix}
 			x_1\\x_1
 		\end{pmatrix}
 		=\begin{pmatrix}
 			x_1\\x_1
 		\end{pmatrix}\neq\emptyset
 		=\begin{pmatrix}
 			x_2\\x_2
 		\end{pmatrix}\begin{pmatrix}
 			x_1\\x_1
 		\end{pmatrix},
 	\end{displaymath}
 	which implies that the path in question is not a left path in $\commgraph{\ptr{X}}$. Since $\commgraph{\ptr{X}}$ has no other non-trivial paths, we can conclude that $\commgraph{\ptr{X}}$ contains no left paths.
 	
 	\smallskip
 	
 	\textit{Case 2:} Suppose that $n=\abs{X}\geqslant 3$. We consider the following transformations:
 	\begin{displaymath}
 		\alpha_1=\begin{pmatrix}
 			x_1 & \cdots & x_n\\
 			x_1 & \cdots & x_1
 		\end{pmatrix}\quad \text{and} \quad
 		\alpha_2=\begin{pmatrix}
 			x_1 & x_2 & \cdots & x_{n-1} & x_n\\
 			x_1 & x_1 & \cdots & x_1 & x_2
 		\end{pmatrix}.
 	\end{displaymath}
 	We have that $\alpha_1\alpha_2=\alpha_1=\alpha_2\alpha_1$, which implies that $\alpha_1-\alpha_2$ is a path in $\commgraph{\tr{X}}$ and in $\commgraph{\ptr{X}}$. Additionally, we have that $\alpha_1\alpha_1=\alpha_1=\alpha_2\alpha_1$ and $\alpha_1\alpha_2=\alpha_1=\alpha_2\alpha_2$. Thus $\alpha_1-\alpha_2$ is a left path in $\commgraph{\tr{X}}$ and in $\commgraph{\ptr{X}}$. Thus $\knitdegree{\tr{X}}=\knitdegree{\ptr{X}}=1$.
 \end{proof}

 \section{Open problems} \label{Open problems}
 
 In this section we discuss four open problems that we approached in the previous sections. Below we list those problems.
 
 \begin{problem}\label{problem T(X)}
 	Suppose that $\abs{X}\geqslant 7$. Determine the maximum size of a commutative subsemigroup of $\tr{X}$ and characterize the maximum-order commutative subsemigroups of $\tr{X}$.
 \end{problem}
 
 \begin{problem}\label{problem P(X)}
 	Suppose that $\abs{X}\geqslant 6$. Determine the maximum size of a commutative subsemigroup of $\ptr{X}$ and characterize the maximum-order commutative subsemigroups of $\ptr{X}$.
 \end{problem}
 
 \begin{problem}\label{problem graph T(X)}
 	Suppose that $\abs{X}\geqslant 7$. Determine the clique number of $\commgraph{\tr{X}}$.
 \end{problem}
 
 \begin{problem}\label{problem graph P(X)}
 	Suppose that $\abs{X}\geqslant 6$. Determine the clique number of $\commgraph{\tr{X}}$.
 \end{problem}
 
 It follows from Lemma~\ref{preli: largest cliques, commutative subsemigroups} that obtaining the maximum size of a commutative subsemigroup of $\tr{X}$ and obtaining the clique number of $\commgraph{\tr{X}}$ are equivalent problems. Hence finding answers for Problems~\ref{problem T(X)} and \ref{problem P(X)} leads to answers for Problem~\ref{problem graph T(X)} and \ref{problem graph P(X)}, respectively. Moreover, it follows from Theorem~\ref{P(X): largest comm smg} that, when $\abs{X}\geqslant 6$, the maximum size of a commutative subsemigroup of $\tr{Y}$, where $Y$ is a set such that $\abs{Y}=\abs{X}+1$, is an upper bound for the maximum size of a commutative subsemigroup of $\ptr{X}$. So solving Problem~\ref{problem T(X)} also has implications for finding the solution of Problem~\ref{problem P(X)}. Therefore solving Problems~\ref{problem T(X)}--\ref{problem graph P(X)} can be reduced to solving just Problem~\ref{problem T(X)}. Our conjecture for Problem~\ref{problem T(X)} is the following:
 
 \begin{conjecture}\label{T(X): conjecture}
 	Suppose that $\abs{X}\geqslant 7$. Then the maximum size of a commutative subsemigroup of $\tr{X}$ is $(|X|)\xi+1$. Moreover, the maximum-order commutative subsemigroup of $\tr{X}$ are precisely the semigroups $\nulltr{X}{x_1}{x_t}\cup\set{\id{X}}$, where $t=\alphanull{\abs{X}}$ and $x_1,\ldots,x_t\in X$ are pairwise distinct.
 \end{conjecture}
 
 If this conjecture is true, then we can easily prove that the solutions for Problems~\ref{problem P(X)}, \ref{problem graph T(X)} and \ref{problem graph P(X)} are:
 \begin{enumerate}
 	\item If $\abs{X}\geqslant 6$, then the maximum size of a commutative subsemigroup of $\ptr{X}$ is $\maxnull{\abs{X}+1}+1$ and the maximum-order commutative subsemigroups of $\ptr{X}$ are precisely the null semigroups $\nullptr{X}{B}$, where $B\subseteq X$ is such that $\abs{B}=\alphanull{\abs{X}+1}-1$.
 	
 	\item If $\abs{X}\geqslant 7$, then $\cliquenumber{\commgraph{\tr{X}}}=\maxnull{\abs{X}}$.
 	
 	\item If $\abs{X}\geqslant 6$, then $\cliquenumber{\commgraph{\ptr{X}}}=\maxnull{\abs{X}+1}-1$.
 \end{enumerate}

 Several results support Conjecture~\ref{T(X): conjecture}:
 
 \begin{enumerate}
 	\item In Theorem~\ref{T(X): maximum size comm smg} we proved that, when $\abs{X}\geqslant 7$, the maximum size of a commutative subsemigroup of $\tr{X}$ is at least $\maxnull{\abs{X}}+1$. Moreover, we know that the semigroups $\nulltr{X}{x_1}{x_t}\cup\set{\id{X}}$, where $t=\alphanull{\abs{X}}$ and $x_1,\ldots,x_t\in X$ are pairwise distinct, have size $\maxnull{\abs{X}}+1$ (see Theorem~\ref{null semigroups of maximum size}).
 	
 	\item It follows from Theorem~\ref{T(X): largest comm smg 1 idemp} that, when $\abs{X}\geqslant 7$, the commutative subsemigroups of $\tr{X}$ with a unique idempotent have size smaller than $\maxnull{\abs{X}}+1$. Moreover, when $\abs{X}\geqslant 7$, the largest commutative subsemigroups of $\tr{X}$ with a unique idempotent are precisely the null semigroups $\nulltr{X}{x_1}{x_t}$, where $t=\alphanull{\abs{X}}$ and $x_1,\ldots,x_t\in X$ are pairwise distinct.
 	
 	\item It follows from Theorem~\ref{T(X): maximum size comm smg idemp} that, when $\abs{X}\geqslant 7$, the commutative subsemigroups of $\tr{X}$ formed exclusively by idempotents have size at most $2^{\abs{X}-1}$, which is smaller than $\maxnull{\abs{X}}+1$.
 	
 	\item Computational experimental evidence suggests that, when $|X|=7$, the maximum size of a commutative subsemigroup of $\tr{X}$ is $\maxnull{\abs{X}}+1=\maxnull{7}+1=82$.
 \end{enumerate}
 
 As a consequence of 2 and 3 of the list above we have that, in order to prove Conjecture~\ref{T(X): conjecture}, it is enough to analyse the size of the commutative subsemigroups of $\tr{X}$ that contain at least two idempotents and a non-idempotent transformation.

    \bibliography{Bibliography} 

\newcommand{\etalchar}[1]{$^{#1}$}
\begin{thebibliography}{ACMM25}

\bibitem[ABK15]{Commuting_graph_I_X}
Jo{\~a}o Ara{\'u}jo, Wolfram Bentz, and Janusz Konieczny.
\newblock The commuting graph of the symmetric inverse semigroup.
\newblock {\em Israel Journal of Mathematics}, 207:103--149, 2015.
\newblock \href {https://doi.org/10.1007/s11856-015-1173-9}
  {\path{doi:10.1007/s11856-015-1173-9}}.

\bibitem[ACMM25]{Graphs_arise_as_commuting_graphs_groups}
V.~Arvind, Peter~J. Cameron, Xuanlong Ma, and Natalia~V. Maslova.
\newblock Aspects of the commuting graph.
\newblock {\em Journal of Algebra}, 2025.
\newblock \href {https://doi.org/10.1016/j.jalgebra.2025.07.020}
  {\path{doi:10.1016/j.jalgebra.2025.07.020}}.

\bibitem[AFM07]{2-generated_submonoids_I_X}
J.~M. André, V.~H. Fernandes, and J.~D. Mitchell.
\newblock Largest 2-generated subsemigroups of the symmetric inverse semigroup.
\newblock {\em Proceedings of the Edinburgh Mathematical Society},
  50(3):551–561, 2007.
\newblock \href {https://doi.org/10.1017/S0013091505001598}
  {\path{doi:10.1017/S0013091505001598}}.

\bibitem[AK03]{Characterization_idempotents_1}
João Araújo and Janusz Konieczny.
\newblock Automorphism groups of centralizers of idempotents.
\newblock {\em Journal of Algebra}, 269(1):227--239, 2003.
\newblock \href {https://doi.org/10.1016/S0021-8693(03)00499-X}
  {\path{doi:10.1016/S0021-8693(03)00499-X}}.

\bibitem[AKK11]{Commuting_graph_T_X}
João Araújo, Michael Kinyon, and Janusz Konieczny.
\newblock Minimal paths in the commuting graphs of semigroups.
\newblock {\em European Journal of Combinatorics}, 32(2):178--197, 2011.
\newblock \href {https://doi.org/10.1016/j.ejc.2010.09.004}
  {\path{doi:10.1016/j.ejc.2010.09.004}}.

\bibitem[BG89]{Symmetric_group}
J.~M. Burns and B.~Goldsmith.
\newblock Maximal order abelian subgroups of symmetric groups.
\newblock {\em Bull. Lond. Math. Soc.}, 21(1):70--72, 1989.
\newblock \href {https://doi.org/10.1112/blms/21.1.70}
  {\path{doi:10.1112/blms/21.1.70}}.

\bibitem[BRR76]{Non_commutative_nilpotent_semigroups}
R.~G. Biggs, S.~A. Rankin, and C.~M. Reis.
\newblock A study of graph closed subsemigroups of a full transformation
  semigroup.
\newblock {\em Trans. Amer. Math. Soc.}, 219:211--223, 1976.
\newblock \href {https://doi.org/10.2307/1997590} {\path{doi:10.2307/1997590}}.

\bibitem[Cai12]{Nine_chapters_Cain}
Alan~J. Cain.
\newblock {\em Nine Chapters on the Semigroup Art: Lecture notes for a tour
  through semigroups}.
\newblock Porto {\&} Lisbon, 2012.
\newblock URL: \url{https://archive.org/details/cain_semigroups_a4_screen}.

\bibitem[Cam22]{Cameron_commuting_graphs_notes}
Peter~J. Cameron.
\newblock Graphs defined on groups.
\newblock {\em International Journal of Group Theory}, 11(2):53--107, 2022.
\newblock \href {https://doi.org/10.22108/ijgt.2021.127679.1681}
  {\path{doi:10.22108/ijgt.2021.127679.1681}}.

\bibitem[CEF{\etalchar{+}}23]{Null_semigroups}
Peter~J. Cameron, James East, Des FitzGerald, James~D. Mitchell, Luke Pebody,
  and Thomas Quinn-Gregson.
\newblock Minimum degrees of finite rectangular bands, null semigroups, and
  variants of full transformation semigroups.
\newblock {\em Combinatorial Theory}, 3(3), 2023.
\newblock \href {https://doi.org/10.5070/C63362799}
  {\path{doi:10.5070/C63362799}}.

\bibitem[CMP24]{Commutative_nilpotent_transformation_semigroups_paper}
Alan~J. Cain, António Malheiro, and Tânia Paulista.
\newblock Commutative nilpotent transformation semigroups.
\newblock {\em Semigroup Forum}, 109:60--75, 2024.
\newblock \href {https://doi.org/10.1007/s00233-024-10444-8}
  {\path{doi:10.1007/s00233-024-10444-8}}.

\bibitem[GM08]{Largest_subsemigroups_transformation}
R.~Gray and J.~D. Mitchell.
\newblock Largest subsemigroups of the full transformation monoid.
\newblock {\em Discrete Math.}, 308(20):4801--4810, 2008.
\newblock \href {https://doi.org/10.1016/j.disc.2007.08.075}
  {\path{doi:10.1016/j.disc.2007.08.075}}.

\bibitem[IJ08]{Commuting_graph_symmetric_alternating_groups}
A.~Iranmanesh and A.~Jafarzadeh.
\newblock On the commuting graph associated with the symmetric and alternating
  groups.
\newblock {\em Journal of Algebra and Its Applications}, 7(1):129--146, 2008.
\newblock \href {https://doi.org/10.1142/S0219498808002710}
  {\path{doi:10.1142/S0219498808002710}}.

\bibitem[MC24]{Commuting_graphs_groups_split}
Xuanlong Ma and Peter~J. Cameron.
\newblock Finite groups whose commuting graph is split.
\newblock {\em Trudy Instituta Matematiki i Mekhaniki UrO RAN}, 30(1):280--283,
  2024.
\newblock \href {https://doi.org/10.21538/0134-4889-2024-30-1-280-283}
  {\path{doi:10.21538/0134-4889-2024-30-1-280-283}}.

\bibitem[Pau25]{My_thesis}
Tânia Paulista.
\newblock {\em Commuting graphs of semigroups}.
\newblock PhD thesis, NOVA School of Science \& Technology, 2025.

\bibitem[Sch78]{Schein_conjecture}
B.~M. Schein.
\newblock On semisimple bands.
\newblock {\em Semigroup Forum}, 16:1--12, 1978.
\newblock \href {https://doi.org/10.1007/BF02194610}
  {\path{doi:10.1007/BF02194610}}.

\bibitem[Vdo99]{Alternating_group}
E.~P. Vdovin.
\newblock Maximal orders of abelian subgroups in finite simple groups.
\newblock {\em Algebra and Logic}, 38(2):67--83, 1999.
\newblock \href {https://doi.org/10.1007/BF02671721}
  {\path{doi:10.1007/BF02671721}}.

\bibitem[Wil96]{Graphs_Wilson}
Robin~J. Wilson.
\newblock {\em Introduction to Graph Theory}.
\newblock Harlow: Longman, 4 edition, 1996.

\end{thebibliography}
\bibliographystyle{alphaurl}

\end{document}